\documentclass[11pt]{amsart}

\usepackage{amsmath,amsfonts,amssymb,amsthm,enumerate,tikz-cd,mathrsfs}

\usepackage{hyperref} 
\hypersetup{
	colorlinks=true, 
	     urlcolor= blue,  
	     linkcolor= blue, 
     citecolor=blue,	
	bookmarksopen=true
	}
\include{macros}

\def\Q{\mathbb Q}
\def\F{\mathbb F}

\def\Aut{\mathrm{Aut}}

\theoremstyle{plain}
\newtheorem{theorem}{Theorem}
\newtheorem{conjecture}{Conjecture}
\newtheorem{lemma}{Lemma}
\newtheorem{proposition}{Proposition}
\newtheorem{corollary}{Corollary}

\theoremstyle{definition}

\newtheorem{definition}{Definition}
\newtheorem{example}{Example}
\newtheorem{remark}{Remark}

\DeclareMathOperator{\ab}{ab}
\DeclareMathOperator{\un}{un}
\DeclareMathOperator{\rig}{rig}

\newcommand{\pr}{\mathrm{pr}}
\DeclareMathOperator{\Gal}{Gal}

\newcommand{\li}{\mathrm{li}}

\newcommand{\Sto}{\mathrm{Sto}}
\newcommand{\val}{\mathrm{val}}
\newcommand{\Rep}{\mathrm{Rep}}
\newcommand{\FZ}{\mathrm{FZ}}

\newcommand{\spe}{\mathrm{sp}}
\newcommand{\Sel}{\mathrm{Sel}}
\newcommand{\Res}{\mathrm{Res}}
\newcommand{\Alb}{\mathrm{Alb}}
\newcommand{\dR}{\mathrm{dR}}
\newcommand{\Be}{\mathrm{Be}}
\newcommand{\Spf}{\mathrm{Spf}}
\newcommand{\Isom}{\mathrm{Isom}}
\newcommand{\GSp}{\mathrm{GSp}}

\newcommand{\codim}{\mathrm{codim}}
\newcommand{\SL}{\mathrm{SL}}
\newcommand{\Sym}{\mathrm{Sym}}

\newcommand{\Spec}{\mathrm{Spec}}
\newcommand{\an}{\mathrm{an}}
\newcommand{\Mat}{\mathrm{Mat}}

\newcommand{\AJ}{\mathrm{AJ}}

\newcommand{\Jac}{\mathrm{Jac}}
\newcommand{\et}{\mathrm{\acute{e}t}}

\newcommand{\loc}{\mathrm{loc}}

\DeclareMathOperator{\rk}{rk}

\newcommand{\Lie}{\mathrm{Lie}}

\newcommand{\Ker}{\mathrm{Ker}}
\newcommand{\gr}{\mathrm{gr}}
\newcommand{\Hom}{\mathrm{Hom}}
\newcommand{\End}{\mathrm{End}}
\DeclareMathOperator{\GL}{GL}
\newcommand{\Z}{\mathbb{Z}}

\DeclareMathOperator{\ext}{Ext}

\begin{document}
\title{$p$-adic integrals and linearly dependent points on families of curves I}
\author{Netan Dogra}
\maketitle
\pagestyle{headings}
\markright{$p$-ADIC INTEGRALS ON FAMILIES OF CURVES I}
\begin{abstract}
We prove that the set of `low rank' points on sufficiently large fibre powers of families of curves are not Zariski dense. The recent work of Dimitrov--Gao--Habegger and K\"uhne (and Yuan) imply the existence of a bound which is exponential in the rank, and the Zilber--Pink conjecture implies a bound which is linear in the rank. Our main result is a (slightly weaker) linear bound for `low ranks'. We also prove analogous results for isotrivial families (with relaxed conditions on the rank) and for solutions to the $S$-unit equation, where the bounds are now sub-exponential in the rank. Our proof involves a notion of the Chabauty--Coleman(--Kim) method in families (or, in some sense, for simply connected varieties). For Zariski non-density, we use the recent work of Bl\`azquez-Sanz, Casale, Freitag and Nagloo on Ax--Schanuel theorems for foliations on principal bundles.
\end{abstract}

\tableofcontents

\section{Introduction}

In their seminal work on uniformity \cite{CHM}, Caporaso, Harris and Mazur showed that Lang's conjectures on the Zariski closure of rational points of varieties of general type over number fields had a number of striking consequences for the arithmetic of curves. Most notably, it implies a uniform bound $N(g,\Q )$ on the number of rational points of a curve of genus $g$ over the rational numbers (or more generally any number field).

If $\pi :X\to S$ is a family of curves, then the problem of understanding $N$-tuples of points is simply the problem of understanding the rational points on the $N$-fold fibre product
\[
X^N _S :=\underbrace{X\times _S X\times _S \ldots \times _S X}_\text{$N$ times}.
\]
Here, by a family of curves we will mean a proper morphism between smooth geometrically connected quasi-projective varieties whose generic fibre is a smooth curve.
The uniformity result of Caporaso, Harris and Mazur is deduced by showing that for $\pi $ as above, there is an $N$ such that $X^N _S$ dominates a variety of general type. The authors refer to this as \textit{having correlation}. One might say that this amounts to the family $\pi $ having \textit{geometric} correlation. Then (under Lang's conjecture), geometric correlation implies what one might term \textit{arithmetic} correlation, i.e. that the set of rational points on $X^N _S$ is not Zariski dense for $N$ sufficiently large.

Whilst the problem of arithmetic correlation remains intractable in general, recently there has been exciting progress establishing its relation with other deep conjectures in arithmetic geometry. Namely, recent work of Dimitrov, Gao and Habegger \cite{DGH21} and K\"uhne \cite{kuhne21} (along with recent work of Yuan \cite{yuan21}) proves that the number of rational points on a curve of genus $g>1$ can be bounded in terms of its genus and the Mordell--Weil rank of its Jacobian: more precisely, for any $g>1$ and number field $K$ there is a constant $c(g,K)$ such that for any smooth genus $g$ curve $C$ over $K$ whose Jacobian has Mordell--Weil rank $r$, $\# C(K)<c(g,K)^{1+r}$. This can be seen as a higher genus analogue of the bounds proved by Evertse \cite{evertse84} (building on work of several others) on solutions to the $S$-unit equation (see \eqref{eqn:evertse} below), and also a generalisation of results of Silverman \cite{silverman:twist} (see \eqref{eqn:silverman}) on the case of isotrivial families of higher genus curves.

If we view uniformity as a testing ground for the Lang's conjecture \footnote{Perhaps one of the original motivations for the Caporaso--Harris--Mazur theorem, see \href{https://quomodocumque.wordpress.com/2014/07/20/are-ranks-bounded/}{https://quomodocumque.wordpress.com/2014/07/20/are-ranks-bounded/}}, then it is natural to wonder about \textit{effective} arithmetic correlation, i.e. about the optimal $n$ such that $X^n _S$ has non-Zariski dense rational points. Whilst the work of Dimitrov, Gao, Habegger, K\"uhne and Yuan does not address this question directly, it does control a certain subset of $X^n _S (K)$. For $s\in S(K)$, let $J_s (K)$ denote the Jacobian of $X_s$. We define 
\[
\Alb _s :X_s (K)\to J_s (K)
\]
to be the map sending $x$ to $(2g-2)\cdot x -K_{X_s}$. We say an $n$-tuple $(x_1 ,\ldots ,x_n )\in X_s(K)^n$ has rank $r$ if $\Alb _s (x_1 ),\ldots ,\Alb _s (x_n )$ generate a rank $r$ subgroup of $J_s (K)$. We define the subset $X^n _S (K)_{\rk r}$ of rank $r$ points to be the union, over all $s\in S(K)$, of the set of rank $r$ $n$-tuples of points in $X_s (K)$. Then we may ask about effective arithmetic correlation for rank $r$ points: that is, what is the least $n$ such that $X^n _S (K)_{\rk r}$, or $X^n_S (\overline{K})_{\rk r}$, is not Zariski dense?

The results of \cite{DGH21}, \cite{kuhne21} and \cite{yuan21} confirm that such an $n$ exists: precisely, they show that $X^n _S (K)_{\rk r}$ is not Zariski dense in $X^n _S$ for all $n>c(g,K)^{1+r}$. In fact, this does not follow from the their work as we stated it above, since it does not involve the rank of the Jacobian, but it does follow from the theorems they prove (see \cite{gao21}). However, conjecturally this bound is far from optimal: Stoll showed \cite{stoll:JEMS} that the Zilber--Pink conjecture implies that the least such $n$ can be bounded linearly in $g$ and $r$ (see Theorem \ref{cor:ZilbPink} for the exact statement). Since there will be many cases where the optimal such $n$ is smaller than what one obtains from a uniform bound on $\# X_s (K)$, it is natural to ask if there are methods for proving non-Zariski density of $X^n _S (K)$ without arguing via uniformity.

The basic observation of this paper is that the Chabauty--Coleman--Kim method is well-suited to this problem. The reason for this is roughly that this method actually gives a \textit{formula} for low rank $n$-tuples. This is well-known in computational applications, where one often tries to compute $X_s (K)$ by finding enough known rational points to produce a nontrivial analytic relation all other rational points must satisfy. Note that the Chabauty--Coleman--Kim method can also be used to prove bounds on the number of points \cite{coleman}, \cite{betts21}, but the bounds we obtain for non-density are stronger than what one would get via these results (although both these papers are used in the proof of our main results).
\subsection{Low rank results via the Chabauty--Coleman method in families}
In fact, in this paper we will study a slightly more general notion of rank, involving an arbitrary subgroup of the Jacobian of $X$ over the function field of $S$. Let $J$ denote the Jacobian of $X$ over $S$. Let $\Gamma $ be a finitely generated subgroup of $J(K(S))$. We write a $K$-point of $X^N _S $ as $((x_i )_{i=1}^n ;s)$ if its image in $S$ is $s$ and its $i$th coordinate is $x_i$. For an $N$-tuple $x=((x_i );s)\in X^N _S (K)$, we say that $x$ has $\Gamma $-rank $r$ if $X_s$ is smooth, and the subgroup of $J_s (K) $ generated by $\Alb _s (x_1 ),\ldots ,\Alb (x_N )$ and $s^* \Gamma $ has rank $r$. We denote the set of $K$-points of $X^n _S (K)$ of $\Gamma $-rank $\leq r$ by $X^n _S (K)_{\Gamma -\rk \leq r}$. In the case when $\Gamma =\{ 1\}$, $\Gamma $-rank coincides with the previous notion of rank.
If $\pi $ is a generically smooth family of projective curves, we define the $\Gamma $-rank $\leq r$ subspace of $X^n _S (K)$, denoted $X^n _S (K)_{\Gamma -\rk \leq r}$ to be $\cup _{s\in S'(K )}X^n _s (K)_{\Gamma -\rk r}$, where $S'(K) \subset S(K)$ is the set of points where $X_s $ is smooth. 
\begin{theorem}\label{thm:main1}
Let $K$ be a finite extension of $\Q _p$ ,and let $S$ be a smooth quasiprojective variety over $\mathcal{O}_K$ of dimension $s$. Let $X$ be a smooth quasiprojective variety over $\mathcal{O}_K$, and
\[
\pi :X\to S
\]
a flat, proper morphism whose generic fibre is a smooth curve of genus $g$. Let $\Gamma $ be a subgroup of $J(K(S))$ of rank $d$.
\begin{enumerate}
\item Suppose $\pi $ is smooth. Then, for any $r\in \{0,\ldots ,g-2\}$, $X^n _S (K )_{\Gamma -\rk \leq r}$ is not Zariski dense in $X^n _S$ for all 
\[
n>\frac{s+(r-d )(g-r)}{g-r-1}.
\]
\item Suppose $\pi $ is stable, and smooth outside a strict normal crossing divisor $Z\subset S$. Then for any $r\in \{0 ,\ldots ,g-3\}$, $X^n _S (K)_{\rk \leq r}$ is not Zariski dense in $X^n _S$ for all 
\[
n> \frac{(r+1)(g-1-r)+s}{g-r-2}\cdot 3g +\frac{(r +1)(g-r)+s}{g-r-1}\cdot 2g .
\]
\end{enumerate}
\end{theorem}

Note that these results are new also in the case when $n=0$. Here they give a generalisation of a result of Lawrence and Zannier \cite{lawrence2020p} which deals with the case where $n=r=0$, and is motivated by $p$-adic analogues of the work of Andr\'e, Corvaja and Zannier \cite{ACZ} on work in the complex setting when $n=r=0$ and $d=1$.

If instead we work over number fields, this result has the following corollary on low rank points on more general families of curves.
\begin{corollary}\label{cor:NFs}
Let 
\[
\pi :X\to S
\]
be a generically smooth family of curves over a number field $K$. Then, for any $r\in \{0 ,\ldots ,g-3\}$, $X^n _S (K)_{\rk \leq r}$ is not Zariski dense in $X^n _S$ for all 
\[
n> \frac{(r+1)(g-1-r)+s}{g-r-2}\cdot 3g +\frac{(r +1)(g-r)+s}{g-r-1}\cdot 2g .
\]
\end{corollary}

Our strategy can be seen as being inspired by an observation of Stoll \cite{stoll:JEMS} that the Zilber--Pink conjecture implies a linear (in $r$ and in $\dim S$) bound on the least $n$ such that $X^n _S (K)_{\rk r}$ is not dense. In this paper, we do not prove any new cases of the Zilber--Pink conjecture. Instead, our results may be viewed as special cases of a (strictly) simpler $p$-adic Zilber--Pink conjecture. 

In general, it seems hard to go from the Zariski non-density result of Theorem \ref{thm:main1} to a concrete uniformity result. However, we can do so in the following simple case.

\begin{corollary}\label{cor:low_gonality}
Let $C$ be a smooth projective geometrically irreducible curve of genus $g>1$ over a $p$-adic field $K$, with $\End (\Jac (X))=\mathbb{Z}$, $G<\Jac (X)(K)$ a subgroup of rank $d$, and $r\geq \min \{d,g-1\}$. Suppose the gonality $\gamma $ of $C$ satisfies
\[
\gamma \geq \frac{(r-d)(g-r)}{g-r-1}
\]
Then there are only finitely many rank $r$ subgroups $\Gamma $ of $\Jac (C)(K)/G$ such that 
\[
\# (C(K)\cap \Gamma /G )>\frac{(r-d)(g-r)}{g-r-1}.
\] 
\end{corollary}
Here, we write $C(K)\cap \Gamma /G$ to mean the subset of $C(K)$ whose image in $\Jac (C)/G$ is contained in $\Gamma /G$.

Although we have phrased Theorem \ref{thm:main1} in the context of schemes, it is easy to deduce from it the following statement about rational points on stacks. For a field $K$, we say that a $K$-point $(x_1 ,\ldots ,x_n ;C)$ of the stack $\mathcal{M}_{g,n}$ is rank $\leq r$ if $\AJ (x_1 ),\ldots ,\AJ (x_n )$ generate a rank $\leq r$ subgroup of the Jacobian of the curve $C$.
\begin{corollary}\label{cor:Mg}
Let $K$ be a finite extension of $\Q _p$, $g>1$ and $0\leq r\leq g-3$. Then the set of rank $\leq r$ points in $\mathcal{M}_{g,n}(K )$ is not Zariski dense in $\mathcal{M}_{g,n}$ for all 
\[
n> \frac{(r+1)(g-1-r)+3g-3}{g-r-2}\cdot 3g +\frac{(r +1)(g-r)+3g-3}{g-r-1}\cdot 2g .
\]
\end{corollary}
For example, when $r=g-3$, we obtain that rank $\leq r$ points of $\mathcal{M}_{g,n}(K)$ are not Zariski dense in $\mathcal{M}_{g,n}$ for $n>21g^2-30g$. In the case $K=\Q _p $, this is very close to the Katz--Rabinoff--Zureick-Brown bound \cite{KRZB}, which states that $\# C(\Q )< 84g^2 - 98g + 28$ for $C$ of genus $g$ and rank less than $g-2$.
\subsection{Higher rank results for $S$-unit and twist families via Chabauty--Coleman--Kim in families}
Fix an integer $r>0$, and consider how solutions to the $S$-unit equation
\[
x+y=1,
\]
$x,y\in S$ vary when $S$ ranges over all rank $s$ subgroups of $\Q ^\times $. There is an extensive literature on this and related problems, which we will not attempt to survey here. We only mention the strongest upper and lower bounds. Firstly, Evertse proved \cite{evertse84} that, for any $S$, we have
\begin{equation}\label{eqn:evertse}
\# X(S)<3\cdot 7^{2s+3}.
\end{equation}
On the other hand, by work of Erd\"os, Stewart and Tijdeman \cite{EST88}, we know that for any $\epsilon >0$ there is a $c_{\epsilon }$ such that there exist arbitrarily large $s$ such that there is an $S$ of rank $s$ with
\[
\# X(S)>c_\epsilon e^{(4-\epsilon )\sqrt{s/\log (s)}}
\]
(in fact, they show that $S$ can be taken to be of the form $\mathbb{Z}_T ^\times $, for $T$ a set of $s$ primes). Stewart has conjectured (see \cite{EGST}) that there is a uniform subexponential bound on the number of solutions to the $S$-unit equation (more precisely that $\log (\# X(S))$ is $O(s^{1/3})$). In the spirit of Lang's conjecture (and because we cannot answer Stewart's question) we show the existence of an $n$ such that $\cup _{\rk S=s}X(S)^n $ is not Zariski dense in $\mathbb{A}^n$ with $n$ subexponential in $s$.

An analogous result for twists of a smooth projective curve of genus $g>1$ was proved by Silverman \cite{silverman:twist}, who showed that, for any such curve $C/\Q $, there is a constant $\gamma _C$ such that
\begin{equation}\label{eqn:silverman}
\# C' (\Q )\leq \gamma _C \cdot 7^{\rk \Jac (C')(\Q )}
\end{equation}
for all twists $C'$ of $C$. We also prove a result on non--Zariski density of $n$-fold fibre products when the rank is $r$, with $n$ subexponential in $r$. In this case, our methods are conditional on a certain special case of the Bloch--Kato conjectures (see Conjecture \ref{BK} for a precise statement).
\begin{theorem}\label{thm:main2}
\begin{enumerate}
\item For all $s>5$, there is a proper Zariski closed subset of $\mathbb{A}^n$ whose $\Q $-points contain all $n$-tuples of solutions to the $S$-unit equation for subgroups $S<\Q ^\times $ of rank $s$ whenever
\[
n>59\cdot (2s)^{\frac{\log (2s)+\log \log (2s)}{\log (2)}+5}\cdot \log (2s).
\]
\item 
Let $X\to S $ be an isotrivial family of curves of genus $g>1$ over $\Q $. Assume conjecture \ref{BK}. Then, for all $r>11g$, $X^n _S (\Q )_{\rk r}$ is not Zariski dense in $X^n _S$ for all
\[
n> 3\prod c_v \cdot r^{\frac{\log (r)+\log \log (r)+\log (2)}{\log (g+\sqrt{g^2 -1}) }+4}\cdot \log (r) (g+\sqrt{g^2 -1}) ^3
\]
where $c_v$ is the constant defined in Definition \ref{defn:cv}.
\end{enumerate}
\end{theorem}

\begin{remark}
Stoll has obtained much sharper results for low rank (i.e. rank less than the genus) quadratic twists of hyperelliptic curves (see \cite{stoll:twist}). Rather than fixing a prime (as we do in this paper), his method works by choosing, for each twist, a prime where the corresponding cohomology class ramifies. It would be interesting to develop these methods in a non-abelian context.
\end{remark}

\begin{remark}
In the usual setting of the Chabauty--Coleman--Kim method, one can obtain unconditional proofs of the finiteness of rational points when one restricts to solvable covers of $\mathbb{P}^1 $, or curves with CM Jacobians \cite{ellenberg2017rational}, \cite{CK10}. It is not clear (to the author) that one can similarly prove unconditional results analogous to part (2) of Theorem \ref{thm:main2} if one restricts to twist families of this form. The issue is that in the usual Chabauty--Coleman--Kim method one needs to know that an inequality of dimensions (between certain depth $n$ local and global Selmer varieties) is \textit{eventually} satisfied for $n\gg 0$. On the other hand, to prove the non-Zariski density results above, we need uniformity in the $n$, and in all the intermediate dimensions.
\end{remark}

The proof of Theorems \ref{thm:main1} and \ref{thm:main2} depend on a recent unlikely intersection result due to Bl\`azquez-Sanz, Casale, Freitag and Nagloo, who prove an Ax--Schanuel-type theorem for the intersection of leaves of the foliation associated to a principal $G$-bundle with connection on a complex variety with subvarieties of the principal bundle. As we explain in section \ref{sec:AxSchan}, the proof of the $\nabla $-special Ax--Schanuel theorem of Bl\`azquez-Sanz, Casale, Freitag and Nagloo actually implies a stronger result, which can be formulated over an arbitrary field of characteristic zero, and is effective in the following sense. The $\nabla $-special Ax--Schanuel theorem (as formulated below) gives a criterion for non-Zariski density of an intersection of formal schemes inside the formal completion (at a given $K$-point $x$) of a variety $X$ over a field $K$ of characteristic zero. We show that the proof in fact gives an effectively computable proper closed subvariety of $X$ containing this Zariski closure. Note that Urbanik has recently proved related (but much stronger) results about effectivity and fields of definition in a related (but different) context \cite{urbanik1} \cite{urbanik2}.
\begin{corollary}\label{cor:its_effective}
\begin{enumerate}
\item Let $X=\mathbb{P}^1 -\{ 0,1,\infty \}$, $s>0$, and $n$ as in Theorem \ref{thm:main2} part (2). Then there is an effectively computable proper closed subvariety $V$ of $(\mathbb{P}^1 )^{n+1} $ defined over $\Q $ such that, for a sufficiently small $p$-adic neighbourhood $B$ of $x$,
\[
B^{n+1} \cap \cup _{\rk S=s}X(S)^{n+1} \subset V(\Q _p ).
\]
\item Let $\pi :X\to S$ be an isotrivial family of curves over $\Q $, $x\in X(\Q )$ lying above $s\in S(\Q )$, and $r>0$. Let $n$ be as in Theorem \ref{thm:main2} part (2). Then, for any prime $p$ of potential good reduction for $X_s$, there is an effectively computable proper closed subvariety $V$ of $X^n _S $ such that, for a sufficiently small $p$-adic neighbourhood $B$ of $(x,\ldots ,x)$ in $X^n _S (\Q _p )$,
\[
B \cap X^n _S (\Q )_{\rk \leq r} \subset V(\Q _p ).
\]
\item Let $\pi :X\to S$ be a smooth family of curves over $\mathcal{O}_{K_v}$, for $K_v$ a finite extension of $\Q _p$, with a descent of $\pi _{K_v}$ to a number field $K$. Let $r>0$, and $n$ as in part (1) of Theorem (1). Then there is an effectively computable proper closed subvariety $V$ of $X^{n+1}_S $ defined over $K$ such that, for a sufficiently small $p$-adic neighbourhood $B$ of $(x, \ldots ,x)$ in $X^{n+1} _S (\mathcal{O}_{K_v} )$,
\[
B \cap X^{n+1} _S (\mathcal{O}_{K_v})_{\rk \leq r} \subset V(K_v ).
\]
\end{enumerate}
\end{corollary}

\begin{remark}
The elephant in the room is what happens for non-isotrivial families of curves (even assuming Conjecture \ref{BK}). This case will be discussed in the sequel to this paper. The results in this case are weaker, in the sense that some kind of conditions need to be placed on the reduction data of the curve. The reason that such conditions arise is that, unlike in the abelian setting, the Chabauty--Coleman--Kim sets $X_s (\Q _p )_n$ involve nontrivial information from primes of bad reduction, and hence one needs to be able to control this information. The reason it is still possible to say something new in this case is that this information has an explicit, rather algebraic, structure which can be controlled using joint work with Alex Betts \cite{bettsdogra}.
\end{remark}
\subsection{Relation to the Zilber--Pink conjecture}
Pink observed \cite{PinkZP} that the Mordell conjecture can be deduced from a sufficiently strong conjecture about special points on mixed Shimura varieties. Roughly speaking, the observation is that to prove finiteness of $X(\Q )$ for a curve $X/\Q $, it is enough to prove Zariski non-density of $X(\Q )^n \subset X^n _{\Q }$ for some $n>0$. If $n>\rk \Jac (X/\Q )$, then each point of $X(\Q )^n$ will lie in a special subvariety of $\Jac (X/\Q )^n$ (note that this is essentially the same approach to proving finiteness of $X(\Q )$ in Chabauty's theorem).

In the same paper, Pink formulated the following very general conjecture about special points on subvarieties of mixed Shimura varieties.
\begin{conjecture}[Pink,\cite{PinkZP} Conjecture 1.2]
Let $S$ be a mixed Shimura variety over $\mathbb{C}$, and $Z$ an irreducible closed subvariety, with special closure $S_Z$. Then the intersection of $Z$ with the union of all special subvarieties of dimension less than $\codim _{S_Z}Z$ is not Zariski dense in $S$.
\end{conjecture}

Stoll observed that this implies not only the Mordell conjecture, but also the following strong statement about rank $d$ points in general families of curves.
\begin{theorem}[Stoll, \cite{stoll:JEMS}]\label{cor:ZilbPink}
Let $\pi :X\to S$ be a family of curves of genus $g$. Let $\eta $ be the generic point of $S$ and let $s$ be the dimension of the closure of $\eta $ in $\mathcal{M}_g$. Then the Zilber--Pink conjecture implies that the set of rank $r$ points in $X^n _S (\mathbb{C})$ is not Zariski dense in $X^n _S$ whenever 
\[
n>\frac{gr+s}{g-1}.
\]
\end{theorem}
This bound is better than (but similar to) what we prove in the good reduction Chabauty--Coleman case (and of course, it is has no restrictions on the rank in terms of the genus). For bad reduction Chabauty--Coleman, and for general twist families under the Bloch--Kato conjectures \cite{BK}, the bounds we prove are much worse, but as far as we are aware are still significantly better than what is known.

In the case of the unit equation, the Zilber--Pink conjecture predicts that 
$
\cup _{\rk S=r}X^n (S)$ is not Zariski dense in $X^n$ whenever $n>r$. Again, our results are much weaker than those predicted by the Zilber--Pink conjecture, but stronger than what can be proved using known uniformity results.

Note that, from the perspective of tring to understand rational or integral points on varieties, the Zilber--Pink conjecture is far stronger than the statement needed in the proof of Stoll's result. Indeed, to prove Stoll's result over a number field $K$, we only need a `Zilber--Pink conjecture over $K$' that describes the Zariski closure of the intersection of $K$-points of a variety with special subvarieties, where $K$ is a fixed number field. In fact, our method of proof may be viewed as applying Stoll's strategy by proving certain special cases of a `Zilber--Pink conjecture over a $p$-adic field'. The main technical results needed to carry out such an extension are the $\nabla $-special Ax--Schanuel theorem of Bl\`azquez-Sanz, Casale, Freitag and Nagloo, and an extension of the notion of Coleman integration to families. It would be interesting to investigate other instances where this method can be applied to prove cases of such a $p$-adic Zilber--Pink conjecture. Although we do not carry out such an investigation in this paper we illustrate the idea with the following simple case.
\begin{proposition}\label{prop:padicZP}
Let $A$ be a geometrically simple abelian variety of dimension $g$ over a $p$-adic field $K$, let $n>0$, and let $A^{n,[r]}\subset A^n (K )$ denote the union of $K$-points of all codimension $\geq r$ special subvarieties of $A^n$. Let $V\subset A^n$ be an irreducible subvariety of dimension $d$. If $r>g(n- \min \{ n,g\})$, then $V(K)\cap A^{n,[r]}$ is not Zariski dense in $V$ whenever
\[
d\leq \frac{r}{n}-\frac{r(ng-r)}{ng^2}
\]
\end{proposition}
\subsection{Relation to work of Besser}
As explained above, the main technical result needed in the proof of Theorem \ref{thm:main1} is an interpretation of Coleman integration in families in terms of parallel transport with respect to the Gauss--Manin connection. This question has been considered before: in \cite{besser} and \cite{Besser2011differential}, Besser considered the problem of studying Coleman integrals in families. Although the framework he uses is slightly different, the formulas we obtain describing the variation of Coleman integrals in families in terms of the Gauss--Manin connection are, after translation, the same as those used in \cite[1.7.1]{besser}. Besser also gives an interesting interpretation in terms of differential Tannakian categories, which we do not consider in this paper.

It should be emphasised that, in this paper, our aims when studying Coleman integrals in families are rather modest: we are largely concerned with the \textit{local} variation of Coleman integrals. That is, when we work with a family $\pi :X\to S$ of varieties, we only study the variation of Coleman integrals on individual residue disks in $S$. It would be interesting to have a more global theory.
\subsection{Comparison with the work of Manin, Coleman and Chai}
Our method of proof bears some relation with the (Coleman-)Manin proof of the Mordell conjecture over function fields \cite{manin}, \cite{coleman:manin}, \cite{chai}, \cite{bertrand}. Namely, in these works one studies sections of a family of curves $\pi :X\to S$ (i.e .the set of $S$-valued points $X(S)$) using a certain `universal extension' of flat connections coming from the Gauss--Manin connection associated to $\pi $. Somewhat more precisely, this universal extension is a certain tautological extension of the pullback of the Gauss--Manin connection along $\pi $ by the trivial connection on $X$. Pulling back along the section defines a map from $X(S)$ to the abelian group of extensions of the Gauss--Manin connection by the trivial connection in the category of flat connections on $S$.

In this paper, we use the same connection to study the rational, rather than $S$-valued, points of $X$. Namely, suppose $X, S$ and $\pi$ are defined over a finite extension $K$ of $\Q _p$. Then the universal connection can be enriched with the structure of a filtered $F$-isocrystal, and pulling back along a $K$-point gives a map from $K$-points to extensions of filtered $\phi $-modules -- although now the underlying ext group is not fixed, but varies with the base $S$.
\subsection{Comparison with the work of Kim and Lawrence--Venkatesh}
As will be clear, Kim's work on a nonabelian generalisation of the Chabauty--Coleman method is used in an essential way in the proof of Theorem \ref{thm:main2}. Kim's method makes use of a commutative diagram of the form
\begin{equation}\label{eqn:the_square}
\begin{tikzcd}
X(\mathbb{Z}_T ) \arrow[d] \arrow[r] & H^1 _{\mathcal{L}}(G_{\Q },U) \arrow[d] \\
X(\mathbb{Z}_p ) \arrow[r]           & H^1 _{\mathcal{L}}(G_{\Q _p },U).          
\end{tikzcd} 
\end{equation}
Here $U$ is the set of $\Q _p$-points of a nonabelian (unipotent) algebraic group over $\Q _p $ with a continuous action of $G_{\Q }$, $T$ denotes a finite set of primes, and $\mathcal{L}$ denotes certain finiteness conditions on a nonabelian cohomology set.

The Lawrence--Venkatesh method \cite{LV}, roughly speaking, makes use of a somewhat similar diagram, where the top right term is now replaced by a finite set of semisimple weight $1$ Galois representations unramified outside a finite set of primes, and the bottom right term is now a set of isomorphism classes of filtered $\phi $-modules \footnote{This is not strictly accurate in that the existence of such a diagram uses a semisimplicity result which the authors do not assume, but the (ingenious) workaround is less relevant to the discussion.}.

One way to make these arguments work in families is to consider families of varieties with good reduction outside a fixed finite set of primes. By Faltings' proof of the Shafarevich conjecture, this means that one is only considering finitely many curves, hence the potential applications are of a more quantitative or effective nature. In the context of the Lawrence--Venkatesh method, this has been considered in a recent preprint of Lemos and Torzewski \cite{lemos2022bounds}. In the context of the Chabauty--Coleman--Kim method something somewhat related is considered in the thesis of Noam Kantor \cite{kantor2020chabauty} (although strictly speaking the latter considers a fixed curve). 

In this paper, we take a different approach, and do not impose a condition of good reduction outside a fixed finite set of primes. This forces us to in some sense only work globally fibre-by-fibre. More precisely, as we vary our curve $X$ in a family, by the Kim--Olsson comparison theorem \cite{kim3}
\[
H^1 _f (G_{\Q _p },U)\simeq U_{\dR} /F^0 ,
\]
we can meaningfully vary the bottom right term of \eqref{eqn:the_square} in families. The variation of the unipotent Albanese morphism in families introduces locally analytic functions on a $p$-adic variety which nontrivially extend the notion of Coleman functions introduced by Besser \cite{besser}.

However, the top right term $H^1 _{\mathcal{L}} (G_{\Q },U)$ varies in a less structured way. The simple idea which we pursue in this paper is to instead work with fibre products, and make use of the fact that $H^1 _f (G_{\Q _p },U)$ is a subvariety of the affine space of \textit{explicitly bounded degree}. To simplify things, consider the classical Chabauty--Coleman setting. Suppose we have a family of genus $g$ curves $\pi :X\to S$, and infinitely many $\Q $-points $s\in S(\Q )$ where $\Jac (X_s )$ has rank $r\leq g$. For each such $s$, the Chabauty--Coleman method works by intersecting the curve $X_s (\Q _p )$ with the (at most) $r$-dimensional linear subspace $\overline{\Jac (X_s )(\Q )}_{\Q _p }$ of $H^1 (X_{s,\Q _p },\mathcal{O})$ generated by $\Jac (X_s (\Q _p ))$. 

If we now consider the total space of the vector bundle $R^1 \pi _* \mathcal{O}_X$, then there is no reason to expect any kind of $p$-adic continuity in the different $r$-dimensional subspace we get. However, when we look instead at $X^n _S $, for $n>r$, we find that the subspaces $\overline{\Jac (X_s )(\Q )}^n _{\Q _p }$ lie in a proper subscheme of the total space of $R^1 \pi _* \mathcal{O}_X ^{\oplus n}$, namely the rank $\leq r$ degeneracy locus, which gives locally analytic functions on $X^n _S (\Q _p )$ vanishing on rank $r$ $n$-tuples. Using a recent theorem of Bl\`azquez-Sanz, Casale, Freitag and Nagloo \cite{BCFN}, we show that the zero locus of these functions is not Zariski dense.

To extend this argument to the nonabelian setting, we have to replace linear subspaces with `filtered subschemes' in the sense of Betts \cite{betts21}. This construction is carried out in section \ref{sec:filtered}.

\subsection{The role of the $p$-adic Betti map}
This construction of local analytic isomorphisms between different fibres of a family of abelian varieties can be thought of as a $p$-adic analytic analogue of the \textit{Betti map} \cite{ACZ} \cite{CGHX} \cite{gao:betti} which is ubiquituous in recent breakthroughs in Diophantine geometry using unlikely intersections such as \cite{kuhne21} \cite{DGH21} \cite{gao2023relative} \cite{gao2021uniform}. We hope that this relation is not merely superficial. For example, a natural question would be to use the notion of degenerate subvarieties (in the sense of \cite{gao:betti}) to obtain a refined description of the Zariski closure of the sets $X^n _S (K)_{\Gamma -\rk \leq r}$ arising in the theorems above.

\subsection{Notation}
By a $p$-adic field $K$ we will always mean a finite extension of $\Q _p$. If a $p$-adic field is denoted by $K$, then $\mathcal{O}_K$ will denote the ring of integers, $\pi _K$ the uniformiser, and $k$ the residue field.

If $A$ and $B$ are two filtered objects in an abelian category, whose associated gradeds are isomorphic, we will say that an isomorphism $f :A\to B$ is \textit{unipotent} if it respects the filtrations and is equal to the identity on the associated gradeds.

If $X^n _S$ is an $n$-fold fibre product, we shall denote the $i$th projection $X^n _S \to X$ by $\pr _i $, and the projection $X^n _S \to S$ by $\pi ^n$. If $U\to S$ is an object over $S$, we define $X^n _U :=(X^n _S )\times _S U$.

\subsection*{Acknowledgements}
Much of this paper was inspired by ongoing joint work with Jan Vonk. The author is also grateful to Alex Betts for helpful discusssions regarding filtered schemes, and to Dan Abramovich for some corrections to an earlier version of this paper.
\section{The $\nabla $-special Ax--Schanuel theorem for principal $G$-bundles}\label{sec:AxSchan}
Let $X$ be a variety over a field $K$ of characteristic zero, and $G$ a group over $K$. Let $\pi :P\to X$ be a principal $G$-bundle. A connection on $P$ is a section $\nabla :\Omega _{P|K}\to \pi ^* \Omega _{X|K}$ of the exact sequence of $\mathcal{O}_P$-modules
\begin{equation}\label{eqn:connectionSES}
0\to \pi ^* \Omega _{X|K}\to \Omega _{P|K}\to \Omega _{P|X}\to 0.
\end{equation}
We will sometimes think of $\nabla $ as a morphism $\pi ^* TX\to TP$. A connection is $G$-principal if it is $G$-equivariant with respect to the natural $G$-action on \eqref{eqn:connectionSES}.

A connection $\nabla$ is flat if it respects differentials, i.e. if the diagram
\[
\begin{tikzcd}
\Omega _{P|K} \arrow[d] \arrow[r, "\nabla "] & \pi ^* \Omega _{X|K} \arrow[d] \\
\Omega ^2 _{P|K} \arrow[r, "\nabla \wedge \nabla"]           & \pi ^* \Omega ^2 _{X|K}          
\end{tikzcd}
\]
commutes. Dually, thought of as a morphism between tangent bundles, a flat connection is one which respects the Lie bracket.

The isomorphism $P\times _X P\simeq P\times _K G$ induces an isomorphism $\Omega _{P|X}^* \simeq \mathfrak{g}\otimes \mathcal{O}_P$. Via this isomorphism, the section $\nabla $ defines the connection form $\Omega \in H^0 (P,\Omega _{P|K}\otimes \mathfrak{g})$.

The most important example in this paper will be the frame bundle of a vector bundle. If $\mathcal{V}$ is a vector bundle of rank $n$, then the frame bundle $P_{\mathcal{V}}$ is the sheaf
\[
U\mapsto \Isom (\mathcal{O}_U ^{\oplus n},\mathcal{V}|_U ).
\]
This is a principal $\GL _n$-bundle.
Suppose $P_{\mathcal{V}}$ is the frame bundle associated to a vector bundle with flat connection $(\mathcal{V},\nabla )$. Then we define a connection on $P$ as follows. Write $\mathcal{V}$ in terms of Cech cocycles as $(U_{\alpha },\psi _{\alpha \beta })$, and a cocycle description $(\Lambda _{\alpha },g_{\alpha \beta })$ of $\nabla $. Then the frame bundle $P_{\mathcal{V}}$ can be trivialised on $U_{\alpha }$, and the connection form is given on $U_{\alpha }$ by 
\[
T^{-1}dT-T^{-1}\Lambda _{\alpha }T,
\]
where $T$ is the identity morphism on $\GL_n$.
\subsection{Foliations and the $\nabla$-special Ax--Schanuel theorem}
By a \textit{foliation} on a variety $Y$ over $K$ we shall mean a sub-bundle $\mathcal{F}$ of the tangent bundle which is closed under the Lie bracket (called an \textit{involutive sub-bundle} in \cite{Bost01}). We say a foliation $\mathcal{F}$ on a connected variety $Y$ has a \textit{rational first integral} if there is a non-constant $f\in K(Y)$ such that $df \in \mathcal{F}\otimes _{\mathcal{O}_Y} K(Y)$.
By a formal horizontal leaf of $\mathcal{F}$ at a point $x\in X(K)$, we shall mean the unique closed formal subscheme $i:\widehat{\mathbb{A}}^d \hookrightarrow \widehat{X}_x$ such that $T\widehat{\mathbb{A}}^d \simeq i^* \mathcal{F}$ as subbundles of $i^* T\widehat{X}_x$, as in \cite[3.4.1]{Bost01}.

Let $\nabla $ be a principal $G$-connection on a principal $G$-bundle $\pi :P\to X$. The foliation associated to $\nabla $ is defined to be $\nabla (\pi ^* TX)$. This foliation may be characterised as the sub-bundle of $\Omega _{P|K}$ spanned by the coordinates of the connection form $\Omega \in \Gamma (P,\Omega _{P|K})\otimes _K \mathfrak{g}$ with respect to a $K$-basis of $\mathfrak{g}$.

A formal horizontal leaf of $\nabla$ at $x\in P(K)$ is a formal horizontal leaf of the associated foliation. Note that a formal horizontal leaf is, in particular, a section of $\widehat{\pi }_x :\widehat{P}_x \to \widehat{X}_{\pi (x)}$, the formal completion of $\pi $ at $x$.

We say that a principal $G$-bundle with connection $(P,\nabla )$ has Galois group $G$ if $(P,\nabla )$ does not descend to a principal $H$-bundle with connection for any proper subgroup $H$ of $G$. 

\begin{lemma}
$(P,\nabla )$ has Galois group $G$ if and only if for any point $x\in P$, the formal horizontal leaf of the associated foliation is Zariski dense in $G$.
\end{lemma}
\begin{proof}
If $(P,\nabla )$ descends to an $H$-bundle with connection $P'$ for $H<G$, then the formal horizontal leaf will have Zariski closure contained in $P'$. Conversely, suppose that a formal horizontal leaf is not Zariski dense. Without loss of generality $K$ is finitely generated over $\Q $, and we may choose an embedding $K\hookrightarrow \mathbb{C}$. Let $\rho :\pi _1 (X_{\mathbb{C}},x)\to G(\mathbb{C})$ be the corresponding monodromy representation. We have a proper subvariety $P'$ of $P$ such that the parallel transport isomorphism is contained in $P'$ near $x$. Hence we deduce that $\rho $ is not Zariski dense in $G$, and hence $(P,\nabla )$ descends to a principal $H$-bundle with connection for some $H$.
\end{proof}
Following \cite{BCFN}, we say that $(P,\nabla )$ has Galois group $H<G$ if $(P,\nabla )$ descends to a principal $H$-bundle with connection with Galois group $H$.

We say a group $G$ over $K$ is sparse if every proper sub-Lie algebra $\mathfrak{h}\to \mathfrak{g}$ of $\mathfrak{g}:=\Lie (G)$ is contained in Lie algebra of a proper subgroup $H\to G$. 
\begin{lemma}
If the maximal reductive quotient $G_0$ of $G$ is semi-simple, then $G$ is sparse.
\end{lemma}
\begin{proof}
As in \cite{BCFN}, defined the algebraic envelope $\overline{L}$ of a Lie subalgebra $L \subset \mathfrak{g}$ to be the smallest Lie-subalgebra of an algebraic subgroup of $G$ containing $L$. Then $L$ is an ideal of $\overline{L}$, hence the result follows when $G$ is semi-simple. In general, we reduce to the case when $L$ surjects onto $\Lie (G_0 )$. The kernel $I$ of $L\to \Lie (G_0 )$ is a $G_0$-stable sub-Lie algebra of the nilpotent Lie algebra $\Ker (G\to G_0 )$ (and by hypothesis is proper), and hence can be exponentiated to a $G_0$-stable proper subgroup $\exp (I)$ of $\Ker (G\to G_0 )$. The induced semi-direct product of $G_0 $ by $\exp (I)$ is then the desired proper subgroup of $G$.
\end{proof}
A $\nabla $-special subvariety of $X$ (relative to $\nabla $) is a subvariety $i:Y\hookrightarrow X$ such that the $G$-principal connection $i^* \nabla $ on $P\times _X Y$ has Galois group strictly smaller than $G$. The $\nabla$-special Ax--Schanuel theorem of Bl\`azquez-Sanz, Casale, Freitag and Nagloo is as follows.
\begin{theorem}[\cite{BCFN}, Theorem 3.6]\label{thm:BCFN}
Let $G$ be a sparse group over $K$, and let $\nabla $ be a $G$-principal connection on $\pi :P\to X$ with Galois group $G$. Let $V$ be a subvariety of $P$, $x\in V(K)$, and let $\mathcal{L}\subset \widehat{P}_x$ be a formal horizontal leaf through $x$. Let $W$ be an irreducible component of $V\cap \mathcal{L}$. If
\[
\dim V<\dim W +\dim G,
\]
then the projection of $W$ in $X$ is contained in a finite union of $\nabla$-special subvarieties.

If $X, G, \pi ,x$ and $V$ are defined over a number field $K$, then there is a proper closed subvariety of $X$ containing a neighbourhood of $\pi (W)$ which is effectively computable in terms of $X,G ,\pi $ and $V$.
\end{theorem}
Note that Bl\`azquez-Sanz, Casale, Freitag and Nagloo only state this theorem in the case $K=\mathbb{C}$, and hence in particular do not formulate the second claim about explicit computability. However, as we explain below, these results follow straightforwardly from the proof they give in loc. cit. The claim about fields of definition is immediate, so we focus on explaining the effectivity.

First, note that we may assume that $\dim W>0$. Let $\Omega \in H^0 (P,\Omega _{P|K}\otimes \mathfrak{g}) $ denote the connection form for $P$ as defined above. We can identify the image of the tangent space of $\mathcal{L}$ in that of $\widehat{P}_x$ with the kernel of $\Omega $ as explained above. Hence the tangent space of $W$ is contained in 
\[
\Ker (\Omega |_V )\simeq \Ker (\Omega )\cap i_* T_x V.
\]
Suppose that $\Omega |_V$ generically has rank $r$. This is equivalent to saying that the induced section $\xi _r$ of $\Omega ^r _{V|K}\otimes \wedge ^r \mathfrak{g}$ is nonzero, but the induced section $\xi _{r+1}$ of $\Omega ^{r+1} _{V|K}\otimes \wedge ^{r+1} \mathfrak{g}$ is zero. The locus of rank $<r$ points is exactly the vanishing locus of $\xi _r$. Let $\mathcal{H}\subset \mathfrak{g}^* \otimes \mathcal{O}_X$ denote the kernel of the induced map
\[
\mathfrak{g}^* \otimes \mathcal{O}_X \to \Omega _{X|K}.
\]

If $A$ is an integral, complete $K$-algebra topologically of finite type, with field of fractions $F$ we define $\widehat{\Omega }_{F|K}:=\widehat{\Omega }_{A|K}\otimes _A F$, and define the differential map $d:F\to \widehat{\Omega }_{F|K}$ in the obvious way. If $f\in F$ satisfies $df=0$ in $\widehat{\Omega }_{F|K}$, then $K[f]$ is finite over $K$. 

\begin{lemma}
At least one of the following holds.
\begin{enumerate}
\item The foliation on $V$ has an effectively computable rational first integral.
\item There is an effectively computable Lie subalgebra $\mathfrak{h}\subset \mathfrak{g}$ such that $\Omega (\Omega _{K(V)|K}^* )\subset \mathfrak{h}$.
\end{enumerate}
\end{lemma}
\begin{proof}
Let $v_1 ,\ldots ,v_n$ be a basis of $\mathfrak{g}$. Suppose the rank of $\Omega $ is $m$. Re-ordering the $v_i$ if necessary, we may write (i.e. effectively compute) a basis of $\Ker (\Omega \otimes K(V) )$ of the form
\[
v_{m+i}^* +\sum _{i=1}^m f_{ij} v_j ^* .
\]
For any derivation $D$ in the kernel of $\Omega $, we have $D(f_{ij})=0$, hence $df_{ij}=0$ in $\widehat{\Omega }_{K(\! (W)\! )|K}$. If all the $f_{ij}$ are constant, then as in \cite[Proof of Theorem 3.6]{BCFN}, they define a proper Lie subalgebra of $\mathfrak{g}$ to which $\Omega |_V$ descends. If not, then we obtain an effectively computable rational first integral.

\end{proof}
Given an effectively computable rational first integral $df$, we obtain an effectively computable function in the local ring at $x$ which vanishes along $\mathcal{L}$. If $f_i $ lies in $\mathcal{O}_{X,x}$, this function is just $f-f(x)$. Suppose not, so that $f=g/h$, where $g\in \mathcal{O}_{X,x}$, and $h\in \mathfrak{m}_{X,x}$. Then $g$ must also vanish at $x$, and $\mathcal{L}$ is contained in the zero locus of $g\cdot h$. Indeed, we have that $df$ lies in the kernel of 
\[
\widehat{\Omega }_{P,x}[1/h]\to \Omega _{L|K}[1/h],
\]
from which it follows that $f$ is constant on $\mathcal{O}(\mathcal{L})[\frac{1}{g}]$, and hence must be equal to zero.

We deduce that, under the assumptions in the statement of Theorem \ref{thm:BCFN}, either there is an effectively computable proper subvariety of $V$ containing $W$, or there is an effectively computable proper Lie subalgebra $\mathfrak{h}$ of $\mathfrak{g}$ such that 
\[
\Omega (T_x V) \subset \mathfrak{h}.
\]
for all $x\in V(K)$. Iterating this process, we arrive at the effectively computable closed subvariety of $X$ in the statement of Theorem \ref{thm:BCFN}.

\section{$p$-adic integrals in families and a $p$-adic Betti map}\label{sec:families}
In this section $K$ will denote a finite extension of $\Q _p$ with residue field $k$ and ring of integers $\mathcal{O}$.
\subsection{Review of rigid geometry}
Let $P\to \Spf (\mathcal{O})$ be a formal $\mathcal{O}$-scheme which is flat and finite type over $\mathcal{O}$. Let $P_K$ denote the generic fibre in the sense of Berthleot \cite{Ber96}, $P_k $ the special fibre. Let $\spe:P_K \to P_k$ denote the specialisation map. The \textit{tube} of a subvariety $X\subset P_k$ is $\spe^{-1}X\subset P_K$. As explained in \cite{Ber96},\cite{lestum}, this inherits the structure of a rigid analytic space.

\begin{definition}
For $P$ as above, we will say that a function $f:P_K (K)\to K$ is \textit{locally analytic} if, for each point $x\in P_k (k)$, $f|_{]x[(K)}$ comes from a rigid analytic function on the tube $]x[$.
\end{definition}

\begin{lemma}\label{lemma:analytic2formal}
Let $X$ be a geometrically irreducible variety over $\mathcal{O}_K$, and $x\in X(k)$. Let $f_1 ,\ldots ,f_n $ be rigid analytic functions on $]x[$. Let $Y\subset X_K$ be a geometrically irreducible subvariety of the generic fibre such that $Y^{\an }\cap ]x[ (K)$ is nonempty. Suppose that, for a Zariski dense set of points $x\in Y(K)$, the Zariski closure of the common zeroes of $f_1 ,\ldots ,f_n$ restricted to the formal completion of $Y$ along $x$ is not Zariski dense in $Y$. Then the common zeroes of $f_1 ,\ldots ,f_n$ are not Zariski dense in $Y$.
\end{lemma}
\begin{proof}
Since the $K$ points of $]x[$ are contained in a closed disk $D\subset ]x[$, it is enough to prove this for a closed disk $D\subset ]x[$. The common zeroes of $f_1 ,\ldots ,f_n$ on $D$ will then have finitely many irreducible components $Z_i$. It is hence enough to prove that $Z_i \cap Y$ is not Zariski dense in $Y$ for each $i$. Let $Y_0 \subset Y$ be the closed subvariety outside of which the Zariski nondensity condition on formal completions in the statement of the lemma is satisfied. Pick $y\in Z_i \cap Y (K)$. We may assume that $y$ does not lie in $Z$. The Zariski nondensity condition on $\widehat{Y}_y$ implies there is a $g\in \mathcal{O}_{Y,y}$ vanishing on $\widehat{\mathcal{O}}_{Z_i \cap Y,y}$. Hence $g$ vanishes in $\mathcal{O}_{Z_i \cap Y,y}$, and thus $Z_i \cap Y$ is not Zariski dense in $Y$.
\end{proof}

Given a variety $Y$ over $k$, we define a \textit{rigid triple} (over $\mathcal{O}$) for $Y$ to be a triple $T=(Y,X,P)$, together with $j:Y\hookrightarrow X$ an open immersion into a smooth projective $k$-variety, and $i:X\hookrightarrow P$ a closed immersion into a formal $\mathcal{O}$-scheme $P$, where $P\to \Spf (\mathcal{O} )$ is smooth along an open neighbourhood of $X$. Abusing notation somewhat, we will also sometimes refer to $(Y,X,P)$ as a rigid triple when $i:X\hookrightarrow P$ is a closed immersion of $X$ into a smooth $\mathcal{O}_K$-scheme $P$ -- in this case the actual rigid triple is given by taking the closed immersion of $X$ into the formal completion of $X$ along its special fibre. We denote the category of overconvergent isocrystals, in the sense of \cite{Ber96} and \cite{lestum}, on $T$ by $\mathcal{C}(T)$, and the category of unipotent overconvergent isocrystals by $\mathcal{C}^{\un }(T)$.

\begin{lemma}\label{global_sections}[\cite{lestum}, Corollary 7.1.7]
An overconvergent isocrystal $\mathcal{V}$ on the tube $T$ of a point admits a basis of horizontal sections, i.e.
\[
\mathcal{V}\simeq \mathcal{V}(T)^{\nabla =0}\otimes _K \mathcal{O}_T .
\]
\end{lemma}
Lemma \ref{global_sections} implies that, for any $\overline{x}\in Y(k)$,
\[
\mathcal{V}\mapsto \mathcal{V}(]\overline{x}[)^{\nabla =0}
\] 
induces a fibre functor $\overline{x}$ on $\mathcal{C}(T)$, and an isomorphism of fibre functors 
\begin{equation}\label{eqn:basis_of_sections}
x^* \simeq \overline{x}^*,
\end{equation} for any $\mathcal{O}_K$ point $x$ of $P$ specialising to $\overline{x}$. Given a flat connection $(\mathcal{V},\nabla )$ over a smooth variety $X/\mathcal{O}_K$, and points $x,y \in X(K)$ with a common reduction $z\in X(k)$, \eqref{eqn:basis_of_sections} induces a \textit{parallel transport} isomorphism
\begin{equation}\label{eqn:parallel_transport}
x^* \mathcal{V} \simeq y^* \mathcal{V}
\end{equation}
defined simply as the composite
\[
x^* \mathcal{V} \simeq x^* \mathcal{V}^{\an }\simeq \mathcal{V}^{\an }(]z[)^{\nabla =0}\simeq y^* \mathcal{V}^{\an }\simeq y^* \mathcal{V}.
\]

\begin{definition}
Let $X$ be a variety over a field $K$ of characteristic zero, and $x\in X(K)$. We define the unipotent de Rham fundamental group $\pi _1 ^{\dR}(X,x)$, to be the Tannakian fundamental group associated to the category of unipotent vector bundles with flat connection, and fibre functor $x^* $. 

For $K$ a finite extension of $\Q _p$ as above, and a rigid triple $T=(Y,X,P)$, and a $k$-point $x\in Y(k)$, we define the rigid fundamental group $\pi _1 (T,x)$ to be the Tannakian fundamental group associated to the neutral Tannakian category of unipotent overconvergent isocrystals on $T$, and fibre functor $x$. 
\end{definition}

For two different rigid triples $T=(Y,X,P)$ and $T'=(Y',X',P')$, with points $x\in Y(k),x'\in Y'(k)$, a $k$-isomorphism $f:Y\stackrel{\simeq }{\longrightarrow }Y'$ sending $x$ to $x'$ induces an isomorphism on fundamental groups $f_* :\pi _1 (T,x)\stackrel{\simeq}{\longrightarrow }\pi _1 (T',x')$ \cite[Corollary 7.1.7]{lestum}. In particular, $\pi _1 (T,x)$ can be viewed as only depending on the pair $(Y,x)$, defining a rigid fundamental group $\pi _1 ^{\rig }(Y/\mathcal{O}_K ,x):=\pi _1 (T,x)$ which is functorial in pointed varieties over $k$. Functoriality of the fundamental group also implies an action of $\phi $ on $\pi _1 ^{\rig }(Y/\mathcal{O}_K ,x)$, where $\phi $ is a suitable power of absolute Frobenius acting as the identity on $k$.

The Berthelot--Ogus comparison in cohomology implies the following isomorphism of fundamental groups.
\begin{theorem}[\cite{CLS}]\label{thm:CLS}
For any projective variety $X$ over $\mathcal{O}_K$, with formal completion $\mathfrak{X}$ along $X_k$, and any open affine $Y\subset X$, the analytification functor defines an equivalence of categories between the category $\mathcal{C}^{\dR}(Y_K /K)$ of unipotent flat connections on $Y_K /K$ and $\mathcal{C}^{\un }(T)$, where  where $T$ denotes the rigid triple $(Y_k ,X_k ,\mathfrak{X})$.
\end{theorem}
For any $\mathcal{O}_K$-points $x,y $ of $X$, with reductions $\overline{x}$ and $\overline{y}$, this defines a $\phi $-action on the de Rham fundamental group $\pi _1 ^{\dR}(X_K /K ,x)$ and path torsor $\pi _1 ^{\dR}(X_K /K ;x,y)$, via the isomorphism of fibre functors $x^* \simeq \overline{x}^* $ and $y^* \simeq \overline{y}^*$ induced by \eqref{eqn:basis_of_sections}.

The Coleman integral (defined in the next subsection) can be constructed using the following crucial result due to Besser (previous incarnations of this result are also used in the original definitions of Coleman).
\begin{theorem}[\cite{besser}]\label{thm:besserpath}
For any variety $V/\mathcal{O}_K$ and any $x,y\in V(\mathcal{O}_K )$, there is a unique $\phi $-invariant path $P_{x,y}=P_{x,y}^V \in \pi _1 ^{\dR}(V_K /K ;x,y)$. Similarly, for any $x,y\in V(k)$ there is a unique $\phi $=invariant path $P_{x,y}^V \in \pi _1 ^{\rig }(V_k /\mathcal{O}_K ;x,y)$.
\end{theorem}
\subsection{Universal connections}
Given a pointed variety $(X,x)$ over $K$, we define a pointed connection to be a pair $(\mathcal{V},v)$ consisting of vector bundle with flat connection $(\mathcal{V},\nabla )$ on $X/K$ and a vector $v\in x^* \mathcal{V}$. A morphism of pointed connections is a morphism of connections which sends one distinguished vector to the other. Given $(X,x)$ as above, we say that a pointed connection $(\mathcal{V},v)$ is a universal $n$-unipotent connection if for any $n$-unipotent pointed flat connection $(\mathcal{W},w)$, there is a unique morphism of pointed connections $(\mathcal{V},v)\to (\mathcal{W},w)$.

We say a neutral Tannakian category over a field $K$ is unipotent if its fundamental group is pro-unipotent, and $\pi _1 (\mathcal{C},\omega )^{\ab }$ is finite dimensional over $K$.
\begin{lemma}\label{lemma:tannak}[\cite{wildeshaus}, \cite{besser:heidelberg}]
Let $(\mathcal{C},\omega )$ a unipotent Tannakian category. The action of $\pi _1 (\mathcal{C},x)$ on the set of isomorphism classes of extensions of the unit object $\mathbf{1}$ by itself induces an isomorphism of vector spaces
\[
\pi _1 (\mathcal{C},\omega )^{\ab }\simeq \ext ^1 _{\mathcal{C}}(\mathbf{1},\mathbf{1})^*
\]
\end{lemma}
In the case of $\mathcal{C}^{\dR }(X_K /K )$ and $\mathcal{C}^{\rig  }(X_k /\mathcal{O}_K )$ this defines isomorphisms
\begin{equation}
H^1 _{\dR}(X_K /K)^* \simeq \pi _1 ^{\dR, \un}(X/K ,x)^{\ab }
\end{equation}
and 
\begin{equation}
H^1 _{\rig}(X_k /K)^* \simeq \pi _1 ^{\rig, \un}(X_k /\mathcal{O}_K  ,y)^{\ab }
\end{equation}
respectively, for any $x\in X(K)$ and $y\in X(k)$.

For any unipotent Tannakian category $(\mathcal{C},\omega )$, we can define a pointed object as a pair $(V,v)$ where $V\in \mathcal{C}$ and $v\in \omega (V)$, and we can similarly define a notion of universal $n$-unipotent pointed object.
\begin{lemma}
In any unipotent Tannakian category $(\mathcal{C},\omega )$, a universal $n$-unipotent pointed object exists and is unique up to unique isomorphism.
\end{lemma}
\begin{proof}
By Tannaka duality, it is enough to prove the corresponding claim in the Tannakian category $\Rep (\pi _1 (\mathcal{C},\omega ))$ of (finite-dimensional) representations $\pi _1 (\mathcal{C},\omega )$ (with the usual forgetful fibre functor). We can then take the universal pointed object to be the quotient of the pro-universal enveloping algebra of $\Lie (\pi _1 (\mathcal{C},\omega ))$ by the $(n+1)$th power of its augmentation ideal. This is an object of $\Rep (\pi _1 (\mathcal{C},\omega ))$ because of the finiteness assumption in the definition of a unipotent Tannakian category.
\end{proof}
Let $(\mathcal{E}_n ^{\dR}(X,x), e_n )$ denote a universal $n$-unipotent pointed connection on $(X,x)$. By construction, the collection of $(\mathcal{E}_n ^{\dR}(X,x),e_n )$ for $n>0$ have a unique structure of a pro-object in pointed connections on $(X,x)$. Note that the universal property of $(\mathcal{E}^{\dR}_n (X,x),e_n )$ implies that $P^X _{x,y}$, or indeed any element of $\pi _1 ^{\dR}(X/K ;x,y)$, is uniquely determined by its value on all the $e_n$.

As we only study (or rather, only try to prove new things about) \textit{abelian} Coleman integrals in this paper, we only need to consider this object when $n=1$. In this case, the underlying connection is simply the universal extension $\mathcal{E}(X)$ of the trivial connection $\mathcal{O}_X $ by the trivial connection $H^1 _{\dR}(X/K)^* \otimes \mathcal{O}_X$, sitting in an exact sequence
\begin{equation}\label{eqn:univ_ext}
0\to H^1 _{\dR}(X/K)^* \otimes _K \mathcal{O}_X  \to \mathcal{E}(X)\to \mathcal{O}_X \to 0,
\end{equation}
with the property that the class of $\mathcal{E}(X)$ in $\ext ^1 _{\mathcal{C}^{\dR}(X)}(\mathcal{O}_X ,H^1 _{\dR}(X)^* \otimes \mathcal{O}_X )$ is identified with the identity endomorphism via the isomorphism
\[
\ext ^1 _{\mathcal{C}^{\dR}(X)}(\mathcal{O}_X ,H^1 _{\dR}(X)^* \otimes \mathcal{O}_X ) \simeq H^1 _{\dR}(X,H^1 _{\dR}(X/K)^* \otimes \mathcal{O}_X )\simeq \End (H^1 _{\dR}(X/K)).
\]
Recall that, when we say that $\mathcal{E}(X)$ is the universal connection, we are also implicitly fixing a choice of $e_1 \in x^* \mathcal{E}(X)$ mapping to $1\in x^* \mathcal{O}_X$. This construction also makes sense in an analytic, or rigid analytic category, and in that context we shall use identical notation.

This construction admits the following relative formulation. Given a smooth morphism $\pi :X\to S$ of varieties over a field $K$ of characteristic zero, let $\mathcal{V} $ denote the dual of the Gauss--Manin connection $R^1 \pi _* \Omega _{X|S}$, which has the structure of a vector bundle with flat connection on $S$. Given a flat connection $\mathcal{W}$ on $X$, from the Leray spectral sequence in de Rham cohomology \cite{hartshorne} \cite{KO}, we obtain an exact sequence
\begin{equation}\label{eqn:leray}
0\to H^1 _{\dR} (S/K ,\pi _* \mathcal{W}\otimes \Omega _{X|S}^\bullet )\to H^1 _{\dR} (X/K,\mathcal{W})\to (R^1 \pi _* \mathcal{W}\otimes \Omega ^\bullet _{X|S})^{\nabla =0}.
\end{equation}
In the case where $\mathcal{W}=\pi ^* \mathcal{V}$, the right hand term may be identified with $\End (\mathcal{V})$, i.e. with the  homomorphisms of connections from $\mathcal{V}$ to itself. Let $\mathcal{E}(X/S)$ denote a vector bundle on $X$ with a connection on $X/S$ (i.e. taking values in $\Omega _{X|S}\otimes \mathcal{E}(X/S)$) which is an extension of constant connections
\[
0\to \pi ^* (R^1 \pi _* \Omega _{X|S})^* \to \mathcal{E}(X/S)\to \mathcal{O}_X \to 0,
\]
whose class in $H^1 _{\dR}(X/S,(R^1 \pi _* \Omega _{X|S})^* )$ is sent to the identity element under the isomorphism
\[
H^1 _{\dR}(X/S,(R^1 \pi _* \Omega _{X|S})^* )\simeq \End (H^1 _{\dR}(X/S )).
\]
The following is just a special case of \cite[Theorem 3.1]{CDPS} (related constructions also occur in \cite{lazda} and \cite{AIK}).
\begin{lemma}\label{lemma:lifting_univ_conn}
The relative connection $\mathcal{E}(X/S)$ on $X/S$ lifts to a connection on $X/K$, and such a lift $\mathcal{E}$ is uniquely determined by the extension class in $H^1 _{\dR}(S,\mathcal{V})$ obtained from pulling back $\mathcal{E}$ along a section of $\pi $.
\end{lemma}
\begin{proof}
By construction the class of the identity in $H^1 _{\dR}(X/S ,(R^1 \pi _* \Omega _{X|S})^* )$ maps to zero in $H^2 _{\dR}(S,(R^1 \pi _*  \Omega _{X|S})^* )$. Let $s$ be a section of $\pi $. Then $s^*$ induces a section of \eqref{eqn:leray}, since the morphism $H^1 _{\dR}(S/K,\pi _* \mathcal{V})\to H^1 _{\dR}(X/K,\mathcal{V})$ is given by pulling back along $ \pi $.
\end{proof}
\begin{definition}\label{defn:lifting_univ_conn}
We define the relative universal connection of a family $X\to S$ with a section $s$ to be a lift $\mathcal{E}$ as in Lemma \ref{lemma:lifting_univ_conn} whose pullback along $s$ vanishes in $H^1 _{\dR}(S,\mathcal{V})$, together with a choice of splitting $e_1 :\mathcal{O}_S \to s^* \mathcal{E}$.
\end{definition}
\subsection{Coleman integration}\label{subsec:coleman}
Recall that Besser \cite{besser} \cite{besser:heidelberg} defines an \textit{abstract Coleman function} as an overconvergent isocrystal $(\mathcal{V},\nabla )$ together with a morphism (of bundles) $s:\mathcal{V}\to \mathcal{O}$, and a tuple $(f_x )_{x\in X(k)}$ of horizontal sections $f_x \in \mathcal{V}(]x[)^{\nabla =0}$ such that, for all $x,y \in X(K)$, 
\begin{equation}\label{eqn:whatis...}
P^X_{x,y}(f_x)=f_y .
\end{equation}
By the \textit{value} of an abstract Coleman integral $((\mathcal{V},\nabla ),s,(f_x ))$ at $z\in X(K)$, we shall mean the value of $s(f_x )$, where $z\in ]x[(K)$. 
For any $z\in X(K)$, given $((\mathcal{V},\nabla ),s)$ as above, the set of $(f_x) _{x\in X(k)}$ extending it to an abstract Coleman function is a $z^* \mathcal{V}$-torsor. In particular, given $(\mathcal{V},\nabla )$, a tuple $(f_x)$ satisfying \eqref{eqn:whatis...} is uniquely determined by $z^* f_{z_0 }$, where $z_0 \in X(k)$ is the reduction of $z$. 

In particular, given $X$ a variety, $\omega \in H^0 (X,\Omega _{X|K})$, and $x,y\in X(\mathcal{O}_K )$, we can define $\int ^y _x \omega $ to be the value of the abstract Coleman function $((\mathcal{V},\nabla ),s,(f_x ))$ at $y$, where $\mathcal{V}=\mathcal{O}_X ^{\oplus 2}$, $\nabla $ is the connection \[
\nabla := d-\left( \begin{array}{cc} 0 & 0 \\ \omega & 0 \end{array} \right) ,
\]
$s$ is projection onto the second coordinate, and $(f_x)$ is the unique $\phi $-compatible tuple of horizontal sections such that $f_x (x)=0$.

From Besser's construction of the Coleman integral, we obtain the following characterisation of Coleman integration which will allow us to vary in families, and to define a $p$-adic Betti map.
\begin{corollary}\label{cor:bess2}
Let $V$ be as above, and $x,y\in V(\mathcal{O}_K )$. Let $(\mathcal{V},\nabla )$ be a unipotent connection on $V_K$. Let $P$ denote the frame bundle whose $U$-sections are unipotent isomorphisms $x^* \mathcal{E}\otimes \mathcal{O}_U \simeq \mathcal{E}$. Then the Frobenius invariant path from Theorem \ref{thm:besserpath} defines a Frobenius invariant element $p^\phi $ of $y^* P$. Moreover for any trivialisation $\mathcal{E}\simeq \mathcal{O}_U ^{\oplus n}$ on an open $U$ containing $x$, the matrix entries of $p^{\phi }$ are abstract Coleman functions for the isocrystal $(\mathcal{V},\nabla )$.
\end{corollary}

\subsection{The Coleman integral in families}\label{subsec:families}
Let $\pi :X\to S$ be a smooth proper morphism of varieties over $\mathcal{O}_{K}$, and $\overline{s}\in S(k)$ such that $X_{\overline{s}}$ is smooth. Then we can consider the category of unipotent isocrystals on $X_{\overline{s}}$ relative to the frame $(X_{\overline{s}} ,X_{\overline{s}} ,]X_{\overline{s}} [_{X})$. Let $s\in ]\overline{s}[(K)$.
\begin{lemma}\label{lemma:GM}
For any two rigid triples $T_1 =(Y,X,P_1 )$ and $T_2 =(Y,X ,P_2 )$, a morphism of formal schemes $f:P_1 \to P_2 $ which restricts to the identity on $Y$ induces an equivalence of categories $\mathcal{C}(Y,X,P_1 ) \simeq \mathcal{C}(Y,X,P_2 )$, and hence an isomorphism of path torsors $\pi _1 (T_1 ;x,y)\simeq \pi _1 (T_2 ;x,y)$.

In particular, the closed immersion $i_{s}:X_{s }\hookrightarrow X$ induces an equivalence of categories
\[
\iota _{s}^* :\mathcal{C}(X_{\overline{s}} ,X_{\overline{s}},X)\simeq \mathcal{C}(X_{\overline{s}} ,X_{\overline{s}},X_s )
\]
\end{lemma}
\begin{proof}
Recall that Berthelot relates $T_1 $ and $T_2 $ via the rigid triple $T_3 =(Y,X,P_1 \times _{\Z _p }P_2 )$, where $X$ is mapped into $P_1 \times P_2 $ diagonally. Let $\pr _1 $ and $\pr _2 $ denote the coordinate projections from $]Y[_{P_1 \times P_2 }$ to $]Y[_{P_1 }$ and $]Y[_{P_2 }$ respectively. 
Then \cite{Ber96} \cite[Theorem 7.18]{lestum} shows that $\pr _1 ^* $ and $\pr _2 ^*$ induce equivalences of categories $\mathcal{C}(T_i )\simeq \mathcal{C}(T_3 )$ for $i=1,2$. The morphism $(1,f)$ defines a section of $\pr _1$, and hence $(1,f )^*$ must give an equivalence of categories. Since $f=\pr _1 \circ (1,f)$, $f^*$ is the composite of two equivalences.

\end{proof}
Given a rigid triple $T=(Y,X,P)$, we can in this way view $\mathcal{C}(T)$ as an invariant of $Y$, denoted $\mathcal{C}^{\rig }(Y)$. This invariant is functorial in the sense that, given a morphism of rigid triples $f:T_1 \to T_2 $, the induced morphism $\mathcal{C}(Y_1 )\to \mathcal{C}(Y_2 )$ only depends on the morphism $Y_1 \to Y_2 $ (this is essentially a consequence of applying Lemma \ref{lemma:GM} to the graph of $f$). In particular, $\mathcal{C}^{\rig }(Y)$, and hence $\pi _1 ^{\rig }(Y;x_1 ,x_2 )$, carries an action of $\phi $.

Given $X$ and $\overline{s}$ as above, let $T_0 $ denote the rigid triple $(X_{\overline{s}},X_{\overline{s}},X)$. Let $s_1 ,s_2 \in ]\overline{s}[(K)$, and let $T_j$ denote the rigid triple $(X_{\overline{s}},X_{\overline{s}},X_{s_j })$ for $j=1,2$. Let $x\in X_{\overline{s}}(k)$. By Lemma \ref{lemma:GM}, $\iota _{s_j }^*$ induces an isomorphism of fundamental groups
\[
\pi _1 (T_0 ,x) \stackrel{\simeq }{\longrightarrow }\pi _1 (T_j ,x)
\]
or more generally an isomorphism of path torsors
\[
\pi _1 (T_0 ;x,y) \stackrel{\simeq }{\longrightarrow } \pi _1 (T_j ;x,y)
\]
for any two points $x,y\in X_{\overline{s}}(k)$. 
Given two points $s_1 ,s_2 $ in $]s[(K)$, we define the \textit{parallel transport isomorphism}
\begin{equation}\label{eqn:defnG}
G(s_1 ,s_2 ):\pi _1 (T_1 ;x,y )\stackrel{\simeq }{\longrightarrow }\pi _1 (T_2 ;x,y)
\end{equation}
by $G(s_1 ,s_2 ):=\iota _{s_2 *}^{-1}\circ \iota _{s_1 *}$. By Theorem \ref{thm:CLS}, this induces an isomorphism
\[
G(x_1 ,x_2 ):\pi _1 ^{\dR}(X_{s_1 } ,x_1 )\stackrel{\simeq }{\longrightarrow }\pi _1 ^{\dR}(X_{s_2 } ,x_2 )
\]
for any $x_i $ in $]x[$ lying above $s_i $.
Note that the composite of the isomorphism of functors
\[
]x[_{X_1 }\to ]x[_{X}\to ]x[_{X_2 }
\]
defined is simply a relative version of the parallel transport isomorphism defined in \eqref{eqn:parallel_transport}. In particular, if $x_i \in ]x[_{X_i }(K)$, and $\mathcal{V}$ is of the form $(j^{\dagger }\mathcal{O}_{]X[},d-\Lambda )$, then the isomorphism
\[
i_1 ^* \mathcal{V}(]x[_{]X_1 [})\stackrel{\simeq }{\longrightarrow }i_2 ^* \mathcal{V}(]x[_{]X_2 [})
\]
is uniquely determined by sending a section $v\in i_1 ^* \mathcal{V}(]x[_{]X_1 [})^{\nabla =0}$ to the unique horizontal section $w\in i_2 ^* \mathcal{V}(]x[_{]X_2 [})^{\nabla =0}$ satisfying
\[
w(x_2 )=H(x_2)\circ v(x_1 ),
\]
where $H\in \Mat _n (j^{\dagger }\mathcal{O})$ satisfies 
\[
dH=\Lambda \cdot H
\]
and $H(x_1 )=1$.
When applied to the fibre functors $x_i$, it has the simpler form that the isomorphism $x_1 ^* \mathcal{V}\to x_2 ^* \mathcal{V}$ is simply given by $H(x_2 )$.
\begin{lemma}\label{lemma:phi_equivt}
For $s\in S(k)$, $s_1 ,s_2 \in ]s[(K)$, and any $x_i \in X_{s_i } (K)$ lying above $x\in X_s (k)$, $i=1,2$, the parallel transport isomorphism
\[
G(x_1 ,x_2 ):\pi _1 ^{\dR}(X_{s_1 },x_1 )\stackrel{\simeq }{\longrightarrow }\pi _1 ^{\dR}(X_{s_2 } ,x_2 )
\]
is $\phi$ -equivariant.
\end{lemma}
\begin{proof}
By Lemma \ref{lemma:GM}, the isomorphism of fibre functors $x^* \simeq x_0 ^*$ is given, on an isocrystal $\mathcal{V}$ on $(X_s ,X)$, by the maps
\[
(]x[^* (\mathcal{V}))^{\nabla =0} \stackrel{\simeq }{\longrightarrow } x_0 ^* \mathcal{V}
\]
which sends a horizontal section of $\mathcal{V}$ on $]x[$ to its value at the point $x_0 ^*$. Hence, by definition, for $x_1 ,x_2 \in ]x[$, the composite isomorphism
\[
x_0 ^* \mathcal{V} \stackrel{\simeq }{\longrightarrow }(]x[^* (\mathcal{V})^{\nabla =0} \stackrel{\simeq }{\longrightarrow}x_1 ^* \mathcal{V}
\]
is given by parallel transport from $x_0 $ to $x_1$.
\end{proof}

\begin{lemma}\label{lemma:unique_iso}
For any two points $x_1 ,x_2 \in X(K)$, there is a unique $\phi$ -equivariant isomorphism 
\[
\mathcal{E}_n (x_1 )\simeq \mathcal{E}_n (x_2 )
\]
which is unipotent with respect to the filtrations.
\end{lemma}
\begin{proof}
This follows from the fact that $\phi$ acts on $\gr _i \mathcal{E}_n (x)$ with weight $-i$, see \cite{besser}.
\end{proof}

Let $x\in J(k)$, lying above $s\in S(k)$, and let $x_1 ,x_2 $ be $K$-points in $]x[(K)$, lying above $s_1 ,s_2 \in ]s[(K)$. Let 
\[
G(x_1 ,x_2 ):x_1 ^* \mathcal{E}\stackrel{\simeq}{\longrightarrow} x_2 ^* \mathcal{E}
\]
denote the parallel transport isomorphism associated to the connection on $\mathcal{E}$ defined above. Let 
\[
G_{\gr }(x_1 ,x_2 ):x_1 ^* ( \gr ^\bullet  \mathcal{E})\stackrel{\simeq }{\longrightarrow }x_2 ^* (\gr ^\bullet  \mathcal{E})
\]
denote the paralllel transport isomorphism associated to $\gr ^\bullet \mathcal{E}$. For $x\in X(K)$, let
\[
I(x):x^* \gr ^\bullet \mathcal{E} \stackrel{\simeq }{\longrightarrow }\mathcal{E}
\]
denote the unique $\phi $-equivariant unipotent isomorphism. We obtain the following ``$p$-adic Betti'' isomorphism between identifying $p$-adic logarithms on different fibres of $J$.
\begin{proposition}\label{prop:commutes}
The diagram
\[
\begin{tikzcd}
x_1 ^* ( \gr ^\bullet \mathcal{E}) \arrow[r, "{G_{\gr }(x_1 ,x_2 )}"] \arrow[d, "I(x_1 )"] & x_2 ^* (\gr ^\bullet \mathcal{E}) \arrow[d, "I(x_2 )"] \\
x_1 ^* \mathcal{E} \arrow[r, "{G(x_1 ,x_2 )}"]                                 & x_2 ^* \mathcal{E}                         
\end{tikzcd}
\]
commutes.
\end{proposition}
\begin{proof}
Note that $G(x_1 ,x_2 )\circ I(x_1 )\circ G_{\gr }(x_1 ,x_2 )$ is unipotent. Hence, by Lemma \ref{lemma:unique_iso}, it is enough to prove that it is $\phi $-equivariant. Hence it is enough to prove that $G(x_1 ,x_2 )$ and $G_{\gr }(x_1,x_2 )$ are $\phi $-equivariant, which follows from Lemma \ref{lemma:phi_equivt}.
We deduce that a $\phi $-equivariant isomorphism
\[
\mathcal{E} (x_1 )\simeq \mathcal{E}(x_2 )
\]
is given by $G(x_1 ,x_2 )$, and a $\phi $-equivariant isomorphism
\[
\gr ^\bullet \mathcal{E}(x_1 )\simeq \gr ^\bullet \mathcal{E}(x_2 )
\]
is given by $ G_{\gr }(x_1 ,x_2 )$.
\end{proof}
\section{Proof of Theorem \ref{thm:main1}: good reduction case}\label{sec:goodreduction}
\subsection{Proof of Proposition \ref{prop:padicZP}}
Recall the set-up is that we have a geometrically simple abelian variety $A$ over a $p$-adic field $K$, and we consider codimension $\geq r$ special subvarieties of $A^n$ for some $n>0$. Such special subvarieties are necessarily images of homomorphisms of the form $i:A^m \to A^n$ for $m<n$, where $m\leq n-r/g$. On $A^n$ we have the universal connection given by $(\mathcal{E}(A^n ),\nabla )$ together with $e_1 :K\mapsto e^* \mathcal{E}(A^n )$. Let $V:=H^1 _{\dR}(A/K)^* $, and let $P_n$ denote the principal $V^n$ bundle whose $U$-sections are unipotent isomorphisms
\[
e^* \mathcal{E}(A^n )\otimes \mathcal{O}_U \stackrel{\simeq }{\longrightarrow }\mathcal{E}(A^n )|_U ,
\] endowed with the connection induced from that on $\mathcal{E}(A^n )$. We can quotient $P$ by the action of $(H^1 (A,\mathcal{O})^* )^n$, giving a $\Lie (A^n /K)$-torsor. Forgetting the connection, this torsor is trivial (equivalently, the extension class of the bundle $\mathcal{E}(A^n )/(H^1 (A,\mathcal{O})^* )^{\oplus n}\otimes \mathcal{O}_A $ in $H^1 (A^n ,H^0 (A^n ,\Omega _{A^n |K})^* )$ is trivial). The section $e_1 \in e^* \mathcal{E}(A^n )$ gives a choice of splitting, and in this way, we obtain a surjection of bundles on $A^n$
\[
P_n \to \Lie (A/K)^n \times A^n .
\]
Let $p$ denote the composite of this map with the projection to $\Lie (A/K)^n$. Let $x\in A^n (K)$, with mod $\pi _K$ reduction $\overline{x}$. Note that any path $p$ in $\pi _1 ^{\dR}(A^n ,e,x)$ defines a point in $x^* P$, by evaluating the path $p$ at $\mathcal{E}(A^n)$. In particular, Coleman integration, in the form of the Frobenius invariant path of Theorem \ref{thm:besserpath}, defines an element $z\in x^* P$. The following lemma follows from the definition of Coleman integration in terms of parallel transport given in section \ref{subsec:coleman}.
\begin{lemma}\label{lemma:simplest_commutative}
Let $\alpha :]x[(K)\to P_n (K)$ denote the leaf of $P_n$ through $(x,z)$. Then we have a commutative diagram
\[
\begin{tikzcd}
]x[(K)\cap U(K) \arrow[d, "\alpha"] \arrow[rd, "\log"] &               \\
P_n (K) \arrow[r, "p"]                     & \Lie (A/K)^n
\end{tikzcd}
\]
\end{lemma}
\begin{proof}
To apply the notion of abstract Coleman functions recalled earlier, it is convenient to fix an arbitrary nonzero linear functional $\ell :\Lie (A/K)^n \to K$.
Given a point $u\in ]x[(K)\cap U(K)$, $\alpha (y)$ is the unique Frobenius invariant unipotent isomorphism
\[
e^* \mathcal{E}(A^n )\otimes \mathcal{O}_U \simeq y^* \mathcal{E}(A^n ).
\]
Hence $\ell \circ p\circ \alpha $ is the restriction to $]x[$ of the unique abstract Coleman function on $\mathcal{E}(A^n )$ whose value at the identity is zero, whose associated bundle surjection $\mathcal{E}\to \mathcal{O}$ is defined as the composite
\[
\mathcal{E}\to \mathcal{E}/(H^1 (A,\mathcal{O})^* )^n \otimes \mathcal{O}\to \Lie (A)^n \otimes \mathcal{O}\to \mathcal{O}
\]
where the second map is induced by the splitting associated to $e_1 $, and the third is induced by $\ell $. By the characterisation of $\log $ in terms of Coleman integration, this implies $\ell \circ \log =\ell \circ p\circ \alpha $ for all $\ell $.
\end{proof}
For $\ell <\min \{ g,n\}$, let $\mathbb{D}_{\ell } \overline{U}^n \subset \Lie (A/K)^n$ denote the subscheme of rank $\leq \ell $ $n$-tuples of vectors in $\Lie (A/K)$. Then $\mathbb{D}_{\ell } \overline{U}^n $ is a finite union of codimension $\geq (n-\ell )(g-\ell )$ subvarieties of $\Lie (A/K)^n$. Hence Lemma \ref{lemma:simplest_commutative} implies that, in a small disc around $z$, the image of rank $\leq \ell $ points of $A^n$ in the leaf of the connection on $P_n $ through $z$ lies in a finite union of codimension $\geq (n- \ell )(g-\ell )$ subvarieties of $P_n$. Hence, by Theorem \ref{thm:BCFN}, $]x[(K)\cap U(K)\cap A^{n,[g(n-\ell )]}$ is not Zariski dense in $A^n$. 

\begin{remark}
In fact, in this case, the Zariski non-density could also be deduced from Ax's theorem \cite{AxII}, as in this case Theorem \ref{thm:BCFN} specialises to Ax's theorem for the abelian variety $A^n $.
\end{remark}
\subsection{The variation of the $p$-adic logarithm in families}\label{subsec:variation}
In this subsection, we relate the results of section \ref{sec:families} to those of section \ref{sec:AxSchan}. Let $\pi :X\to S$ be a family of smooth projective curves over $\mathcal{O}_K$, where $K$ is a $p$-adic field. Let $\pi _J :J\to S$ denote the Jacobian of $X/S$. Let $(\mathcal{E},e_1 )$ denote the universal connection on $\pi _J :J_K \to S_K$ relative to the identity section $e$. Pick $x\in J(\mathcal{O}_K )$ lying above $s\in S(\mathcal{O}_K )$, and let $\widetilde{P}$ be the associated frame bundle whose $U$-sections are isomorphisms
\[
x^* \mathcal{E}\otimes \mathcal{O}_U \stackrel{\simeq }{\longrightarrow }\mathcal{E}|_U .
\]
Let $G<\GL (x^* \mathcal{E})$ denote the monodromy group of $\mathcal{E}$ (we follow the convention that the monodromy group is the connected component of the identity of the Tannakian fundamental group of the Tannakian subcategory of the category of flat connections generated by $\mathcal{E}$). Let $P$ denote a descent of $\widetilde{P}$ to a $G$-bundle. Let $G_0 $ denote the reductive quotient of $G$, and let $P_0$ denote the pushout of $P$ along $G\to G_0$. So $G$ is an extension of $G_0 $ by a vector group over $K$ of dimension $2g$, $G_0$ is (isomorphic to) the monodromy group of the Gauss--Manin connection, and $P_0$ can be identified with the frame bundle for $R^1 \pi _* \Omega _{J_K |S_K }$. Let $\mathcal{U}$ denote the sheaf whose values on $U$ are the set of unipotent isomorphisms
\[
\mathcal{E}|_U \stackrel{\simeq }{\longrightarrow }\gr ^\bullet \mathcal{E}|_U .
\]
Via Corollary \ref{cor:bess2}, Coleman integration together with the trivialisation of $e^* \mathcal{E}$ defined by $e_1$, defines a unipotent isomorphism
\[
\alpha_0 : \gr ^\bullet (x^* \mathcal{E})\simeq x^* \mathcal{E}.
\]
Then we have a map
\[
p_1 :P\to \mathcal{U}
\]
defined locally by sending an isomorphism $\rho :\mathcal{E}|_U \stackrel{\simeq }{\longrightarrow }x^* \mathcal{E} \otimes \mathcal{O}_U $ to
\[
(\gr ^\bullet \rho )^{-1}\circ (\alpha _0 \otimes 1_U )\otimes \rho :\mathcal{E}|_U \stackrel{\simeq }{\longrightarrow } \gr ^\bullet \mathcal{E}|_U .
\]
$\mathcal{U}$ is naturally an $(R^1 \pi _* \Omega _{J_K |S_K })^*$-torsor over $J$. Pushing out along $(R^1 \pi _* \Omega _{J_K |S_K })^* \to \pi  ^* \Lie (J_K /S_K )$ we obtain a torsor which may be trivialised via its trivialisation along the identity section coming from the isomorphism
\begin{equation}\label{eqn:trivialise}
e^* \mathcal{E}\simeq e^* (\gr ^\bullet \mathcal{E}).
\end{equation}
Hence we obtain a map
\[
p_2 :\mathcal{U}\to \Lie (J/S)_K .
\]
Composing with $p_1$, we obtain a morphism of $S_K$-schemes
\begin{equation}\label{eqn:whatisp}
p:=p_2 \circ p_1 :P\to \Lie (J_K /S_K ).
\end{equation}
\begin{proposition}\label{prop:GM_2_log}
Let $\pi :X\to S$ be a smooth family of curves over $\mathcal{O}_K$.
Let $x\in J(\mathcal{O}_K )$ reducing to $\overline{x}\in J(k)$ and mapping to $s\in S(K)$. Let $z\in x^* P(K)$ be a point corresponding (via Corollary \ref{cor:bess2}) to $p$-adic integrals $\omega \mapsto \int _{o}^x \omega $. Let
\[
\alpha :]\overline{x}[(K)\to P(K)
\]
denote the leaf of $P$ through $z$. Then the diagram
\[
\begin{tikzcd}
]\overline{x}[(K) \arrow[d, "\alpha"] \arrow[rd, "\log"] &               \\
P(K) \arrow[r, "p"]                     & \Lie (J/S)(K)
\end{tikzcd}
\]
commutes.
\end{proposition}
\begin{proof}
Let $\mathcal{U}(K)$ denote the $K$-points of the total space of $\mathcal{U}$. We define a map of sets
\[
\beta :J(\mathcal{O}_K )\to \mathcal{U}(K)
\]
by sending $x\in J(\mathcal{O}_K )$ lying above $s\in S(\mathcal{O}_K )$ to the trivialisation of $x^* \mathcal{E}(X)$ induced by the Coleman integrals $\omega \mapsto \int _{o}^x \omega $. Then it is enough to prove that the diagram
\[
\begin{tikzcd}
]\overline{x}[(K) \arrow[d, "\alpha"] \arrow[rd, "\beta"] &               \\
\mathcal{U}(K) \arrow[r, "p_2 "]                     & \Lie (J/S)(K)
\end{tikzcd}
\]
commutes, since $\log =p_1 \circ \beta $. This follows from Proposition \ref{prop:commutes}.
\end{proof}

\subsection{Construction of the principal bundle $P$ in Theorem \ref{thm:main1}}
To apply Theorem \ref{thm:BCFN} to deduce Theorem \ref{thm:main1} part (1), we first need to define the relevant principal bundle. Let $\pi :X\to S$ and $J$ be as in Theorem \ref{thm:main1}, part (1). Let $P_1 ,\ldots ,P_{d_0 }\in J(S)$ be sections generating a finite index subgroup of the subgroup $\Gamma $. Let $s\in U_S (\mathcal{O}_{K })$ be a point where $s^* P_1 ,\ldots ,s^* P_{d_0 }$ generate a rank $d_0$ subgroup of $J_s (K )$ (by assumption, the set of such $S$ is the intersection of $U_S (\mathcal{O}_{K })$ with the $K$-points of a Zariski open subset of $S$). To prove that $X^n _S (\mathcal{O}_K )_{\Gamma -\rk \leq r}$ is not Zariski dense, it will enough to prove it for each residue disk, hence we fix $\overline{x}=(\overline{x}_1 ,\ldots ,\overline{x}_n ;\overline{s})\in X^n _S (k)$.

Let $(x_1 ,\ldots ,x_n ;s)$ in $X^n _S (K)$. If the rank of the subgroup of $J_s (K )$ generated by $\Alb _s (x_1 ),\ldots ,\Alb _s (x_n ),s^* P_1 ,\ldots ,s^* P_{d_0 }$ is $\leq d+d_0$, then the $K $-dimension of the subspace of $\Lie (J_s )_{K }$ generated by ${\log }_s \circ \AJ (x_1 ),\ldots ,\log _s \circ \AJ(x_n )$ and $\log _s (s^* P_1 ),\ldots ,\log _s (s^* P_{d_0 })$ is $\leq d+d_0 $, where $\log _s $ denotes the $p$-adic logarithm from $J_s $ to $\Lie (J_s )_K$. This motivates the following definition.
\begin{definition}
The $d$th degeneracy locus of $\Lie (J_s )^{n+d_0 }$, denoted $\mathbb{D} _{d } (\Lie (J_s )^{n+d_0 })$, is defined to the subscheme of $\Lie (J_s )^{n+d_0 }$ of points of rank $\leq d$. Equivalently, it is the vanishing locus of the map
\[
\wedge ^{d+1}(\mathcal{O}^{n+d_0 }_{\Lie (J_s )^{n+d _0 }})\to \wedge ^{d+1}\underline{\Lie (J_s )}
\]
obtained from the tautological map 
\[
\mathcal{O}^{n+d_0 }_{\Lie (J_s )^{n+d _0 }}\to \underline{\Lie (J_s )}
\]
of vector bundles on $\Lie (J_s )^{n+d_0 }$, where $\underline{\Lie (J_s )}$ denotes the trivial vector bundle with fibre $\Lie (J_s )$.
\end{definition}

This subscheme has codimension $(n+d_0 -r)(g-r)$ in $\Lie (J_s )^{n+d_0 }$ (this follows from the corresponding fact for determinantal varieties of $n\times m$ matrices of rank $r$, see e.g. \cite[Proposition 12.2]{harris:GTM}). This construction can be generalised to families as follows.
\begin{definition}
Given a rank $m$ vector bundle $\mathcal{V}$ over a scheme $S$, and positive integers $d\leq r$ with $d\leq m$, we define the $d$th degeneracy locus of $\mathcal{V}^{\oplus r}$ to be the subscheme of the total space of $\mathcal{V}^{\oplus r}$ consisting of points of rank $\leq d$. Equivalently, if $\tau :\mathcal{V}^{\oplus r}\to S$ denotes the structure map from the total space of the vector bundle $\mathcal{V}^{\oplus r}$ to $S$, then the $d$-degeneracy locus can be defined as the vanishing locus of the map
\[
\wedge ^{d+1}\mathcal{O}_{\mathcal{V}^{\oplus r}}^{\oplus r}\to \wedge ^{d+1}\tau ^* \mathcal{V}
\]
obtained by taking the $(d+1)$th wedge power of the tautological map
\[
\mathcal{O}_{\mathcal{V}^{\oplus r}}^{\oplus r}\to \tau ^* \mathcal{V}.
\]
\end{definition}

Fix $x=(x_1 ,\ldots ,x_n ;s)\in ]\overline{x}[(K)$. Let $V:=H^1 _{\dR}(X_s /K)$, and let $G_0 <\GSp _{2g}(V)$ be the monodromy group with respect to $\pi $. By Deligne's semi-simplicity theorem \cite[Theorem 4.3.6]{hodgeii}, $G_0$ is a sparse group. Recall that $J$ denotes the Jacobian of $X$ over $S$, and that we assume that we have $r$ generically independent sections $s_1 ,\ldots ,s_{d_0 }$ of $J$ over $S$, generating a subgroup $\Gamma $. 

Define $j^n _{\Gamma }$ to be the map
\[
j^n _{\Gamma }:X^n _S (\mathcal{O}_K )\to \Lie (J/S)^{n+r}(K)
\]
sending $(x_1 ,\ldots ,x_n ;s)$ to $(\log _s (x_1 ),\ldots ,\log _s (x_n ),\log _s (s_1 (s)),\ldots ,\log _s (s_d (s));s)$. Then, for each $s\in S(\mathcal{O}_K )$, $j^n _{\Gamma }|_{X^n _s (K)}$ maps $X^n _s (\mathcal{O}_K )_{\Gamma -\rk \leq r}$ into $\mathbb{D}_r (\Lie (J_s )^{n+d})$. 

We deduce the following global property of $j^n _{\Gamma }(X^n _S (\mathcal{O}_K ))$.
\begin{lemma}
The set $X^n _S (\mathcal{O}_K )_{\Gamma -\rk \leq r}$ maps into a finite union of codimension $\geq (n+d-r)(g-r)$ subvarieties of $\Lie (J/S)^{n+d} _S (\mathcal{O}_K )$ under the map $j^n _{\Gamma }$.
\end{lemma}

For each section $s_i$, we have a frame bundle $P(s_i )$ on $S$ which is a $G_1$-torsor, where $G_1 $ is an extension of $G_0$ by the vector group $H^1 _{\dR}(X_s /K)^*$. Corresponding to $\Gamma $ and $n$ we can form a frame bundle $P^n _{\Gamma }$ as follows. For each $i$ let $\mathcal{E}(s_i )$ be the extension of connections on $S$ corresponding to the section $s_i$, and let $\mathcal{E}(i)$ be the pullback of $\mathcal{E}(J)$ along $\pr _i$. Then for each $i$ we have a frame bundle $P_{s_i }$ or $P(i)$. We define $P^n _{\Gamma }\subset \prod P_{s_i }\times \prod P(i)$ to be the subvariety of unipotent isomorphisms which act diagonally on
\[
\bigoplus _i (\gr ^\bullet  \mathcal{E}(i) \oplus \gr ^\bullet \mathcal{E}(i)) \simeq \mathcal{O}_{X^n _{S_K} }^{\oplus (n+r)} \oplus ((R^1 \pi _* \Omega _{X_K|S_K})^* )^{\oplus (n+r)}.
\]
This is a torsor under an extension of $G_0$ by $(H^1 _{\dR}(X_s /K)^* )^{\oplus (n+r)}$. Then Proposition \ref{prop:GM_2_log} implies that we have a commutative diagram.
\[
\begin{tikzcd}
]\overline{x}[(K) \arrow[d, "\alpha"] \arrow[rd, "\log"] &               \\
P^n _{\Gamma }(K) \arrow[r, "p"]                     & \Lie (J/S)(K)^{n+r}
\end{tikzcd}
\]

To complete the proof of case (1) of Theorem \ref{thm:main1}, by Lemma \ref{lemma:analytic2formal} we reduce to the proving the following statement. Let $x\in X^n _S (\mathcal{O}_K )_{\rk r}$, let $z:=j^n _{\Gamma }(x)$, and let $\widehat{\mathbb{D}}_z$ denote the formal completion of $\mathbb{D}_r (\Lie (J/S) ^{n+r}_{S_K} )$ at $z$, viewed as a formal subscheme of the formal completion of $\Lie (J/S)^{n+r}_{S_K}$ at $z$. Then it is enough to prove that $(j^n _{\Gamma })^{-1}(\widehat{\mathbb{D}}_z )$ is not Zariski dense in $X^n _S $. By Proposition \ref{prop:GM_2_log}, this is a consequence of Theorem \ref{thm:BCFN}, since $(p ^{n+r})^{-1}(\mathbb{D}_r (\Lie (J/S)^{n+r}_{S_K} ))$ is a finite union of codimension $\geq \dim X^n _S $ subvarieties of $P^n _\Gamma $.
\section{Filtered schemes in the Chabauty--Coleman--Kim method}\label{sec:filtered}
In this section and the next, $X$ will denote a smooth projective curve over $\Q $, and $Y\subset X$ a nonempty open subscheme over $\Q $ with complement $Z$, such that $2g(X)+\# Z(\overline{\Q })>3$. Let $\mathcal{X}$ denote a regular model for $X$ over $\mathbb{Z}$, and let $\mathcal{Z}$ denote the closure of $Z$ in $\mathcal{X}$. For $S\subset \Spec (\mathbb{Z})$, let $Y(\mathbb{Z}_S )$ denote the set of $\mathbb{Z}_S $-sections of $\mathcal{X}$ which do not reduce to $\mathcal{Z}_v$ modulo any prime $v$ not in $S$.
\subsection{Review of the Chabauty--Coleman--Kim method}
We first review the Chabauty--Coleman--Kim method, as developed in \cite{Kim}, \cite{kim3}, \cite{kim_tamagawa}. 
Let $p$ be a prime of good reduction for $X$, let $T_0$ be a finite set of primes not containing $p$, and let $T\supset T_0$ a finite set of primes containing $p$ and all primes of bad reduction. 
Let $b$ be a $T_0$-integral point of $Y$, and let $U_n (X,x)$ denote the maximal $n$-unipotent quotient of the $\Q _p$-unipotent \'etale fundamental group of $Y_{\overline{\Q }}$ at $b$. Then we obtain a map
\[
j_{n,b}^{\et} :Y(\mathbb{Z}_{T_0 })\to H^1 _{f,T_0} (G_{\Q ,T},U_n (X,b)).
\]
which sits in a commutative diagram
\[
\begin{tikzcd}
Y(\mathbb{Z}_{T_0 }) \arrow[r, "{j_{n,b}^{\et}}"] \arrow[d] & H^1 _{f,T_0}(G_{\Q ,T},U_n (Y,b)) \arrow[d, "\loc _p"] \\
Y(\mathbb{Z}_p ) \arrow[r, "{j_{n,b}}"]        & U_n ^{\dR}(Y_{\Q _p },b)/F^0   ,
\end{tikzcd}
\]
defined as follows.
\begin{enumerate}
\item $H^1 _{f,T_0}(G_{\Q ,T},U_n (Y,b))$ is the set of nonabelian cohomology classes in $H^1 (G_{\Q ,T},U_n (Y,b))$ which are unramified outside $T_0 \cup \{p \}$ and crystalline at $p$ in the sense of \cite{Kim}.
\item $U_n ^{\dR}(Y_{\Q _p  },b)$ is the maximal $n$-unipotent quotient of the de Rham fundamental group of $Y_{\Q _p }$ with basepoint $b$. This has a filtration $F^i$ and $F^0$ is a subgroup.
\item The set $H^1 _f (G_{\Q _p },U_n (Y,b))$ of crystalline cohomology classes has the structure of the set of $\Q _p $-points of a scheme, so that the natural maps between crystalline cohomology classes of different subquotients of $U_n (Y,b)$ are all algebraic. There is an isomorphism of schemes
\[
H^1 _f (G_{\Q _p },U_n (Y,b)\simeq U_n ^{\dR}(Y_{\Q _p },b)/F^0 .
\]
\item The horizontal maps are defined by sending a point $x$ to the class of $\pi _1 ^{\et }(Y_{\overline{\Q }};b,x)$ in $H^1 (G_{\Q } ,U_n (Y,b))$ under the map $H^1 (G_{\Q } ,\pi _1 (Y_{\overline{\Q }},b))\to H^1 (G_{\Q } ,U_n (Y_{\overline{\Q }},b))$.
\end{enumerate}
In fact, by imposing local conditions at primes away from $p$, we can define a closed subscheme $\Sel _{T_0 }(U_n )\subset H^1 (G_{\Q ,T},U_n (Y,b))$ as follows. For $w\in T-T_0 \cup \{ p \}$, let $\overline{j_{n,b}^w (Y(\mathbb{Z}_w ))}$ denote the Zariski closure of the set $j_{n,b}^w (Y(\mathbb{Z}_w ))$. 
We define the closed subscheme $\Sel _{T_0 }(U_n )\subset H^1 (G_{\Q ,T},U_n (Y,b))$ by 
\[
\Sel _{T_0 }(U_n ):= H^1 (G_{\Q ,T},U_n (Y,b)) \times  _{\prod _{w\in T-\{p\}}H^1 (G_{\Q _w },U_n (Y,b))}\prod _{w\in T-T_0 \cup \{p\}}\overline{j_{n,b}^w (Y(\mathbb{Z}_w ))}.
\]

Then we have a commutative diagram
\[
\begin{tikzcd}
Y(\mathbb{Z}_{T_0 }) \arrow[r, "{j ^{\et }_{n,b}}"] \arrow[d] & \Sel _{T_0 }(U_n ) \arrow[d, "\loc _p"] \\
Y(\mathbb{Z}_p ) \arrow[r, "{j_{n,b}}"]        & U_n ^{\dR}(Y_{\Q _p },b)/F^0   ,
\end{tikzcd}
\]
where the objects and maps involved have the following properties.
\begin{enumerate}
\item $\Sel _{T_0 }(U_n )$ is a scheme of finite type over $\Q _p$, and $U_n ^{\dR}(Y_{\Q _p })/F^0 $ is a variety over $\Q _p$, in such a way that $\loc _p $ is a morphism of schemes over $\Q _p$.
\item The map $j_{n,b}$ is locally analytic, in the sense that, on each residue disk $D$ of $Y$, $j_{n,b}|_D$ is obtained from a rigid analytic function $D\to (U_n ^{\dR}(Y_{\Q _p },b)/F^0 )^{\an }$.
\end{enumerate}

The map $j_{n,b}$ admits the following description, for any smooth curve $Y$ over $\mathcal{O}_{K}$ for $K$ a $p$-adic field with residue field $k$ (see \cite{kim3} or \cite{Hadian} for more details). Let $\mathcal{E}_n ^{\dR}(b)$ be the $n$-unipotent universal connection on $Y_K$, and $z\in Y(\mathcal{O}_K )$. Let $\widetilde{P}_n ^{\dR}$ be the bundle whose $U$-sections are $n$-unipotent isomorphisms
\[
(b^* \mathcal{E}_n )\otimes \mathcal{O}_U \stackrel{\simeq }{\longrightarrow }\mathcal{E}_n ^{\dR}|_U ,
\]
with the connection induced from the connection on $\mathcal{E}_n ^{\dR}$. Then $\widetilde{P}_n ^{\dR}$ descends to a principal $U_n ^{\dR}$-bundle, which we denote $P_n ^{\dR}$. In fact, this principal bundle is exactly the unipotent de Rham fundamental groupoid on  $Y_K \times Y_K$, restricted to $Y_K \times \{b\}$ and pushed out along 
\[
\pi _1 ^{\dR}(Y_K ,b)\to U_n ^{\dR}(b).
\]
In particular, the fibre of $P_n ^{\dR}$ at $z\in Y(\mathcal{O}_K )$ consists of unipotent isomorphisms $b^* \mathcal{E}_n ^{\dR}\simeq z^* \mathcal{E}_n ^{\dR}$, and the unique $\phi $-invariant unipotent isomorphism from Lemma \ref{lemma:unique_iso} defines an element $p^\phi (z)$ of $z^* P_n ^{\dR} $. The bundle $P_n ^{\dR}$ carries a Hodge filtration, compatible with that on $U_n ^{\dR}(b)$, which can for example be defined as the restriction of the Hodge filtration on $(b^* \mathcal{E}_n )^* \otimes \mathcal{E}_n ^{\dR}$. A choice of section $p^H $ in $F^0 P_n ^{\dR}$ defines a trivialisation of $P_n ^{\dR}$, giving a dominant morphism
\[
p:P_n ^{\dR}\to U_n ^{\dR}/F^0 
\]
independent of choices. Let $\alpha :]x[(K )\to P_n ^{\dR}(K )$ be the leaf of $P_n $ through $p^\phi (z)$.
Let $x\in Y(k )$ be the mod $\pi _K$ reduction of $z$. By \cite{kim3}, the map $j_n$ is given by sending $y$ to $p^\phi (y)\cdot (y^* p^H)^{-1}$. Hence the diagram
\begin{equation}\label{eqn:fundamental_triangle}
\begin{tikzcd}
]x[(K ) \arrow[d, "\alpha"] \arrow[rd, "j_n"] &               \\
P_{n} (K ) \arrow[r, "p"]                     & U_{n}^{\dR}/F^0 (K)
\end{tikzcd}
\end{equation}
commutes.

Now suppose that $X=Y$. Then the maps $j_{n,b}^\ell$ at primes $\ell \neq p$ of bad reduction for $X$ convey a finite amount of information. More precisely, let $\ell $ be a prime of bad reduction, and $K_v $ a finite extension of $\Q _{\ell }$ over which $X$ acquires a semistable model $\mathcal{X}/\mathcal{O}_{K_v }$. Then for $x,y \in Y(\Q _{\ell })$, by \cite[Proposition 3.8.1]{bettsdogra}, if the mod $\pi _v$ reduction of $x$ lies on the same irreducible component of $\mathcal{X}_{k_v }$ as that of $y$, then $j_{n,b}^\ell (x)=j_{n,b} (y)$. In particular, we deduce that the map $j_{n,b}^{\ell }$ factors through the map to the irreducible components of a minimal regular model over $\Z _{\ell }$. 

With this in mind, we introduce the notation of a reduction type of a point, following \cite{betts21}. For a prime $v$ of bad reduction over a finite extension $K$ of $\Q$, let $\Gamma _{v}$ denote the set of irreducible components of the special fibre of a minimal regular model of $X$ over $\mathcal{O}_{K_v}$. We define the reduction type of a $K_v$-point to be the element of $\Gamma _v$ to which it reduces. For a $K$-point $x$ we define its reduction type to be $(\gamma _v )_v $, where $v$ ranges over primes of bad reduction and $\gamma _v \in \Gamma _v $ is the reduction type of $x$ in $X(K_v )$.

Hence we deduce that $\Sel (U_n )$ is a finite disjoint union of at most $\prod _v \# \Gamma _v$ connected components. In particular, if we restrict to the subset $Y(\mathcal{O}_{K,T_0 })_b \subset Y(\mathcal{O}_{K,T_0 })$ of points $x$ for which the reductions mod $v$ of $x$ and $b$ lie on the same irreducible component of the special fibre of a minimal regular model of $X$ at $\mathcal{O}_v$ for all primes $v$ of bad reduction, then the image of $Y(\mathcal{O}_{K,T_0 })$ in $\Sel _{T_0 }(U_n )$ is contained inside $H^1 _{f,T_0 }(G_{\Q },U_n )$, the scheme of torsors which are crystalline $p$ and unramified outside $T_0 \cup \{ p\}$. We hence need a bound on the numbers of such irreducible components which is uniform across the fibres of our isotrivial family of curves. This motivates the following definition.

\begin{definition}\label{defn:cv}
Given a smooth projective curve $X$ over $\Q $, let $H:=\Aut _{\overline{\Q }}(X)$, and for a prime $v$ define $c_v$ to be the number of irreducible components of a minimal regular model of $X$ over the ring of integers of the minimal extension $K_v$ of $\Q _v $ which trivialises all $H'$-torsors for group schemes $H' <H$ defined over $\Q $. That, is $K_v$ has the property that for all such $H'$, the restriction map $H^1 (G_{\Q _v },H')\to H^1 (G_{K_v },H')$ is zero.
\end{definition}

Hence to prove part (2) of Theorem \ref{thm:main2}, it will be enough to prove that for any reduction type $\gamma $, $X^n _S (\Q )_{\gamma ,\rk r}$ is not Zariski dense in $X^n _S $ whenever
\begin{equation}\label{eqn:thm2reduction}
n>2\cdot r^{\frac{\log (r)+\log \log (r)}{\log (g+\sqrt{g^2 -1})}+4}\cdot \log (r) (g+\sqrt{g^2 -1})^3 .
\end{equation}

For general curves, to prove results using Chabauty--Coleman--Kim (and not just Chabauty--Coleman) one needs to assume results in Galois cohomology which are implied, for example, by the Bloch--Kato conjectures. The only part of the Bloch--Kato conjectures we will need in this paper is the following very special case.
\begin{conjecture}[Bloch--Kato \cite{BK}]\label{BK}
Let $X$ be a smooth projective geometrically irreducible curve of genus $g$ over $\Q $. Then, for all $n>1$,
\[
H^1 _f (G_{\Q },H^n _{\et }(X_{\overline{\Q }},\Q _p (1)))=0.
\]
\end{conjecture}

This conjecture has the following consequence.
\begin{proposition}\label{prop:BKconsequence}
Let $\gr _i U_n (X)$ denote the $i$th graded piece of $U_n (X)$. Then, assuming conjecture \ref{BK}, for all $i>1$,
\[
\dim H^1 _f (G_{\Q _p },\gr _i U_n (X))-\dim H^1 _f (G_{\Q },\gr _i U_n (X))=\dim U_n (X)^c +\dim U_n (X)(-1)^{G_{\Q }},
\]
where $c\in \Gal (\overline{\Q }|\Q )$ denotes complex conjugation with respect to embedding $\overline{\Q }\hookrightarrow \mathbb{C}$.
\end{proposition}
\begin{proof}
This is a special case of \cite[Remark 1.2.4]{FPR}.
\end{proof}
In the case $X=\mathbb{P}^1 -\{ 0,1,\infty \}$, by Soul\'e \cite{soule} we have the simpler estimate
\begin{equation}\label{eqn:uniteqn_bds}
\dim H^1 _f (G_{\Q },\gr _i U_n (x))=\begin{cases} \dim H^1 _f (G_{\Q _p },\gr _i U_n ), i \in 1+2\cdot \mathbb{Z}_{>0} \\ 0, i\in 2\cdot \mathbb{Z}_{>0} \end{cases}.
\end{equation}

\subsection{Chabauty--Coleman--Kim for non-integral points}\label{sec:nonint}
For applications to case (1) of Theorem 2, we shall need some features of the theory when the rational points in question are allowed to be non-integral at $p$ (but $p$ is still a prime of good reduction for $X$). In this case, we have a map 
\begin{equation}\label{eqn:deRham}
Y(\Q _p )\to H^1 _g (G_{\Q _p },U_n )
\end{equation}
from points to torsors which are merely de Rham at $p$, rather than crystalline.

By work of Betts \cite{betts2017} and Betts--Litt \cite[Theorem 4.1]{betts2019}, we have an isomorphism
\[
H^1 _g (G_{\Q _p },U_n )\simeq D_{\dR}(U_n )/F^0 \times \mathcal{V}(U_n),
\]
We shall denote the $D_{\dR}(U_n )/F^0 $ component of \eqref{eqn:deRham} by $j_n$. 

The general definition of $\mathcal{V}(U_n )$ given in \cite{betts2019} is somewhat involved. Specifically, one takes 
\[
\mathcal{V}(U_n ):=\gr ^\bullet \Lie (U_n )(-1)^{W_{\Q _p ,K}}
\]
Here $\gr ^\bullet \Lie (U_n )(-1)$ is viewed as a Weil--Deligne representation, an invariance is with respect to the action of the Weil group of $\Q _p $ and the corresponding monodromy $\SL_2 $ action. Finally, the grading is with respect to the weight filtration (see loc. cit. for the conventions for this).

Our case of interest is when $U_n$ is the maximal $n$-unipotent quotient of the $\Q _p $-unipotent completion of the \'etale fundamental group of $\mathbb{P}^1 -\{ 0,1,\infty \}$. In this case the associated graded with respect to the weight filtration is the same (as a Galois representation) as the associated graded with respect to the central series filtration (in general, these will only differ when $U_1$ is not pure, which is when $X$ is a projective curve of genus  greater than zero with at least two points removed). Hence $\gr ^\bullet \Lie (U_n )$ is simply a direct sum of copies of $\Q _p (i)$ for $i\leq n$, and hence the monodromy operator acts trivially, and 
\[
\gr ^\bullet \Lie (U_n )(-1)^{W_{\Q _p ,K}}\simeq U_1 (-1)\simeq \Q _p ^{\oplus 2}.
\]
Hence we have an isomorphism
\begin{equation}\label{bettslitt}
H^1 _g (G_{\Q _p },U_n )\simeq D_{\dR}(U_n )\times \Q _p ^{\oplus 2}.
\end{equation} 
Since, in this case, $D_{\dR}(U_n )^{\phi =1}=\{ 1\}$, the image of $z$ in $D_{\dR}(U_n )$ is simply given by the unique $\phi $-invariant path from $t$ to $z$.
Let $z$ reduce to zero mod $p$. Then, by \cite[Proposition 4.16]{betts2019}, the image of $z$ in the RHS of \eqref{bettslitt} is given by
\begin{equation}\label{eqn:furusho}
(G_0 (z),(v_p (z),v_p (1-z)))
\end{equation}
where $G_0\in \Q _p \langle \! \langle A,B \rangle \! \rangle$ is Furusho's generating series for $p$-adic multiple polylogarithm \cite[Theorem 2.3]{furusho}, and $A$ and $B$ are certain canonical generators of the pro-universal enveloping algebra of $\Lie (\pi _1 ^{\dR}(X,t))$. In particular, it follows from \eqref{eqn:fundamental_triangle} that we have a commutative diagram of formal schemes
\begin{equation}\label{eqn:fundamental_triangle_log}
\begin{tikzcd}
\widehat{X}_z  \arrow[d, "\alpha"] \arrow[rd, "j_n"] &               \\
\widehat{P}_{n,G_0 (z)} \arrow[r, "p"]                     & \widehat{U}_{n,j_n (z)}^{\dR}.
\end{tikzcd}
\end{equation}
The entries of $G_0$ are Coleman functions (in the usual sense) when restricted to $]\mathbb{P}^1 -\{ 0,1,\infty \} (\F _p )[$, but on the residue disks at $0,1$ and $\infty $ they are Coleman functions with log singularities. However, these log singularities only occur in depth one, in a certain sense, which in the case of $p$-adic polylogarithms expresses the fact that the $p$-adic polylogarith $\li _n (z)$ is a Coleman function on the residue disk at zero. Hence, for applications to non-density, we project onto a (scheme-theoretic) quotient of $U_n $.

The part of $j_n$ which is logarithmic varies for each of the residue disks $]0[, ]1[$ and $]\infty [$. However, to prove results about the Zariski closure of $\cup _{\rk S=s}X(S)^n$, we can ignore the residue disks at $1$ and $\infty $, and only consider points on the tube of $\mathbb{A}^1 -\{ 1\} _{\F _p }$. Indeed, $\cup _{\rk S=s}X(S)^n$ is closed under the action of $S_3 ^n$, and hence so is its Zariski closure. It follows that the Zariski closure of $\cup _{\rk S=s}X(S)^n $ will be equal to the $S_3 ^n $ orbits of $\cup _{\rk S=s}X(S)^n \cap ]\mathbb{A}^1 -\{ 1\}[ ^n $, and hence if the latter is not Zariski dense in $(\mathbb{P}^1 )^n$, neither is the former.
\begin{lemma}\label{lemma:Vn}
For all $n>0$, there is a variety $V_n$ and a surjection $\tau _n :U_n ^{\dR}\to V_n$ whose fibres have dimension one, such that $\tau _n \circ j_n |_{](\mathbb{A}^1 -\{ 1\})[(\Q _p )}$ is given by Coleman functions.
\end{lemma}
\begin{proof}
We define the projection as follows.
With respect to the coordinates $A$ and $B$, this amounts to sending $f \in \Q _p \langle A , B \rangle $ to $f\cdot \exp (-\lambda \cdot A)$, where the $A$ coefficient of $f$ is equal to $\lambda \in \Q _p $. By Furusho's explicit formula for a $p$-adic solution to the KZ equation \cite[Theorem 3.15]{furushoI}, the entries of $\tau _n \circ j_n $ are Coleman functions.
\end{proof}
\subsection{Filtered schemes}
The Chabauty--Coleman--Kim method crucially uses the fact that a certain morphism of nonabelian cohomology sets (with values in a unipotent group) is in fact a morphism of schemes, extending the linear morphisms of vector spaces on graded pieces. In \cite{betts21}, the exact algebraic nature of the map is made somewhat more explicit.
\begin{definition}
Let $R$ be a $K$-algebra. A filtered affine scheme over $R$ is an affine scheme $\Spec (A)$ of finite type over $R$, with an increasing, exhaustive filtration $(A_i )_{i\geq 0}$ by $R$-modules, such that
\begin{enumerate}
\item $A_0 =R$.
\item $A_i \cdot A_j \subset A_{i+j}$.
\end{enumerate}

A morphism of filtered affine schemes is a morphism of affine schemes respecting the filtration on coordinate rings.
\end{definition}

If $R$ and $S$ are filtered affine schemes, with coordinate rings $A$ and $B$, and $B$ is generated by $B_n$, then morphisms of filtered affine schemes $f:R\to S$ are uniquely determined by $f|_{B_n }$.

\begin{example}[The filtered scheme associated to a filtered free $R$-module]\label{example:filteredfree}
Let $V$ be a free $R$-module space with a separating exhaustive increasing filtration in non-positive degreees
\[
V=F^{-1} V \supset F^{-2}V \supset \ldots \supset 0.
\]
Then we obtain a filtration on $\Sym ^i V^* $ for all $i$, and hence obtain a filtration on the coordinate ring $\mathcal{O}(V)\simeq \oplus _{i\geq 0}\Sym ^i V^*$ giving $V$ the structure of a filtered scheme.
\end{example}
Note that a morphism of filtered affine schemes between two filtered affine schemes associated to filtered free $R$-modules does not necessarily come from a morphism of $R$-modules (even when the filtration is trivial). Indeed, if $V$ and $W$ are filtered free $R$-modules, and we choose trivialisations $V\simeq \gr ^\bullet V, W\simeq \gr ^\bullet W$ of the filtrations, then the morphism of schemes from $V$ to $W$ will be dual to a collection of morphisms of $R$-modules
\begin{equation}\label{eqn:dual}
\phi _i \in \Hom ((\gr ^i W )^* ,\oplus _{0\leq j\leq i}(\Sym ^j V^* )_{-i}).
\end{equation}
For example, when $V$ and $W$ are free $R$-modules with filtration concentrated in degree $-1$, a morphism of filtered affine schemes will just mean an affine linear map.

For non-negative integers $d_1 ,\ldots ,d_n$, we define $\mathbb{A}(d_1 ,\ldots ,d_n )_R$ to be the filtered affine scheme associated to the graded $R$-modules $\oplus _{i=1}^n R^{d_i }$, where the $i$th summand $R^{d_i }$ is in degree $-i$.

We define a filtered affine space to be a filtered affine scheme isomorphic (as a filtered affine scheme) to the filtered affine scheme associated to a filtered free $R$-module.

\begin{definition}
Let $V$ be a filtered affine space.
A $(d_1 ,\ldots ,d_n )$-framed affine subspace is a morphism of filtered schemes
\[
f:\mathbb{A}(d_1 ,\ldots ,d_n )_R \to V.
\]
\end{definition}

Let $S_i \subset \mathbb{Z}_{\geq 0}^n$ denote the set of tuples $(m_j )$ such that $\sum _{j=1}^n m_j \cdot j=i$.

\begin{lemma}\label{lemma:universal}
Let $W$ be a filtered affine space over a ring $R$, whose graded pieces have dimensions $e_i$ for $1\leq i\leq n$. Let $d_1 ,\ldots ,d_n$ be non-negative integers. Let $\mathcal{F}$ denote the functor sending an $R$-algebra $S$ to the set of morphisms of filtered schemes
\[
\mathbb{A}(d_1 ,\ldots ,d_n )_S \to W\otimes _R S .
\]
Then $\mathcal{F}$ is represented by an affine space $Z_{d_1 ,\ldots ,d_n }V$ of dimension
\[
J(d_1 ,\ldots ,d_n ;e_1 ,\ldots ,e_n ):=\sum _{i=1}^n e_i \cdot \sum _{0\leq j\leq i}D_j ,
\] 
where $D_0 :=1$ and 
\[
D_i :=\sum _{(m_j )\in S_i}\prod _{j=1}^n \binom{d_j +m_j -1}{m_j }
\]
for $i>0$. Moreover, $Z_{d_1 ,\ldots ,d_n }V$ is stable under base change, i.e. for any $R$-algebra $S$, we have an isomorphism of $S$-schemes
\[
(Z_{d_1 ,\ldots ,d_n }V)\times _R S \simeq Z_{d_1 ,\ldots ,d_n }(V\otimes _R S).
\]
\end{lemma}
\begin{proof}
Without loss of generality $V=\mathbb{A}(e_1 ,\ldots ,e_n )$. Then \eqref{eqn:dual} allows us to identify the $S$-points of $\mathcal{F}$ with 
\[
\prod _{i=1}^n \Hom _S (S^{\oplus e_i } ,\oplus _{0\leq j\leq i}(\Sym ^j V^* )_{-i}\otimes _R S).
\]
Hence we see that $\mathcal{F}$ is stable under base change, and can be identified with an affine space of dimension
\[
\sum _{i=1}^n e_i \cdot \sum _{0\leq j\leq i}\dim _R (\Sym ^j V^* )_{-i}.
\]
Recall that $(\Sym ^j V^* )_{-i}$ is dual to the subspace of $\Sym ^j V$ of degree $\leq i$. Via our choice of grading, we can identify the degree $k$ part of $\oplus _{0\leq j\leq i}\Sym ^j V)$ with
\[
\oplus _{(m_\ell )\in S_k}\Sym ^{m_{\ell }}V[\ell ],
\]
which implies the first part of the lemma.
\end{proof}
\begin{lemma}\label{lemma:non_density}
Let $R$ be a $K$-algebra, and $V$ a filtered affine space over $R$. Let $e_i$ denote the dimension of $\gr _i V$. Let $(d_1 ,\ldots ,d_n )$ be positive integers such that $\sum e_i >\sum d_i$. 
If
\[
N>\frac{J(d_1 ,\ldots ,d_n ;e_1 ,\ldots ,e_n )}{\sum _{i=1}^n (e_i -d_i ) },
\]
then there is a proper closed subscheme $\mathbb{D} _{d_1 ,\ldots ,d_n }V^N _R \subset V^N _R$ such that, for any closed subscheme $A\subset \Spec (R)$, and any morphism $f:W\to V_A$ from a $(d_1 ,\ldots ,d_n )$-filtered affine space over $A$ to $V_A$, the image of the $N$-fold product of $f$ is contained in $\mathbb{D} _{d_1 ,\ldots ,d_n }V^N _R \times _R A$.
\end{lemma}
\begin{proof}
Let $Z:=Z_{d_1 \ldots ,d_n }V$. Let $g:\mathbb{A}(d_1 ,\ldots ,d_n )_{Z}\to R_{Z}$ be the universal $(d_1 ,\ldots ,d_n )$-dimensional filtered subspace of $R$. Let 
\[
g^N :(\mathbb{A}(d_1 ,\ldots ,d_n ) _{Z})^N _Z \to V^N _Z 
\]
be the $N$-fold fibre product of $g$ over $Z$. We define $\mathbb{D} _{d_1 ,\ldots ,d_n }V^N _R$ to be the image of the composite morphism of $R$-schemes
\[
(\mathbb{A}(d_1 ,\ldots ,d_n ) _{Z})^N _Z \stackrel{g^N}{\longrightarrow} V^N _Z \to V^N _R .
\]

By Lemma \ref{lemma:universal}, we have
\[
(\mathbb{D} _{d_1 ,\ldots ,d_n}V^N _R ) \times _R A \simeq \mathbb{D} _{d_1 ,\ldots ,d_n}(V\times _R A)^N _A ,
\]
hence to prove the lemma we may take $A=\Spec (R)$. Choosing a frame, we may give $f$ the structure of a framed filtered subspace of $V$. Hence $f$ comes from an $R$-point of $Z$, and the map $f^N$ factors through $\mathbb{D} _{d_1 ,\ldots ,d_n }V^N _R$. 
\end{proof}

We now state the relevant properties of the filtered scheme structures on Selmer varieties and unipotent Albanese varieties
The following result is proved by Betts in \cite{betts21}. An alternative proof is also given in forthcoming work of the author and Carl Wang-Erickson.
\begin{theorem}[\cite{betts21}, Lemmas 3.2.5, 3.2.6 and 3.2.7]\label{thm:betts}
\begin{enumerate}
\item The schemes $H^1 _f (G_{\Q },U_n (x))$ have filtered structures, in such a way that the exact sequences
\[
1\to H^1 _f (G_{\Q },Z(U_n (x)))\to H^1 _f (G_{\Q },U_n (x))\to H^1 _f (G_{\Q },U_{n-1}(x))
\]
are strict with respect to the filtration, and such that $H^1 _f (G_{\Q },Z(U_n (x)))$ is given the filtered structure corresponding to the filtration
\[
H^1 _f (G_{\Q },Z(U_n (x)))=W_n \supset W_{n+1}=0
\]
on the underlying vector space.
\item The commutative diagram
\[
\begin{tikzcd}
1 \arrow[r] & H^1 _f (G_{\Q },Z(U_n (x))) \arrow[r] \arrow[d] & H^1 _f (G_{\Q },U_n (x)) \arrow[r] \arrow[d] & H^1 _f (G_{\Q },U_{n-1}(x)) \arrow[d] &   \\
1 \arrow[r] & Z (U_n ^{\dR}(x))/F^0 \arrow[r]           & U_n ^{\dR}(x)/F^0 \arrow[r]           & U_{n-1}^{\dR}(x)/F^0 \arrow[r] & 1
\end{tikzcd}
\]
is a diagram in filtered schemes.

\end{enumerate}
\end{theorem}

\section{Proof of Theorem \ref{thm:main2}}\label{sec:thm2}
To prove Theorem \ref{thm:main2}, we use Theorem \ref{thm:BCFN} in the following settings. In case (1), we take as our variety $X^n =(\mathbb{P}^1 -\{ 0,1,\infty \} )^n$, and our principal bundle is the $U_m ^{\dR}(X^n )$-bundle of frames (with notation as in section \ref{sec:nonint}). We have a commutative diagram
\[
\begin{tikzcd}
]x[_{(\mathbb{P}^1 _{\mathbb{Z}_p })^n }\cap X^n (\Q _p) \arrow[d, "\alpha"] \arrow[rd, "\tau _m \circ j_m"] &               \\
P_{n} (\Q _p ) \arrow[r, "\tau _n \circ p"]                     & V_m ^{\dR}(\Q _p)^n
\end{tikzcd}
\]
for any $x\in ](\mathbb{A}^1 -1)^n [$, where $V_m$ is the variety from Lemma \ref{lemma:Vn}. Recall from Lemma \ref{lemma:Vn} that the map $\tau _m \circ j_m$ is analytic, and the map $\tau _n $ (and hence $\tau _n \circ p$) is algebraic. Hence by Theorem 4, our goal is to show that $X(S)^n $ is contained in subvariety of $V_m ^{\dR}(\Q _p )^n $ of codimension greater than $n$.

To do this, we will need a modification of the previous notion of Selmer scheme to allow for the fact that we work with $S$-units where $S$ is an arbitrary rank $r$ subgroup of $\Q ^\times $. For such an $S$, let $T_0 \subset \Spec (\mathbb{Z})$ be the set of primes $\ell$ for which $\val _\ell (z)\neq 0$ for some $z\in S$. Let $v_{T_0 }$ denote the map
\[
\Q ^\times \to \Q _p ^{T_0 }
\]
sending $x$ to $(\val _{\ell }(x))_{x\in T_0 }$. Then the image of $v_{T_0 }$ spans a vector space of dimension $r$, which we denote by $V_S$. 
Define $\Sel _{S}(U_n )$ to be the subscheme of cohomology classes whose image in $\prod _{v\in T_0 }H^1 _{g|e}(G_{\Q _v},U_n )\simeq (\Q _p ^{T_0 })^2 $ is contained in $V_S \times V_S$. Then the dimension of $\Sel _S (U_1 )$ is $2r$.
Then for any $S$, we have
\[
j_m (X(S)^n ) \subset \loc _p \Sel _S (U_m )^n ,
\]
where $\loc _p $ denotes the composite map
\[
\Sel (U_m )^n \to H^1 _g (G_{\Q _p },U_{m,n})\to U_{m,n}^{\dR}(\Q _p ),
\]
and the second map is the projection from \eqref{bettslitt}.

In case (2), we pass to a finite cover $X'\to S'$ over which the family becomes trivial, and a finite extension $K$ of $\Q _p$ for which the $K $ points of $X'$ surject onto the $\Q _p$-points of $X$. The extension $K$ is as in Lemma \ref{lemma:fin_extn}. Explicitly, if $H:=\Aut (X)$, then we can take $K$ to be the minimal extension of $\Q _p$ trivialising all $H'$-torsors for $H'<H$. We take our group $G$ to be the $m$-unipotent fundamental group $U_m (X^n )$, for $m$ such that the Bloch--Kato conjectures imply finiteness of $C(\Q _p )_m$ for any genus $g$ curve $C$ of Mordell--Weil rank $r$. Recall that it is enough to prove the bound in \eqref{eqn:thm2reduction} for the set $X^n _S (\Q )_{\gamma ,\rk r}$, where $\gamma $ is a fixed reduction type. We may suppose that $x$ is a point of $X(\Q )$ with reduction type $\gamma $, and take $x$ as our base point.

\begin{lemma}\label{lemma:BKfiltered}
Conjecture \ref{BK} implies that $H^1 _f (G_{\Q },U_n (x))$ is a filtered affine space.
\end{lemma}
\begin{proof}
This can be seen by induction on $n$. By Poitou--Tate duality, conjecture \ref{BK} implies that the map
\begin{equation}\label{eqn:jannsen}
H^2 (G_{\Q ,T},Z( U_n (x)))\to \oplus _{v\in T}H^2 (G_{\Q _v },Z(U_n (x)))
\end{equation}
is injective. We have an exact sequence of pointed schemes
\[
1\to H^1 (G_{\Q ,T},Z (U_n (x)))\to H^1 (G_{\Q ,T},U_n (x)) \to  H^1 (G_{\Q ,T},U_{n-1} (x))\to H^2 (G_{\Q ,T},Z(U_n (x))).
\]
When restricted to $H^1 _f (G_{\Q },U_{n-1}(x))$, conjecture \ref{BK} hence implies the boundary map is zero, and hence that $H^1 _f (G_{\Q ,T},U_n (x))$ surjects onto $H^1 _f (G_{\Q ,T},U_{n-1} (x))$. The exact sequence then gives $H^1 _f (G_{\Q },U_n (x))$ the structure of an $H^1 _f (G_{\Q },Z(U_n (x)))$-torsor over the filtered affine space $H^1 _f (G_{\Q },U_{n-1}(x))$. Such a torsor is trivialisable, giving the filtered affine space structure on $H^1 _f (G_{\Q },U_n (x))$.
\end{proof}

By \eqref{eqn:fundamental_triangle}, to prove non--Zariski density of $X^n _S (\Q )_{\Gamma -\rk r}$ for $n$ as in \eqref{eqn:thm2reduction}, it is enough to prove that for some $m>0$, its image in $(U_m ^{\dR}(X)/F^0 )^n$ is contained in a subvariety of codimension greater than or equal to $n$. Hence, by Lemma \ref{lemma:BKfiltered}, we may apply Lemma \ref{lemma:non_density} to reduce to estimating the dimensions of the graded pieces of $U_n ^{\dR}(X)/F^0$ and $H^1 _f (G_{\Q },U_n )$.

By the de Rham comparison theorem, to estimate the dimension of $U_n ^{\dR}(X)$ it is enough to estimate the dimension of $U_n ^{\Be }(X_{\mathbb{C}})$. This dimension is equal to
\[
\sum _{i=1}^n e_i
\]
where $e_i$ is dimension of the $i$th graded piece of the central series filtration of $U_n ^{\dR}(X)$. By Witt's formula we have 
\begin{equation}\label{eqn:labute1}
\sum _{k|n}k\cdot e_k =2^n
\end{equation}
when $X=\mathbb{P}^1 -\{ 0,1,\infty \}$, and by Labute \cite[Proposition 4]{labute2} we have the generating series identity
\[
1-2gt+t^2 =\prod _{n\geq 1}(1-t^n )^{e_n }
\]
when $X$ is a smooth projective curve of genus $g>1$, which can be rewritten as
\begin{equation}\label{eqn:labute2}
\sum _{k|n}k\cdot e_k =(g+\sqrt{g^2 -1})^n +(g-\sqrt{g^2 -1})^n .
\end{equation}

We obtain the estimates
\begin{equation}\label{eqn:labute3}
e_n \leq 2^n /n
\end{equation}
and (for $n\geq 2$)
\begin{equation}\label{eqn:labute4}
e_n \leq (g+\sqrt{g^2 -1})^n /n
\end{equation} 
respectively. We will also need lower bounds for $e_n$. First suppose $X=\mathbb{P}^1 -\{0,1,\infty \}$. From \eqref{eqn:labute3}, we obtain
\[
\sum _{k|n,k<n}k\cdot e_k \leq \sum _{k|n,k<n}2^k ,
\]
which is less than $n^{1/2}2^{n/2+1}$, using the crude estimate that there are at most $2\sqrt{n}$ summands $k$, each at most $n/2$. We deduce from \eqref{eqn:labute1} the lower bound
\begin{equation}\label{eqn:labute5}
e_n \geq 2^n /n  -2^{n/2+1}/n^{1/2}.
\end{equation}
Now suppose $X$ is projective of genus $g$. We define
\[
\alpha _{\pm }:=g\pm \sqrt{g^2 -1}.
\]
Then by the same argument as above, we deduce
\[
e_n \geq \frac{1}{n}(\alpha _+ ^n +\alpha _- ^n )-2(\alpha _+ ^{n/2}+\alpha _- ^{n/2})/n^{1/2}.
\]
for all $n\geq 2$.
\subsection{Estimates for the unit equation}
We now turn to the problem of estimating the dimension of the Selmer variety $\Sel (U_n (X))$. We define $d_i :=\dim H^1 _f (G_{\Q },\gr ^i U_n )$ for $n\geq i$. 
To ease notation, we define $r:=2s$.
\begin{lemma}\label{lemma:N_estimate1}
Let $X=\mathbb{P}^1 -\{ 0,1,\infty \}$. If $\rk S=s$, then $X(\mathbb{Z}_{p })_{S,N}$ is finite whenever 
\[
N>1+\frac{\log (r)+\log (\log (r)+\log (2))}{\log (2)}
\]
\end{lemma}
\begin{proof}
To obtain finiteness of $X(\mathbb{Z}_p )_{S,n}$, it is sufficient to prove 
\begin{equation}\label{eqn:the_dimension_inequality}
\sum _{i=1}^n (e_i -d_i )>0.
\end{equation}
In the case $X=\mathbb{P}^1 -\{ 0,1,\infty \}$, we have
\[
\sum _{i=1}^n (e_i -d_i )=2-r+\sum _{i\leq n/2}e_{2i}.
\]
Then, using \eqref{eqn:labute5}, we have
\[
\sum _{i=1}^n (e_i -d_i )\geq 2-r+\sum _{i\leq n/2}4^i /2i  -2^{i+1/2}/i^{1/2}
\]
Hence, for $n$ to satisfy \eqref{eqn:the_dimension_inequality}, it is sufficient that
\[
r<2^n /n .
\]
If $x>1$, $y>e$ and $\log (x) >1+\log \log (y)/\log (y)$, then $x^m /m$ is greater than $y$ for 
\begin{equation}\label{eqn:loglog}
m>\frac{\log (y)+\log \log (y)}{\log (x)}.
\end{equation}
We apply this when $x=4, y=2r$ and $m=n/2$ to deduce that, if $n>1+\frac{\log (r)+\log (\log (r)+\log (2))}{\log (2)}$, then $r<2^n /n$, from which the Lemma follows.
\end{proof}
\subsection{Estimates for higher genus curves}
As in the previous section we define $d_n :=\dim H^1 _f (G_{\Q },\gr _n U_n )$.
\begin{lemma}\label{lemma:BKimplies}
Let $X$ be a smooth projective curve of genus $g>1$ over $\Q $. Then the Bloch--Kato conjectures predict that $X(\Q _p )_N$ is finite whenever 
\[
N>\frac{\log (r)+\log (2)+\log (\log (r)+\log 2)}{\log (\alpha _+ ) },
\]
where $\alpha _+ :=g+\sqrt{g^2 -1}$.
\end{lemma}

We first estimate $d_n$ under the hypothesis of the Bloch--Kato conjectures. By Proposition \ref{prop:BKconsequence}, this implies
\[
\dim H^1 _f (G_{\Q _p },U_n )-\dim H^1 _f (G_{\Q },\gr _n U_n )\geq \dim U_n ^{c=1}
\]
where $c\in \Gal (\overline{\Q }|\Q )$ is complex conjugation relative to any embedding $\overline{\Q }\hookrightarrow \mathbb{C}$.

To estimate $e_n -d_n$ assuming the Bloch--Kato conjectures, we use the following `categorification' of Labute's theorem, due to Filip \cite{filip}. Let $V:=U_1 $ and $V_n :=\gr _n U_n$.
\begin{theorem}[Filip]\label{thm:filip}
The character $\chi _n$ of $\GSp (V)$ acting on $V_n $ satisfies
\[
\sum _{k|n}\frac{1}{k}\chi ^{(k)}_{n/k} =\frac{1}{n}\left[ \left( \frac{\chi _V +\sqrt{\chi _V ^2 -4}}{2}\right) ^n +  \left( \frac{\chi _V -\sqrt{\chi _V ^2 -4}}{2}\right) ^n \right]
\]
where $\chi ^{(m)}(g):=\chi (g^m )$ and $r$ is the Mordell--Weil rank of $J$.
\end{theorem}
Let $c$ be a representative of complex conjugation. Since $\dim V_n ^c =\frac{1}{2}(\dim V_n +\chi _n (c))$, to estimate $\dim V_n ^c$ it's enough to estimate $\chi _n (c)$. We see inductively from Theorem \ref{thm:filip} that $\dim V_n ^c =\frac{1}{2}\dim V_n $ when $n$ is odd, because $\chi _V (c)=0$. In the case when $n$ is even, we have
\begin{equation}\label{eqn:filip1}
\sum _{k|n}\frac{1}{k}\chi ^{(k)}_{n/k} =0,
\end{equation}
hence
\[
\chi _n =-n \sum _{k<n ,k|n}\frac{k}{n}\chi _{k}^{(n/k)}.
\]
Note that $\chi _k ^{(m)}(c)$ equals $\chi _k (c)$ when $m$ is odd, and equals $e_k$ when $m$ is even. 
Inductively, we see that 
\[
|\chi _n (c)| \leq \alpha _+ ^{n/2+1}+\alpha _- ^{n/2+1}.
\]
Hence for $n>1$ we have, assuming Conjecture \ref{BK},
\begin{equation}\label{eqn:envsdn}
e_n -d_n \geq \frac{\alpha _+ ^n +\alpha _- ^n }{2n}-\frac{1}{2}(\alpha _+ ^{n/2+1}+\alpha _- ^{n/2+1})-\frac{1}{n^{1/2}}(\alpha _+ ^{n/2}+\alpha _- ^{n/2}).
\end{equation}
Hence the codimension of $H^1 _f (G_{\Q },U_n )$ in $H^1 _f (G_{\Q _p },U_n )$, assuming Bloch--Kato, is at least
\[
\frac{\alpha _+ ^n }{2n}-r.
\]
Using \eqref{eqn:loglog} with $y=2r$ and $x=\alpha _+$ completes the proof of Lemma \ref{lemma:BKimplies}.
\subsection{Estimating dimension of the space of filtered subspaces: the unit equation}
We now use Lemma \ref{lemma:non_density} to complete the proof of case (1) of Theorem \ref{thm:main2}. 
Recall that this means that we need to estimate the size of 
\[
J(d_1 ,\ldots ,d_n ;e_1 ,\ldots ,e_n ):=\sum _{i=1}^n e_i \cdot \sum _{0\leq j\leq i}D_j ,
\] 
where 
\[
D_j =\sum _{(m_k )\in S_j}\prod _{k=1}^n \binom{d_k +m_k -1}{m_k }
\]
for $j>1$ (and $D_0 :=1$), and $S_j :=\{ (m_k )\in \mathbb{Z}_{\geq 0}^n :\sum m_k =j \}$.
\begin{lemma}\label{lemma:Jestimate}
Suppose that $d_1 =r ,d_i \leq \alpha ^i $ for $i>1$, and $e_i \leq \beta ^i /i $ for $i\geq 1$. Then
\[
J(d_1 ,\ldots ,d_n ;e_1 ,\ldots ,e_n ) \leq \begin{cases}\frac{\beta ^{n+1}r^{n+2}}{(r-2e\alpha )(\beta -1)(r-1)}+\frac{(2e\alpha )^{n+2} \beta ^{n+1} }{(\beta -1)(r-2e\alpha )(2e\alpha \beta -1)}, r >2e\alpha \\
\frac{\beta ^{n+1}(2e\alpha )^{n+2}}{(2e\alpha -r)(\beta -1)(2e\alpha -1)}+\frac{\beta ^{n+1} r^{n+2}}{(\beta -1)(2e\alpha -r)(\beta r -1)}.
, 2e\alpha >r . \end{cases}
\]

\end{lemma}
\begin{proof}
We have 
\begin{align*}
J(d_1 ,\ldots ,d_n ;e_1 ,\ldots ,e_n ) & \leq \sum _{i=1} ^n \beta ^i \sum _{0\leq j\leq i} \sum _{(m_k )\in S_j}\prod _{k=1}^n \binom{d_k +m_k -1}{m_k } \\
& \leq  \sum _{i=1} ^n \beta ^i \sum _{0\leq j\leq i} \sum _{(m_k ) \in S_j }\prod _{k=1}^n d_k ^{m_k } \\
&= \sum _{0\leq j\leq n}\frac{\beta ^{n+1}-\beta ^j }{\beta -1} \sum _{(m_k ) \in S_j }\prod _{k=1}^n d_k ^{m_k } \\
\end{align*}
Define $S^{(2)}_i :=\{ (m_j :2\leq j\leq i ):\sum m_j =i \}$.Then we obtain
\begin{align*}
J(d_1 ,\ldots ,d_n ;e_1 ,\ldots ,e_n ) & \leq \sum _{0\leq j\leq n}\frac{\beta ^{n+1}-\beta ^j }{\beta -1} \sum _{0\leq k\leq j}r^k \sum _{(m_\ell )\in S^{(2)}_{j-k}}\prod _{\ell=2}^n d_\ell ^{m_\ell } \\
& \leq \sum _{0\leq j\leq n}\frac{\beta ^{n+1}-\beta ^j }{\beta -1} \sum _{0\leq k\leq j}r^k \cdot \alpha ^{j-k}\cdot \# S ^{(2)}_{j-k}.
\end{align*}
using the estimate $d_k \leq \alpha ^k $ for $k>1$. We can bound $\# S^{(2)}_m$ by $(2e)^m$, giving a bound 
\[
J(d_1 ,\ldots ,d_n ;e_1 ,\ldots ,e_n )  \leq  \sum _{0\leq j\leq n}\frac{\beta ^{n+1}-\beta ^j }{\beta -1} \sum _{0\leq k\leq j}r^k \cdot \alpha ^{j-k}(2e)^{j-k} 
\]
from which the result follows.
\end{proof}
To complete the proof of case (1) of Theorem \ref{thm:main1}, we apply Lemma \ref{lemma:Jestimate} when $\alpha =\beta =2$, giving 
\[
J(d_1 ,\ldots ,d_n ;e_1 ,\ldots ,e_n ) \leq (2r)^{n+2}.
\]
By Lemma \ref{lemma:N_estimate1}, we apply this with
\[
n=2+\frac{\log (r)+\log \log (r)+\log (2)/\log (r)}{\log (2)},
\]
giving 
\[
J(d_1 ,\ldots ,d_n ;e_1 ,\ldots ,e_n )< 59\cdot r^{\frac{\log (r)+\log \log (r)}{\log (2)}+5}\cdot \log (r).
\]
\subsection{Proof of Theorem \ref{thm:main2} case (2)}

The calculation in this case is very similar. Recall that we have 
\[ n\leq 1+\frac{\log (r)+\log (2)+\log \log (r)+\log (2)/\log (r)}{\log (\alpha _+ ) }
\]
From  \eqref{eqn:labute4} and  \eqref{eqn:envsdn} we obtain the estimates $d_n \leq \alpha _+ ^n$ and $e_n \leq \alpha _+ ^n$ for $n>1$. As above, we deduce
\[
J(d_1 ,\ldots ,d_n ;e_1 ,\ldots ,e_n )< 3r^{\frac{\log (r)+\log \log (r)+4\log (\alpha _+ )+\log (2)}{\log (\alpha _+ )}+4}\cdot \log (r) \alpha _+ ^3
\]

\section{Proof of Theorem \ref{thm:main1}: bad reduction case}
In this section (which is independent of the previous two sections) we complete the proof of Theorem \ref{thm:main1}. Recall that $X$ and $S$ are smooth projective varieties over the ring of integers $\mathcal{O}_K$ of a $p$-adic field $K$, and $\pi :X\to S$ is a family of stable curves, smooth outside a strict normal crossing divisor $Z\subset S$. By appealing to section \ref{sec:goodreduction}, we may reduce to the case of \textit{singular} residue discs, i.e. points $s\in S(k)$ such that $X_s$ is a singular stable curve. 

An added layer of complication when $x\in X(K)$ reduces to a singular point of $X_s$ is the relation between Coleman integrals, in the sense of \cite{CDS}, and abelian integrals in the sense of \cite{zarhin}. The latter are, by definition, homomorphisms from $p$-adic points of the Jacobian to the Lie algebra of the Jacobian, but in general (more precisely, when the dual graph of $X_s$ has nonzero first homology) the former are not. This issue also arises in previous work of Stoll and Katz--Rabinoff--Zureick-Brown on uniformity in the Chabauty--Coleman method. As in this previous work, we get around this issue by restricting the abelian logarithm to differentials whose residues locally vanish, giving a quotient of the usual abelian logarithm which we term the Stoll logarithm. As this subspace is in a certain sense a function of the Gauss--Manin connection, we are able to apply the $\nabla $-special Ax--Schanuel theorem, using a slightly different notion of `degeneracy locus'.

In our proof of case (2) of Theorem 1, we will restrict to $n$-tuples of points which lie on a common vertex or edge of the dual graph. Given Zariski non-density results for such $n$-tuples, we can then pigeonhole principle our way to a global bound, using the following lemma on the number of vertices and edges of a stable graph, a proof of which can be found for instance in \cite[Lemma 4.14]{KRZB} (here by a stable graph of genus $g$ we just mean the dual graph of a stable curve of genus $g$).
\begin{lemma}\label{lemma:stablegraph}
Let $\Gamma $ be a stable graph of genus $g$. Then $\# V(\Gamma )\leq 2g-2,$ and $\# E(\Gamma )\leq 3g-3$.
\end{lemma}
Let $\pi :X\to S$ be a stable, generically smooth family of curves over $\mathcal{O}_K$, $s\in S(k)$ a point such that $X_s$ is singular with dual graph $\Gamma $. For $v$ a vertex of the dual graph of $X_s$, let $X_v \subset X_s $ denote the corresponding irreducible component and $U_v \subset X_s$ the corresponding affine open obtained by deleting singular points. Let $]e^n [_{X^n _S }$ denote the tube of the point $(e,\ldots ,e)$ in $X^n _S $. By Lemma \ref{lemma:stablegraph}, the following proposition implies case (2) of Theorem \ref{thm:main1}. 

\begin{proposition}\label{prop:thm1case2}
\begin{enumerate}
\item For any $v\in V(\Gamma )$, $X^n _S (K )_{\rk r} \cap ]U_v ^n   [ _{X^n _S}$ is not Zariski dense in $X^n _S$ whenever
\[
n>\frac{(r-d +1)(g-r)+s}{g-r-1}
\]
\item For any $e\in E(\Gamma )$, $X^n _S (K )_{\rk r} \cap ]e^n [_{X^n _S} $ is not Zariski dense in $X^n _S $ whenever
\[
n>\frac{(r-d +1)(g-1-r)+s}{g-r-2}
\]
\end{enumerate}
\end{proposition}

\subsection{Proof of Proposition \ref{prop:thm1case2} part (1): integration on smooth disks}
We keep the notation from the previous subsection. The following proposition gives a second, more elementary proof of the relation between Coleman (or, in this context, Berkovich--Coleman) integrals and the Gauss--Manin connection, which also applies in a bad reduction setting. The Berkovich--Coleman integral, as defined in \cite{CDS}, \cite{berkovich}, \cite{KRZB}, \cite{katz2022p}, is a function
\[
H^0 (X,\Omega _{X|K})\times \pi _1 (X)\to K
\]
where $\pi _1 (X)$ is the fundamental groupoid of the Berkovich space. That is, it is the groupoid whose objects are $K$-points of $X$, and for which the set of morphisms from $x$ to $y$ is the set $\pi _1 (X;x,y)$ of homotopy classes of paths in the underlying topological space of the Berkovich space $X$.
\begin{proposition}\label{prop:BC_functorial}[\cite{CDS}, \cite{berkovich}, Theorem 9.1.1]
The Berkovich--Coleman integral is functorial on morphisms of affinoids. That is, given a rigid analytic morphism
\[
f:X\to Y
\]
between smooth affinoid varieties over $K$, $x,y \in X(K)$, a path $p$ from $x$ to $y$, and $\omega \in H^0 (Y,\Omega )$, we have
\[
\int ^x _y f^* \omega =\int ^{f(x)}_{f(y)}\omega .
\]
\end{proposition}

This property gives a relation between Coleman integrals on nearby fibres of a family of curves, by the Monsky--Washnitzer lifting theorems. First, we recall the following strengthening of Monsky--Washnitzer lifting, due to Coleman and Iovita. Let $K$ be a finite extension of $\Q _p$ with ring of integers $\mathcal{O}$. Let $R_k :=\mathcal{O}\langle T_1 ,\ldots ,T_n \rangle $, and let $R_{k,n}$ be the weak completion of the $R_k$-algebra $R_k [X_1 ,\ldots ,X_n ]$. A semi-dagger algebra over $R_k$ is a quotient of $R_{k,n}$ by a finitely generated ideal (for some $n$).

\begin{theorem}[\cite{CI10}, Theorem 3.7]\label{thm:CI}
Suppose $R_1 ,R_2 , R_3 $ and $R_4$ are flat semi-dagger algebras over $R_k$ and we have a commutative diagram
\[
\begin{tikzcd}
R_1 \arrow[r] \arrow[d, "f"] & R_3 \arrow[d, "g"] \\
R_2 \arrow[r]                & R_4               
\end{tikzcd}
\]
with $f$ formally smooth and $g$ surjective. Then any $s:R_2 /\mathfrak{p}R_2 \to R_3 /\mathfrak{p}R_3 $ commuting with the mod $\mathfrak{p}$ reduction of the above reduction lifts to $\widetilde{s}:R_2 \to R_3$ commuting with the diagram.
\end{theorem}

This theorem is applied to our situation as follows. Let $D\subset ]s[_S$ be the closed poly-disk of radius $|\pi _K |$ containing all $K$-points of $D$. Define $U_D =U_{v,D}$ to be the affinoid $]U_v [ \times _S D$, and define $U_s =U_{v,s}$ to be the fibre of $U_D$ at $s$. 

\begin{lemma}\label{lemma:locally_trivial}
There is an isomorphism of rigid analytic spaces over $D$
\[
(\gamma ,\delta ) :U_D \stackrel{\simeq }{\longrightarrow } U_s \times D.
\]
restricting to the identity at $s$.
\end{lemma}
\begin{proof}
$D$ has a formal model $\mathfrak{D}$ over $\mathcal{O}_K$ isomorphic to the formal spectrum of $\mathcal{O}_K \langle t_1 ,\ldots ,t_d \rangle $, and in particular the special fibre of $\mathfrak{D}$ is isomorphic to $\mathbb{A}^d _k$. We can form formal models $\mathfrak{U}_{\mathfrak{D}}$ and $\mathfrak{U}_s$ of $U_D$ and $U_s$ such that the special fibre of $U_D$ is isomorphic, over $\mathfrak{D}_k$, to $\mathfrak{U}_{s,k}\times _k \mathfrak{D}$. 

We can apply Theorem \ref{thm:CI} with  $R_1 =\mathcal{O}(\mathfrak{D})$, $R_2 $ equal to the weak completion over $\mathcal{O}(\mathfrak{D})$ of $\mathcal{O}(\mathfrak{U}_{\mathfrak{D}})$, $R_4 $ the weak completion over $\mathcal{O}_K$ of $\mathcal{O}(\mathfrak{U}_s )$, and $R_3 $ the weak completion over $\mathcal{O}(\mathfrak{D})$ of $\mathcal{O}(\mathfrak{U}_s \times _{\mathcal{O}_K }\mathfrak{D})$, to deduce the existence of an isomorphism of $\mathcal{O}(\mathfrak{D})$-algebras $R_2 \simeq R_3 $ making the diagram
\[
\begin{tikzcd}
R_1 \arrow[r] \arrow[d] & R_3 \arrow[d] \\
R_2 \arrow[r]   \arrow[ru]              & R_4 .               
\end{tikzcd}
\]
By the interpretation of weak completion in terms of strict neighbourhoods \cite[Proposition 3.14]{CI10}, we deduce that there are strict neighbourhoods $Y_1 $ and $Y_2 $ of $U_D$ in $X_D$ and $U_s \times D$ in $X_s \times D$, and an isomorphism $Y_1 \simeq Y_2 $ of rigid analytic spaces over $D$, such that the diagram
\[
\begin{tikzcd}
Y_{1,s} \arrow[r] \arrow[d] & Y_2 \arrow[d] \\
Y_1 \arrow[r]   \arrow[ru]              & D .               
\end{tikzcd}
\]
commutes. 
\end{proof}
We now want to consider the universal connection in families. In this setting, we consider the universal connection over a curve rather than an abelian variety. Since we want our relative curve to have a section, we consider $X^2 _S$ as a curve over $X$, with distinguished section the diagonal morphism $\Delta _X :X\to X\times _S X$. Note that the variety $X^2 _S$ will be singular at points $(x,y;s)$ where $s$ in in $Z$ and $x$ and $y$ are singular points of $X_s$. Let $\widetilde{X}^2 _S$ denote a resolution of $X^2 _S$ (recall from \cite[\S 2.1]{CHM} that this resolution exists over $\mathcal{O}_K$, not just $K$). When restricted to the pre-image of the tube of $U_v \times U_v$, this map is an isomorphism, and hence in this subsection we will not need to distinguish between the singular space and its resolution.

Let $\mathcal{E}^2 (X)$ be a universal connection on $X^2 _{S_K -Z_K}$ in the sense of Definition \ref{defn:lifting_univ_conn}, associated to the pair $(\pr _2 ,\Delta)$, where $\pr _2 :X^2 _{S_K -Z_K } \to X_{S_K -Z_K }$ is the projection and $\Delta :X_{S_K -Z_K}\to X^2 _{S_K -Z_K }$ is the diagonal. Abusing notation somewhat, we will also denote by $\mathcal{E}^2 (X)$ the canonical extension to $\widetilde{X}^2 _{S_K}$.

Let $\mathcal{V}$ denote the dual of the Gauss--Manin connection associated to $\pi :X_K \to S_K $, restricted to $]s[$. Now consider $\mathcal{E}^2 (X)$ restricted to $]U_v \times U_v [_{X\times _S X}$. Then $\mathcal{E}^2 (X)$ is an extension of the trivial connection by $\pi ^{2 *}\mathcal{V}$.
\begin{lemma}\label{lemma:BC}
Let $t_1 ,\ldots ,t_n$ be local parameters at $s$ for the divisor $Z_K$. Then $\mathcal{V}$ is isomorphic to a connection of the form
\begin{equation}\label{eqn:formal_to_rigid}
d-\sum N_i \cdot \frac{dt_i }{t_i },
\end{equation}
where the $N_i$ are constant, nilpotent matrices which pairwise commute. 
\end{lemma}
\begin{proof}
We may write $\mathcal{V}$ as $\mathcal{O}_{]s[}^{\oplus 2g}$, with connection given by
$d-\sum _i M_i \frac{dt_i }{t_i }$ for some $M_i $ in $\Mat _{2g}(\Gamma (]s[,\mathcal{O}))$. Then an isomorphism with a connection of the form \eqref{eqn:formal_to_rigid} exists formally, i.e. at the formal completion of $S$ at zero. Then by Baldassarri--Chiarellotto \cite[Corollary 4.9.1]{BC92}, this isomorphism extends to $]s[$.
\end{proof}

\begin{lemma}\label{lemma:kedlaya}
The restriction of $\mathcal{E}^2 (X)$ to $]U_v \times U_v [_{X\times _S X}\cap (\pi ^2 )^{-1}(D)$ is isomorphic, via the isomorphism
\[
]U_v \times U_v [_{X\times _S X}\cap (\pi ^2 )^{-1}(D)\simeq U_{v,s}\times D,
\] to a connection of the form 
\[
d-\Lambda -\Theta ,
\]
where $\Lambda $ is a nilpotent matrix pulled back from $U_{v,s}\times U_{v,s}$, and $\Theta $ is a nilpotent matrix pulled back from $D$ with logarithmic singularities along $D\cap Z$.
\end{lemma}
\begin{proof}
Recall that $\mathcal{E}^2 (X)$ is an extension of the trivial connection by $\pi ^{2*}\mathcal{V}$. By Lemma \ref{lemma:BC}, this implies that the restriction of $\mathcal{E}^2 (X)$ is a unipotent connection on $U_{v,s}^2 \times D$. Hence the lemma follows from Kedlaya's description of unipotent connections with logarithmic singularities on spaces of the form $D\times V$, for $V/K$ affinoid \cite[Theorem 3.3.4]{kedlaya07}.
\end{proof}

This gives the following description of parallel transport on $\mathcal{E}^2 (X)$ in a formal neighbourhood of a point $x$ in $]U_v \times U_v [_{X\times _S X}$ such that $\pi ^2 (x)$ is not in $Z_K$. Pick $x=(x_1 ,x_2 ;\widetilde{s})  \in ]U_v \times U_v [ (K)$. Let $\mathcal{U}$ denote the sheaf on $X^2 _{S_K}$ whose $U$-sections are unipotent isomorphisms
\[
\mathcal{E}^2 |_U \stackrel{\simeq }{\longrightarrow }\gr ^\bullet {\mathcal{E}}^2 |_U .
\]
Then $\mathcal{U}$ is an $(R^1 \pi _* \Omega _{X_K |S_K } ) ^* $-torsor over $X^2 _{S_K}$.

Similar to the good reduction setting, for each $\widetilde{s}\in (]s[-Z)(K)$, Berkovich--Coleman integration defines a map
\[
\int ^{\mathrm{BC}}:(]U_v \times U_v [_{X\times _S X})_{\widetilde{s}} (K)\to \widetilde{s}^* \mathcal{U}(K)
\]
on the fibres of $]U_v \times U_v [_{X\times _S X}$ away from $Z$. This may be seen as follows. Since $(]U_v \times U_v [_{X\times _S X})_{\widetilde{s}} $ has good reduction, we may apply Besser's characterisation of Coleman integration as giving unipotent isomorphisms between the fibres of $\mathcal{E}^2 |_{(]U_v \times U_v [_{X\times _S X})_{\widetilde{s}} }$. On the other hand, on the diagonal $\Delta \cap (]U_v \times U_v [_{X\times _S X})_{\widetilde{s}} (K)$, we have a chosen unipotent trivialisation
\[
\Delta ^* \mathcal{E}^2 \simeq \Delta ^* \gr ^\bullet \mathcal{E}^2 .
\]
It follows that Coleman integration defines a locally analytic section of $\widetilde{s}^* \mathcal{U}(K)$, which we denote
\[
\int ^{\mathrm{BC}}:(]U_v \times U_v [_{X\times _S X})_{\widetilde{s}} (K)\to \widetilde{s}^* \mathcal{U}(K)
\]

Explicitly, this can be defined by choosing a trivialisation of the vector bundle $\mathcal{E}^2$ on $U_{\widetilde{s}}\times U_{\widetilde{s}} $, which can be given for example by a set of differentials $\omega _i $ in $H^0 (U_s ,\Omega )$ forming a basis of $H^1 _{\dR}(X)$. Such a trivialisation defines a trivialisation of $\widetilde{s}^* \mathcal{U}(K)$, i.e. an isomorphism 
\[
\widetilde{s}^* \mathcal{U}(K)\simeq H^1_{\dR}(X_{\widetilde{s}} /K) \simeq K^{2g}.
\]
Then via this isomorphism, $\int ^{\mathrm{BC}}(x,y):=(\int ^x _y \omega _i )_{i=1}^{2g}$.

We now relate this to parallel transport on the residue disk of $x$ in $X\times _S X$. Let $\widetilde{P}$ denote the frame bundle whose $U$-sections are filtered isomorphisms 
\begin{equation}\label{eqn:Psections}
g:x^* \mathcal{E}^2 \otimes \mathcal{O}_U \stackrel{\simeq }{\longrightarrow }\mathcal{E}^2 |_U
\end{equation}
Let $G$ denote the Galois group of $\mathcal{E}^2$, and let $P$ denote a descent of $\widetilde{P}$ to a $G$-principal bundle.  
As in the good reduction setting, we have a map of sheaves
\[
p:P \to \mathcal{U}
\]
given on $U$-sections by sending an isomorphism $g$ as in \eqref{eqn:Psections} to $\gr ^\bullet (g)\circ \int ^{\mathrm{BC}}(x_1 ,x_2 )\circ g^{-1}$. 

Let $\widehat{X}^2 _{S,x}$ denote the formal completion of $X^2 _S$ at $x$. Let
\[
\alpha :\widehat{X}^2 _S \to P
\]
denote the formal leaf of the foliation defined by $\mathcal{E}^2$ passing through the identity element at $x$.
\begin{lemma}\label{lemma:iso_is_GM}
\begin{enumerate}
\item With respect to the chart given in Lemma \ref{lemma:kedlaya}, $\alpha $ is given by 
\begin{equation}\label{eqn:GM_badred}
\widetilde{G}=\exp (\sum N_i \cdot \log (t_i ))\cdot G,
\end{equation} 
where $G$ is parallel transport on $U_{v,\widetilde{s}}$.
\item The map $\beta :=p\circ \alpha  $ extends to a locally analytic map
\[
\beta :U_{v,D}(K)\to \mathcal{U}(K)
\]
\end{enumerate}

\end{lemma}
\begin{proof}
For part (1), \eqref{eqn:GM_badred} is a solution to the equation 
\[
d\widetilde{G}=\Lambda G+\Theta G.
\]
For part (2), we can write the matrices $G$ and $N_i$ in block form as 
\[
\left( \begin{array}{cc}1 & 0 \\ f & 1 \end{array}\right) , \left( \begin{array}{cc} 0 & 0 \\ g_i & J_i \end{array}\right)
\]
respectively. Since the section $s$ passes through $v$. $g_i$ can be taken to be zero. Since $(J_i-1)^2=0$, and we assume commutativity, we find that $\widetilde{G}$ can be written as
\[
\left( \begin{array}{cc} 1 & 0 \\ f & \prod _i J_i  \end{array} \right).
\]
Then the map $]x[(K)\to \mathcal{U}(K)$ is given by 
\[
\left( \begin{array}{cc} 1 & 0 \\ f & 1 \end{array}\right) .
\]
\end{proof}

\begin{lemma}\label{lemma:coleman_to_leaves}
For all $t$ as above, and $y_1 ,y_2 \in U_t (K)$ congruent to $x_1 $ and $x_2 $ modulo $\pi $,
\[
\int ^{\mathrm{BC}}(y_1 ,y_2 )=\beta (y_1 ,y_2 ;t).
\]
\end{lemma}
\begin{proof}
We maintain the notation from Lemma \ref{eqn:GM_badred}. By part (2) of Lemma \ref{eqn:GM_badred}, $\beta $ is simply given by parallel transport along the connection $d-\Lambda $.

Let $\gamma _t $ denote the isomorphism $U_t \simeq U_s$. Then this description of $\beta $ implies that, with respect to a chosen basis $\omega _i$ of $H^1 _{\dR}(X_s /K)$, is simply given by 
\[
\beta (y_1 ,y_2 ;t)=(\int ^{\gamma _t (y )}_{\gamma _t (y_2 ) }\omega _i ).
\]
while 
\[
\int ^{\mathrm{BC}}(y_1 ,y_2 )=\int ^{y_1 } _{y_2 } \gamma _t ^* \omega .
\]
Hence the Lemma follows from Proposition \ref{prop:BC_functorial}
\end{proof}

We define the Faltings--Zhang map to be the morphism
\[
\AJ^{\FZ}:X^n _{S_K -Z_K } \to \Jac (X/S-Z) ^{n-1}_{S_K -Z_K} 
\]
which fibrewise sends $((x_i )_{i=1}^n ;s)$ to $((\AJ (x_i -x_n ))_{i=1}^{n-1};s)$. We define the $p$-adic Faltings--Zhang logarithm $\log ^{\FZ}$ to be the composite of $\AJ ^{\FZ}$ with the $p$-adic logarithm
\[
\Jac (X/S-Z)^n _S (K )\to \Lie (\Jac(X/S-Z))^{n-1}_{S-Z} (K ).
\]
As in the good reduction case, this is a a priori just a map of sets defined fibrewise by $p$-adic integration. However, as in the good reduction it in fact has additional structure, enabling a description in terms of foliations. We now relate this construction to the connection $\mathcal{E}^2 $ on $X^2 _{S_K -Z_K}$. Let $\mathcal{E}^n $ denote the connection on $X^n _{S-Z} $ obtained by taking the direct sum of $(\pr _i ,\pr _n )^* \mathcal{E}^2 $ for $i=1,\ldots ,n$. 
Pick $x_i \in X(K)$ lying above $s\in S(K)$ for $i=1,\ldots ,n-1$, and let $x:=(x_1 ,\ldots ,x_n ;s)\in X^n _S (K)$. Let $\widetilde{P}_n ^{\FZ}$ denote the frame bundle whose $U$-sections are filtered isomorphisms 
\begin{equation}\label{eqn:Psections2}
g:x^* \mathcal{E}^n \otimes \mathcal{O}_U \stackrel{\simeq }{\longrightarrow }\mathcal{E}^n |_U
\end{equation}
Let $G_n ^\FZ$ denote the Galois group of $\mathcal{E}^n$. Explicitly, $G_n ^{\FZ}$ will be an extension of $G_0$, the Galois group of the Gauss--Manin connection on $S_K -Z_K$, by the vector group $(R^1 \pi_* \Omega _{X_K |S_K -Z_K }^* )^{\oplus (n-1)}$. Let $P_n ^{\FZ}$ denote a descent of $\widetilde{P}_n ^{\FZ}$ to a $G_n ^{\FZ}$-principal bundle. Let $\mathcal{U}$ denote the sheaf whose $U$-sections are unipotent isomorphisms
\[
\mathcal{E}^n |_U \stackrel{\simeq }{\longrightarrow }\gr ^\bullet {\mathcal{E}}^n |_U .
\]
Then $\mathcal{U}$ is an $((R^1 \pi _* \Omega _{X_K |S_K } ) ^* )^{\oplus (n-1)}$-torsor over $X^n _{S_K -Z_K}$. 
We have a map of sheaves
\[
P_n \to \mathcal{U}
\]
given on $U$-sections by sending an isomorphism $g$ as in \eqref{eqn:Psections2} to $\gr ^\bullet (g)\circ ( \int ^{\mathrm{BC}}(x_1 ,\ldots ,x_n)\otimes 1 )\circ g^{-1}$, where
\[
\int ^{\mathrm{BC}}(x_1 ,\ldots ,x_n ):=\prod _{i=1}^{n-1} \int ^{\mathrm{BC}}(x_i ,x_n ),
\]
viewed as an isomorphism from $\gr ^\bullet (x_1 ,\ldots ,x_n )^*\mathcal{E}^n $ to $(x_1 ,\ldots ,x_n )^*\mathcal{E}^n $ via the isomorphism
\[
(x_1 ,\ldots ,x_n )^* \mathcal{E}^n \simeq \prod _{i=1}^{n-1}(x_i ,x_n )^* \mathcal{E}^2 .
\]
Then, as in \eqref{eqn:whatisp}, we obtain a dominant morphism of $S$-schemes
\[
p :P_n ^{\FZ}\to \Lie (J/S-Z)^{n-1}_{S_K-Z_K} .
\]
Let $z\in x^* P_n ^{\FZ}(K)$ be a point corresponding (via Corollary \ref{cor:bess2}) to the Coleman integrals $\omega \mapsto \int ^{x_i }_{x_n }\omega $ for $i=1,\ldots ,n-1$. Let $D-Z:=D-Z_K ^{\an }\cap D$.
\begin{lemma}\label{lemma:GM_bad1}
Let $U_v \subset X_s $ be the affine open obtained from an irreducible component $X_v$ by removing the singular points. Let
\[
\alpha :(]U_v [ ^n)_{]s[}(K)\to P^{\FZ}_n |_{]s[}(K)
\]
be the leaf of the foliation through $z$.
Then the diagram
\[
\begin{tikzcd}
(U _v ^n  )_{D-Z} (K) \arrow[d, "\alpha"] \arrow[rd, "\log ^{\FZ}"] &               \\
P^{\FZ} _n | _{D-Z}(K) \arrow[r, "p"]                     & \Lie (J/S-Z)^{n-1} |_{D-Z}(K)
\end{tikzcd}
\]
commutes.
\end{lemma}
\begin{proof}
It is enough to prove the case $n=2$. Let $t:=\pi (y_i )$. Let $(\gamma ,\delta )$ be as in Lemma \ref{lemma:locally_trivial}. Then, via $(\gamma ,\delta )$, the morphism
\[
P_2\to \mathcal{U}
\]
is given by projection onto the first coordinate. Hence it suffices to prove commutativity of 
\[
\begin{tikzcd}
(U _v ^2  )_{D-Z} (K) \arrow[d, "\alpha"] \arrow[rd, "\log ^{\FZ}"] &               \\
\mathcal{U} | _{D-Z}(K) \arrow[r, "p_2 "]                     & \Lie (J/S-Z) _{D-Z}(K),
\end{tikzcd}
\]
which follows from Lemma \ref{lemma:coleman_to_leaves}.
\end{proof}

We now complete the proof of part (1) of Proposition \ref{prop:thm1case2}. Let $P_{n}^{\FZ} \to X^n _S$ be the principal $G$-bundle defined above. Then the image of $X^n _S (K )_{\rk r}\cap ]U_v ^n [\cap (\pi ^n )^{-1}(S-Z)$ in $\Lie (J/S-Z)^{n-1}_{S_K -Z_K} $ under the map $\log ^{\FZ} $ is contained in $\mathbb{D}_r (\Lie (J/S-Z)_{S_K -Z_K}^{n-1})$. By Lemma \ref{lemma:GM_bad1}, we deduce that rank $\leq r$ points map into a codimension $\geq (n-1-r)(g-r)$ subspace of $P_n ^{\FZ}$. Case (1) of Proposition \ref{prop:thm1case2} now follows from Theorem \ref{thm:BCFN}.

\subsection{Proof of Proposition \ref{prop:thm1case2} part (2): integration on annuli and Stoll's trick}
Now let $x\in X_s (k)$ be a singular point of $X_s $. Recall that our goal is to prove Zariski non-density of rank $r$ points on the tube of $(x,\ldots ,x;s)$ inside $X^n _S$. The first problem we encounter is that the fibre product $X^n _S$ will have singularities (over $K$, not just $S$) when $n>1$. These singularities will only lie at the fibre product of singular points of the morphism $\pi $, i.e. at points $(x_1 ,\ldots ,x_n ;s)$ where $X_s$ is singular and at least two of the $x_i$ are singularities of $X_s$. In particular, the tube of $(x,\ldots ,x;s)$ contains singular points. Let $Z\subset S$ be the singular locus of $\pi $. We deal with the issue of singularities by constructing analytic functions on the tube of $(x,\ldots ,x)$ in the smooth $K$-scheme $X^n _K$, and showing that, when pulled back to $X^n _{S_K -Z_K }$ via the map $X^n _{S_K} \to X^n _K$, they are related to Coleman integration and parallel transport for the Gauss--Manin connection.

First, we review Stoll's trick for bounding Chabauty--Coleman sets near points of bad reduction \cite{stoll:JEMS} (as formulated for general curves by Katz, Rabinoff and Zureick-Brown \cite{KRZB}). Let $X$ be a strictly stable curve over $\mathcal{O}_{K}$ with smooth generic fibre. Let $x$ be a $k$ point of $X$. We define the \textit{residue} map
\[
\Res _x :H^0 (X_{K },\Omega )\to K
\]
as follows. If $x$ is a smooth point, the residue map is zero. If it is a singular point, we may identify $]x[$ with an annulus, with parameter $t$, and define $\Res _x (\omega )$ to be the $dt/t$ coefficient of its annular expansion. As in the case of Laurent series, this independent of the choice of $t$, up to sign. We define the Stoll logarithm, denoted $\log _x ^\Sto$ at $]x[$, relative to some fixed basepoint $b\in ]x[$, to be the map
\begin{equation}\label{eqn:stoll}
\log ^\Sto _{x}:]x[ \to \Ker (\Res _x )^*
\end{equation}
obtained from restricting the Berkovich--Coleman integration functional (in the sense of \cite{KRZB}) with basepoint $b$ to $\Ker (\Res _x )$. Whilst this depends on $b$, the map on $]x[\times ]x[$ sending $(x_1 ,x_2 )$ to $\log ^{\Sto }_x (x_1 ) -\log ^{\Sto }_x (x_2 )$ is independent of $b$.

The advantage of $\log _x ^{\Sto }$ is that it factors through the Abel--Jacobi morphism, by the following result of Katz, Rabinoff and Zureick--Brown.
\begin{proposition}[\cite{KRZB}, Proposition 3.16]\label{prop:KRZB}
Let $\log ^{\ab }$ denote Zarhin's abelian logarithm $]x[ \to H^0 (X,\Omega )^* $. Then the diagram
\[
\begin{tikzcd}
{]x[} \arrow[d, "\log ^{\ab }"] \arrow[rd, "\log _x ^{\Sto }"]     &                    \\
{H^0 (X,\Omega )^* } \arrow[r] & \Ker (\Res _x )^*
\end{tikzcd}
\]
commutes. 
\end{proposition}

We now return to the context of the stable family $X\to S$ from before, with $s\in S(k)$ and $x\in X_s (k)$ as above. We deduce that for each $\widetilde{s}\in ]s[$, rank $r$ $n$-tuples in $]x[_{X_{\widetilde{s}}}$ map into rank $\leq r$ subspaces of $(\Ker (\Res _x )^* )^{\oplus n}$, and the map from $]x[_{X_{\widetilde{s}}}$ to $\Ker (\Res _s )^*$ is analytic. To use this to complete the proof of Theorem \ref{thm:main1}, it remains to show that the map $\log ^\Sto _x$ behaves well in families. 

Let $D\subset ]x[$ denote a closed polydisk of radius $|\pi _K |$ containing all $K$-points of $]x[$. Let $\mathcal{E}(X)_D$ denote the canonical extension of the universal connection on $X_{D-Z} \to D$ relative to a fixed section $\sigma :D\to X_D$. The connection $\mathcal{E} (X)$ restricted to $]x[$ again satisfies the hypotheses of the Baldassarri--Chiarellotto theorem \cite[Corollary 4.9.1]{BC92}, and hence we may write it as 
\[
d-N_0 \cdot \frac{du}{u} -\sum _{i=1}^n N_i \cdot \frac{dt_i }{d_i }.
\]
It follows that $\Ker (\Res )^{*}$ is an analytic quotient of $\mathcal{V}|_{D}$, and that the map
\[
D(K)-Z(K)\to \Ker (\Res )^{*}
\]
is analytic.
\begin{lemma}\label{lemma:relation_to_Stoll}
There is a morphism of vector bundles
\[
\widetilde{\Res } :\mathcal{O}_D\hookrightarrow \mathcal{V}|_D
\]
such that
\begin{enumerate}
\item For all $t\in D(K)$, $t^* \widetilde{ \Res }$ is equal to $\Res _{\overline{s}}$ viewed as a functional on $H^1 _{\dR}(X_t /K)$.
\item For all $t\in D(K)$, the pullback of $\widetilde{\Res }$ to $\widehat{S}_t$ is a formal leaf of the foliation associated to the connection on $\mathcal{V}$.
\end{enumerate}
\end{lemma}
\begin{proof}
Write the connection on $\mathcal{V}|_D$ as $d-\sum N_i \cdot \frac{dt_i }{t_i }$ as above. By flatness, the $N_i$ must commute. With respect to this trivialisation, the residue map is equal to the bottom left block of $N_0$. This shows that the residue map is analytic, and that the residue nonzero elements lie in the kernel of $N_i$ for all $i>0$, and hence are horizontal sections of $\mathcal{V}|_D$.
\end{proof}

Let $\mathcal{U}$ denote the rigid analytic $\mathcal{V}$-torsor on $X_D$ whose $U$-sections are unipotent isomorphisms
\[
(\mathcal{E}(X)_D )|_U \stackrel{\simeq }{\longrightarrow } (\pi _D ^* \circ \sigma ^* \mathcal{E}(X)_D ).
\]
If $\mathcal{W}$ is a sub-bundle of $\mathcal{V}$, we define $\mathcal{U}/\mathcal{W}$ to be the $\mathcal{V}/\mathcal{W}$-torsor of unipotent isomorphisms
\[
(\mathcal{E}(X)_D  /\mathcal{W} )|_U \stackrel{\simeq }{\longrightarrow } \gr ^\bullet  \mathcal{E}(X)_D /\mathcal{W} .
\]
Recall that $t_1 ,\ldots ,t_n$ are local parameters for $Z$ at $]s[$. Suppose that the singularity at $x$ is locally given by $uv-t_1$.
We write $\mathcal{E}(X)_D |_{]x[}$ as 
\begin{equation}\label{eqn:chart}
d-N_0 \frac{du}{u}-\sum _{i=1}^n N_i \frac{dt_i }{t_i }.
\end{equation}
As in the previous subsection, we may construct a bundle $P_z$ of isomorphisms $z^* \mathcal{E}(X)_D \otimes \mathcal{O}_{X_D } \simeq \mathcal{E}(X)_D$, and a quotient map
\[
p:P_z \to \mathcal{U}
\]
via choosing a unipotent vector space isomorphism $z^* \mathcal{E}(X)_D \simeq \gr ^\bullet z^* \mathcal{E}(X)_D$. In a formal neighbourhood of a point $z$, parallel transport defines a section of $P_z$, and we may define a map
\[
\beta :\widehat{X}_z \to \widehat{\mathcal{U}}
\]
by composing the parallel transport isomorphism with $p$. As in the previous subsection, we may prove the following.
\begin{lemma}
\begin{enumerate}
\item Let $z$ be a point of $]x[$ not mapping to $Z$. Then, with respect to the trivialisation of the bundle $\mathcal{E}(X)_D$ in \eqref{eqn:chart}, the map
\[
\beta : \widehat{X}_z \to \widehat{\mathcal{U}}
\]
is given by $\exp (N_0 \log (u)).$
\item The composite of $\beta $ with projection to $\mathcal{U}/\mathcal{W}$ extends to analytic map $\beta _{\mathcal{W}}$ on the whole of $]x[_D$.
\end{enumerate}
\end{lemma}

We now explain how to prove part (2) of Proposition \ref{prop:thm1case2} by modifying the definition of the degeneracy locus to accomodate the Stoll logarithm. Note that, since $\mathcal{W}$ is an \textit{analytic} sub-bundle, if we take the naive definition of a degeneracy locus a priori this does not give a way to apply Theorem \ref{thm:BCFN} to prove Zariski non-density of low rank points, because the degeneracy locus will not be an algebraic subvariety.  Let $\mathcal{E}^2 (X)$ be the universal connection on $X/S-Z$, and let $P,P_0 ,G$ and $G_0$ be as in section \ref{subsec:variation}. Let $\widetilde{\mathcal{V}}$ denote the dual of $R^1 \pi _* \Omega _{X_K |S_K }(\log (Z _K ))$ 

By Lemma \ref{lemma:relation_to_Stoll}, we have an overconvergent sub-bundle $\mathcal{W}$ of $\mathcal{V}|_{D-Z}$, such that, for all $z\in (D-Z)(K)$, $z^* \mathcal{W}$ spans $(\Res _x )$, moreover, for all $z\in ]s[-Z (K)$, the pullback of $\mathcal{W}$ to the formal completion of $S$ at $z$ may be identified with the horizontal section extending $z^* \mathcal{W}$.

We may quotient the map $\beta _{\mathcal{W}}$ by $\mathcal{M}$, to form
\[
\log ^{\Sto }:]x[_X \to R^1 \pi _* \mathcal{O}_X |_{D}/\mathcal{W}.
\]
This map has the property that, for any $\widetilde{s}$ in $D-Z$ and $x_1 ,x_2 \in ]x[_{X_{\widetilde{s}}}$, 
\begin{equation}\label{eqn:stoll_equals_stoll}
\log ^{\Sto }(x_1 )-\log ^{\Sto }(x_2 )=\log ^{\Sto }_x (x_1 )-\log _x ^{\Sto }(x_2 ),
\end{equation}
where the righthand side is the Stoll logarithm from \eqref{eqn:stoll} (recall that this difference of Stoll logarithms is independent of the choice of $b$).

This implies the following lemma, which relates the Stoll logarithm defined in terms of Coleman integration with the one defined in terms of parallel transport.
\begin{lemma}\label{lemma:stollFZ}
Choose a vector bundle trivialisation
\[
\rho :R^1 \pi _* \mathcal{O}_X |_D /\mathcal{W}\simeq \mathcal{O}_D ^{g-1}
\]
Let
\[
\log ^{\Sto ,\FZ }_n : ](x,\ldots ,x;s)[_{X^n _{S} } \cap (\pi ^n )^{-1}(D-Z) \to (\Lie (J/S-Z)/\mathcal{W})^{n-1}_{S-Z}
\]
be the map sending $(x_1 ,\ldots ,x_n ;\widetilde{s})$ to $(\log _{\widetilde{s}} ^{\Sto }(x_i )-\log _{\widetilde{s}} ^{\Sto }(x_n ))_{i=1}^{n-1}$. Let $\delta $ be the map 
\[
(R^1 \pi _* \mathcal{O}_X |_D )^n \to (R^1 \pi _* \mathcal{O}_X |_D )^{n-1}
\]
sending $(v_1 ,\ldots ,v_n )$ to $(\rho ^{-1})^{n-1}\circ \delta _0 \circ \rho ^n $, where 
\[
\delta _0 :(\mathcal{O}_D ^{g-1})^n \to (\mathcal{O}_D ^{g-1})^{n-1}
\]
sends $(v_i )_{i=1}^n $ to $(v_i -v_n )_{i=1}^{n-1}$.

Then $\log ^{\Sto ,\FZ} _n $ is the restriction of the composite map
\[
\delta \circ(\log ^{\Sto } )^n :](x,\ldots ,x)[_{X^n _{\mathcal{O}}} \to (R^1 \pi _* \mathcal{O}_X |_{]s[}/\mathcal{W})^n .
\]
to the preimage of $S-Z$ diagonally embedded into $S^n _K$.

In particular, $\log ^{\Sto ,\FZ}_n$ is the restiction to the subvariety $X^n _{S-Z}$ of an analytic function on $](x,\ldots ,x)[_{X^n _{\mathcal{O}}}$
\end{lemma}

Let $z\in ](x;\ldots ;x,s)[(K)$ be a point mapping to $D-Z$. Let
\[
\overline{p} :P_n ^{\FZ}\to (\Lie (J/S-Z)_K /\mathcal{W})^{n-1}_S 
\]
be the composite of the map $p$ with the quotient by $\mathcal{W}$. Let $\theta \in z^* P_n ^{\FZ}(K)$ be the tautological section corresponding to the identity, projecting to $\theta _0 \in z^* P_0$. Let
\[
\alpha :](x,\ldots ,x;s)[ (K)\cap (\pi ^n ) ^{-1}(D-Z)(K))\to P_n ^{\FZ}(K)
\]
be the leaf of $P_n ^{\FZ}$ through $\theta $, and $\alpha _0 $ the leaf of $P_0$ through $\theta _0$. 

\begin{proposition}\label{prop:endgame}
There is an algebraic variety $\mathcal{S}_{W,\mathcal{M}}$ over $X^n _{S_K -Z_K}$, a map of algebraic varieties
\[
\overline{p}:(P_n ^{\FZ })_{]e[^n _{D-Z} }\to \mathcal{S}_{W,\mathcal{M}}
\]
and a closed immersion of rigid analytic spaces over $]e^n [_D$
\[
\iota : \Lie (J/S_K -Z_K )_D \to \mathcal{S}_{W,\mathcal{M},]e^n [_D }
\]
with the following properties.
\begin{enumerate}
\item Upon taking formal completions, the diagram
\begin{equation}\label{eqn:FZ_commutes}
\begin{tikzcd}
\widehat{X}^n _{S,z}
\arrow[d, "\alpha"] \arrow[r, "\log ^{\Sto ,\FZ} _n"] &       (\widehat{\Lie } (J/S-Z)_{\log (z)}/\widehat{\mathcal{W}})^{n-1} \arrow[d, "\iota "]     \\
\widehat{P}_{n,\theta } ^{\FZ} \arrow[r, "\overline{p}"]  & \widehat{\mathcal{S}}_{W,\mathcal{M}}                 
\end{tikzcd}
\end{equation}
commutes.
\item There is a codimension $(g-e-d)(n-d)$ subscheme $\mathbb{D}_r (P_n ^{\FZ })$ of $P_n ^{\FZ}$ a subspace $\mathbb{D}_r ((\Lie (J/S_K-Z_K )_D /\mathcal{W} ^{n-1}))$, and a subscheme $\mathbb{D}_r (\mathcal{S}_{W,\mathcal{M}})$ such that  
\begin{align*}
\log ^{\Sto ,\FZ }_n (X^n _S (K)_{\rk \leq r}\cap ]e[^n _D) & \subset \mathbb{D}_r (\Lie (J/S_K-Z_K )_D /\mathcal{W} ^{n-1}),  \\
\overline{p}^{-1}(\mathbb{D}_r (\mathcal{S}_{W,\mathcal{M}})) & \subset \mathbb{D}_r (P_n ^{\FZ }), \\
\iota ^{-1}(\mathbb{D}_r (\mathcal{S}_{W,\mathcal{M}})) & = \mathbb{D}_r (\Lie (J/S_K-Z_K )_D /\mathcal{W} ^{n-1}).
\end{align*}
\end{enumerate}
\end{proposition}
This Proposition implies part (2) of Proposition \ref{prop:thm1case2}. Indeed, the Proposition implies that it is enough to show that preimage of $\mathbb{D}_r (\mathcal{S}_{W,\mathcal{M}})$ in $]e^n [_{X^n _S}(K)$ under the map $\iota \circ \log ^{\Sto ,\FZ}_n $ is not Zariski dense. By Lemma \ref{lemma:analytic2formal} and Lemma \ref{lemma:stollFZ}, it is enough to show this in a formal neighbourhood of a point of $]e^n [_{X^n _S }$. By the commutativity of \eqref{eqn:FZ_commutes}, it is enough to prove this for the pre-image of $\mathbb{D}_r (P_n ^{\FZ})$, where (by Proposition \ref{prop:endgame}) it is a consequence of Theorem \ref{thm:BCFN}.

The definition of $\mathbb{D}_r (\Lie (J/S_K -Z_K )_D /\mathcal{W})^n $ is simply the previous notion of degeneracy loci applied to the analytic vector bundle $\Lie (J/S_K -Z_K )_D /\mathcal{W}$. By \eqref{eqn:stoll_equals_stoll} and Proposition \ref{prop:KRZB}, $\log ^{\Sto ,\FZ}_n$ maps rank $\leq r$ $n$-tuples of points into rank $\leq r$ $n$-tuples of sections of $\Lie (J/S_K -Z_K)_D /\mathcal{W}$.

Let $P_0$ denote the $G_0$-principal bundle whose $U$-sections are isomorphisms
\[
\widetilde{\mathcal{V}} |_U \simeq z^* \widetilde{\mathcal{V}}\otimes \mathcal{O}_U .
\]
Fix a subspace $W$ of $z^* \widetilde{\mathcal{V}}$ of dimension $e$. Define $\mathcal{S}_W $ to be the sheaf whose $U$-sections are pairs $(\rho ,s)$, where $\rho $ is a $G_0$-equivariant isomorphism 
\[
z^* \widetilde{\mathcal{V}} \otimes \mathcal{O}_U \stackrel{\simeq }{\longrightarrow }\widetilde{\mathcal{V}}|_U ,
\]
and $s$ is a unipotent isomorphism 
\[
\mathcal{E}^2 (X)|_U /\rho (W\otimes \mathcal{O}_U ) \simeq \gr ^\bullet \mathcal{E}^2 (X)|_U /\rho (W\otimes \mathcal{O}_U ) .
\] 
Then $\mathcal{S}_W $ is represented by a scheme over $X$, which is naturally a quotient of $P$, and projects onto $P_0$. We similarly construct a sheaf $\mathcal{S}_{W,n}$ on $X^n _S$ as the $n$-fold fibre product over $P_0$ of the pullbacks along the $n$ coordinate projections of the sheaves $\mathcal{S}_W$
\[
\mathcal{S}_{W,n}:=(\pr _1 ,\pr _n )^* \mathcal{S}_W \times _{P_0 }\ldots \times _{P_0 }(\pr _{n-1},\pr_n )^* \mathcal{S}_W
\]Concretely, it is locally described as follows: let $U_0 \subset S$ and $U_i \subset \pr _i ^{-1} U_0$ for $i=1,\ldots ,n$. Then a section on $U_1 \times _{U_0 }\times \ldots \times _{U_0 }U_n$ is a tuple $(\rho, (s_i )_{i=1}^n )$ where $\rho $ is as above, and $s_i$ is a unipotent isomorphism 
\[
(\pr _i ,\pr _n )^* \mathcal{E}^2 (X)|_{U_i \times _{U_0 }U_n } /\rho (W\otimes \mathcal{O}_{U_i \times _{U_0 }U_n } ) \simeq \gr ^\bullet (\pr _i ,\pr _n )^* \mathcal{E}^2 (X)|_{U_i \times _{U_0 }U_n } /\rho (W\otimes \mathcal{O}_{U_i \times _{U_0}U_n } ) .
\]

Let $\widetilde{x}\in X(K )$ be a point in $]x[$ lying above $\widetilde{s}$ in $]s[-Z$. 
Now we set $W:= \langle \Res _x \rangle $, viewed as a functional on $H^1 _{\dR}(X_{\widetilde{s}}/K)$, and $\mathcal{M}:=\pi _* \Omega _{X|S}(\log (Z))^\perp |_D$. For each $\widetilde{s}\in ]s[-Z (K)$, given a point $b\in ]x[_{X_{\widetilde{s}}}$ we have a Stoll logarithm
\[
\log _{x} ^{\Sto}:]x[_{X_{\widetilde{s}} } \to \Ker (\Res _x )^*
\]
corresponding to integration from the basepoint $b$.

Define $\mathcal{S}_{W,\mathcal{M}}$ to be the scheme over $(\pi ^2 )^* P_0 $ representing unipotent isomorphisms 
\[
\mathcal{E}^2 (X)|_U /(\mathcal{M}|_U +\rho _0 (W\otimes \mathcal{O}_U ))\simeq \gr ^\bullet \mathcal{E}^2 (X)|_U /(\mathcal{M}|_U +\rho _0 (W\otimes \mathcal{O}_U )),
\]
As previously, we may identify such an isomorphism with a section of $\Lie (J/S_K -Z_K )|_U /\rho (W\otimes \mathcal{O}_U )$.
Define 
\[
\mathcal{S}_{W,\mathcal{M},n}:=(\pr _1 ,\pr _n )^* \mathcal{S}_{W,\mathcal{M}}\times _{P_0 }\ldots \times _{P_0 }(\pr _{n-1} ,\pr _n )^* \mathcal{S}_{W,\mathcal{M}}.
\]
We have a commutative diagram
\begin{equation}\label{eqn:a_square}
\begin{tikzcd}
P_n ^{\FZ } \arrow[rd] \arrow[r] & \mathcal{U}\times P_0 \arrow[r] \arrow[d] & \Lie (J/S_K -Z_K )^n _{S_K -Z_K }\times P_0  \arrow[d] \\
& \mathcal{S}_{W,n} \arrow[r]           &    \mathcal{S}_{W,\mathcal{M},n}
\end{tikzcd} 
\end{equation}

Define $\mathbb{D}_d (S_{W,\mathcal{M},n} )\subset \mathcal{S}_{W,\mathcal{M},n}\times P_0 $ to be the subscheme of rank $\leq d $ $n$-points modulo the image of $W$. That is, locally for $U\subset S$, it is given by $(v_1 ,\ldots ,v_n ,\rho _0 )$ such that $v_1 ,\ldots ,v_n $ have rank $\leq d$ in $(\Lie (J/S_K -Z_K )|_U /\rho _0 (W\otimes \mathcal{O}_U )^n $. Then $\mathbb{D} _d \mathcal{S}_W $ is a codimension $(g-m-d)(n-d)$ subscheme of $\mathcal{S}_{W,\mathcal{M}}$. Define $\mathbb{D}_d (P^{\FZ}_n )\subset P^{\FZ}_n$ to be $\overline{p}^{-1}\mathbb{D}(\mathcal{S}_{W,\mathcal{M}})$. Hence the preimage of $\mathbb{D}_d \mathcal{S}_W$ in $P_n ^{\FZ}$ is a codimension $(g-e-d)(n-d)$ subscheme of $\mathbb{D}_d (P^{\FZ}_n )$. 

Recall that, by the description of the residue map in terms of parallel transport from Lemma \ref{lemma:relation_to_Stoll}, the analytic leaf $\alpha _0 :D\to P_0$ defines an isomorphism
\[
\xi : z^* \mathcal{V} \otimes \mathcal{O}_D \simeq \mathcal{V}
\]
such that $\xi (W\otimes \mathcal{O}_D )=\mathcal{W}$. Hence $\xi $ induces a rigid analytic map
\[
\Lie (J/S_K -Z_K )_D ^n /\mathcal{W}\to \mathcal{S}_{W,\mathcal{M}}
\]
extending the leaf $\alpha _0 $, and we may form its $(n-1)$-fold product to get 
\[
\iota :(\Lie (J/S_K -Z_K )_D  /\mathcal{W})^{n-1}\to \mathcal{S}_{W,\mathcal{M},n}
\]
By construction $\iota ^{-1}(\mathbb{D}_r ( \mathcal{S}_{W,\mathcal{M},n})= \mathbb{D}_r ((\Lie (J/S_K -Z_K )_D  /\mathcal{W})^{n-1})$. Commutativity of \eqref{eqn:a_square} implies commutativity of \eqref{eqn:FZ_commutes}, completing the proof of Proposition \ref{prop:endgame}.

\subsection{Proof of Corollary \ref{cor:NFs} and Corollary \ref{cor:Mg}}
We recall some results from \cite{CHM} regarding stable reduction in families and Abel--Jacobi morphisms for families of stable curves. Let $g$ be a positive integer greater than $1$.
\begin{theorem}[\cite{CHM}, Corollary 5.1]\label{thm:chm}
Given any morphism
\[
X\to S
\]
of integral varieties whose generic fibre is a smooth curve of genus $g>1$, there exists a generically finite map $S'\to S$, a (possibly geometrically disconnected) family of stable curves $X'\to S'$ smooth outside a normal crossing divisor $Z\subset S$, and a birational isomorphism of $S'$ schemes
\[
X' \simeq X\times _S S' .
\]
\end{theorem}
\begin{lemma}\label{lemma:fin_extn}
There is a finite extension $L_w$ of $K_v$, and a Zariski open set $U\subset S$ such that 
\begin{itemize}
\item every $K_v $ point of $X\times _S U$ lifts to an $L_w$ point of $X'$, and every rank $r$ $K_v $-point of $X^n_S \times _S U$ lifts to a rank $r$ $L_w$-point of $(X')^n _{S'}$.
\item every irreducible component of $S' _{\overline{K}}$ is defined over $L_w$.
\end{itemize}
\end{lemma}
\begin{proof}
For any finite cover, such an extension exists, since the Galois group of $K_v $ is topologically finitely generated. Indeed, if $ X'\to X$ has degree $d$, then we can take an extension $L_w |K_v $ containing all finite extensions of $K_v$ of degree $\leq d$. Then all $K_v $ points of $X$ lift to $L_w $-points of $X'$. Since the cover is obtained from extending the base, the property of being a rank $r$ $n$-tuple is preserved.
\end{proof}

\begin{proposition}
Suppose Theorem \ref{thm:main1} holds for all families $\pi :X\to S$ and all $p$-adic fields $K_v$, where $\pi $ is a stable curve. Then Theorem \ref{thm:main1} holds for all families $\pi $.
\end{proposition}
\begin{proof}
Given a general family $\pi :X\to S$ over a $p$-adic field $K_v$, by Lemma \ref{lemma:fin_extn}, we can find a finite extension $L_w |K_v$, an alteration $\pi' :X'\to S'$ of $\pi _{L_w }$, such that on a Zariski open of $S$, rank $\leq r$ points of $X^n _S (K_v )$ map to rank $\leq r$ points of $(X')^n _{S'}(L_w )$.
\end{proof}

To complete the proof of Corollary \ref{cor:NFs}, let $S'\to S$ and $\pi ' :X'\to S'$ be as in Theorem \ref{thm:chm}. Let $v$ be a prime of $K$ such that $X' ,S'$ and $Z'$ have good reduction at $v$ and $\pi ' :X'\to S'$ extends to a flat, proper morphism between integral models. Let $L_w |K_v$ be an extension as in Lemma \ref{lemma:fin_extn}. Then we may apply case (2) of Theorem \ref{thm:main1} to $\pi '$, over $L_w$. This completes the proof of Corollary \ref{cor:NFs}.

To prove Corollary \ref{cor:Mg}, we may use Lemma \ref{lemma:fin_extn} to a pass to a finite cover $\mathcal{M}_{g}[\ell ]\to \mathcal{M}_{g}$ where $\mathcal{M}_{g}[\ell ]$ is a scheme, and such that the cover extends to $\overline{\mathcal{M}}_{g}[\ell ]\to \overline{\mathcal{M}}_{g}$, with $\overline{\mathcal{M}}_g [\ell ]$ projective, and $\overline{\mathcal{M}}_{g}[\ell ]-\mathcal{M}_{g}[\ell ]$ a strict normal crossing divisor. As the notation suggests, we may take the finite cover to come from an auxiliary level $\ell$-structure for $\ell $ suitably large. Then the $n$-fold fibre product of the universal curve on $\mathcal{M}_g [\ell ]$, minus the diagonal subvarieties, dominates $\mathcal{M}_{g,n}$, and rank $\leq r$ points on $\mathcal{M}_{g,n}$ lift to rank $\leq r$ points. We deduce from Theorem \ref{thm:main1} that rank $\leq r$ points on the $n$-fold fibre product of the universal curve over $\mathcal{M}_g [\ell ]$ lie on a proper subvariety, and hence the projection of this proper subvariety to $\mathcal{M}_g$ defines a proper substack containing all rank $\leq r$ points on $\mathcal{M}_{g,n}$.
\section{Effectivity and examples}
In this section we give the proof of Corollary \ref{cor:its_effective}. 
To motivate the argument, suppose we are in any of the cases above where we prove $X^n _S (K)$ is not Zariski dense (for $K$ a number field), and suppose that we have a given rational point $y\in X^n _S (K)$. We would like to use Theorem \ref{thm:BCFN}, in its effective form, to deduce that we can effectively compute a proper subvariety $V$ of $X^n _S$ which contains the intersection of $X^n _S (K)_{\rk \leq r}$ with a suitably small $p$-adic neighbourhood of $y$. To do this, we need to make the principal bundle, its connection, the closed subvariety, and the $K$-point on the principal bundle, effective. We first show that the Galois group is effectively computable.

\begin{proposition}\label{prop:eff}
The Galois group $G$ of the Gauss--Manin connection associated to $X\to S$, and a descent of the frame bundle of the Gauss--Manin connection to a $G$-bundle, are effectively computable.
\end{proposition}
\begin{proof}
For any $s_1 , s_2 \in S(\mathbb{C})$, and any path $\gamma $ in $S(\mathbb{C})$ from $s_1 $ to $s_2 $, we may effectively compute the induced isomorphism 
\[
H_1 (X_{s_1 },\mathbb{Z})\to H_1 (X_{s_2 },\mathbb{Z}).
\]
Indeed, for each point $s$ of $S$, we can find a suitable small neighbourhood $U$ and a CW decomposition over $U$. First, we can effecitvely compute a CW decomposition of $S$  \cite{whitney}. This gives an effectively computable generating set for $\pi _1 (S,s)$. For any given set of loops in $\pi _1 (S,s)$, we can compute the induced elements of $\GL (H_1 (X_s, \mathbb{Z}))$. By Derksen, Jeandel and Koiran \cite{DJK}, we can effectively compute the Zariski closure of a finite set of matrices over a number field.

By the Riemann--Hilbert correspondence, $G$ is isomorphic, over $\mathbb{C}$, to the base change to $\mathbb{C}$ of the Zariski closure of $\pi _1 (S,s)$ in $\GL (H_1 (X_s ,\Q ))$. We can order the elements of $\GL _n (K(S))$, and for each such element $g$, we can compute the Lie algebra containing $g^{-1}dg -g^{-1}\Lambda g$. In this way we can get upper bounds for $G$, and the process terminates when we recover a group isomorphic over $\mathbb{C}$ to the group obtained by the Betti computation, giving a descent of the frame bundle to $G$.
\end{proof}
The remaining subtlety in applying Theorem \ref{thm:BCFN} is making the $K$-point on the principal bundle effective. The reason for this is that the $K$-point on $P$ usually has coordinates given by $p$-adic integrals, and hence is not defined over $K$. To apply the effectivity statement of Theorem \ref{thm:BCFN}, we have to put ourselves in a situation where these $p$-adic integrals actually have rational values. We do this by making them vanish.

\begin{proof}[Proof of Corollary \ref{cor:its_effective}]
Recall that the proofs of Theorem \ref{thm:main1} and Theorem \ref{thm:main2} work by interpreting the sets $X^n _S (\mathcal{O}_{K_v })_{\Gamma -\rk \leq r}$ as the projections of intersections of leaves of a foliation on a principal bundle $P$ on $X^n _S$ with certain proper subvarieties of $P$, and then applying Theorem \ref{thm:BCFN}. Hence it will be enough to prove that the the principal bundle, its connection, and the closed subvariety being interesected with are effectively computable and defined over $K$. Effective computability of the principal bundle and its connections is established by Proposition \ref{prop:eff}. In all three cases, we start with a point $(x,\ldots ,x)$. In case one, we take as our vector bundle connection $\mathcal{E}$ to be either $\mathcal{E}^n (X)$, viewed as a vector bundle with connection on $X^{n+1}_S $ via the Faltings--Zhang map. In cases two and three, we do the same thing with $\mathcal{E}^n _N (X)$, for $N$ as in Lemma \ref{lemma:N_estimate1} and Lemma \ref{lemma:BKimplies} respectively. In all three cases, we then construct the relevant principal bundle $P^{n+1}$ as was done earlier in the paper, and at the point $x$, we have a lift $z\in P(K)$ of $(x,\ldots ,x)$ to $P^{n+1}(K)$, given by the identity isomorphism
\[
(x,\ldots ,x)^* \mathcal{E} \simeq (x,\ldots ,x)^* \mathcal{E}.
\]
Then the leaf of the foliation on $P$ through $z$ corresponds to the $n$-fold product of the unipotent Albanese maps (or $p$-adic logarithm maps) at the basepoint $x$, as required.
\end{proof}
\subsection{Gonality bounds}
We now discuss the problem of classifying the $\nabla $-special subvarieties arising in Theorem \ref{thm:main2}, considering first the projective case. Here, when we say a proper subvariety of $X^n$ is $\nabla $-special, we mean that it is the projection of a $\nabla $-special subvariety of the principal bundle with connection associated to the depth $m$ universal connection on $X^n$, for some $m>0$.

We say that a subvariety of $X^n$ is diagonal if it is an intersection of subvarieties of the form $X\times _{X\times X}X^n$, where the maps are $\Delta :X\hookrightarrow X\times X$ and $(\pr _i ,\pr _j ):X^n \to X^2 $ for some $i\neq j$. We say that a variety is generalised diagonal if it is an intersection of diagonal subvarieties and subvarieties of the form $\pi _i ^{-1}(x)$ for $x\in X(K)$. Clearly, generalised diagonal subvarieties will be $\nabla $-special.
\begin{lemma}\label{lemma:abvar}
A geometrically irreducible subvariety $Z$ of $X^n$ is $\nabla$-special if and only if there is a map $f:X^n \to A$ to an abelian variety such that $f(X^n )$ generates $A$ and $Z$ is contained in $f^{-1}(0)$, where $0\in A(K)$ denotes the identity. 
\end{lemma}
\begin{proof}
If $Z$ is contained in $f^{-1}(0)$ as in the statement of the lemma, then $H_1 (Z_{\mathbb{C}},\mathbb{Q})$ does not surject onto $H_1 (X^n _{\mathbb{C}},\mathbb{Q})$, and hence is special. Conversely, for any special $Z$, by removing the singular locus and resolving singularities, the image of $\pi _1 (Z_{\mathbb{C}})$ in $\pi _1 (X^n _{\mathbb{C}})$ factors through a map $\pi _1 (\widetilde{Z})\to \pi _1 (X^n _{\mathbb{C}})$, where $\widetilde{Z}/\mathbb{C}$ is smooth and projective. It follows that the unipotent fundamental group of $\widetilde{Z}$, thought of as a unipotent group in mixed Hodge structures, is generated in weight $-1$, and hence the unipotent fundamental group of $Z$ surjects onto the $n$-unipotent fundamental group of $X^n$ if and only if $H_1 (Z_{\mathbb{C}},\Q )$ surjects onto $H_1 (X^n _{\mathbb{C}},\Q )$. Finally, this morphism is non-surjective if and only if the Albanese of $Z$ does not surject onto that of $X^n$.
\end{proof}
If $n$ is greater than or equal to the gonality of $X$, then $X^n$ will have many special subvarieties which are not generalised diagonal. Indeed, given any degree $n$ function $f:X\to \mathbb{P}^1 $, we have a corresponding map $\mathbb{P}^1 \to \Sym ^n X$, and the pre-image of this curve in $X^n$ will map to a point under the natural map to $\Jac (X)$.

\begin{proposition}\label{prop:generalised_diagonal}
Let $X$ be a curve of genus $g>1$ and gonality $\gamma$ such that $\End (\Jac (X_{\overline{K}}))=\mathbb{Z}$. Then all positive-dimensional $\nabla$-special subvarieties of $X^n $ are contained in a generalised diagonal subvariety for $n<\gamma$.
\end{proposition}
\begin{proof}
Let $Z$ be a positive dimensional $\nabla $-special subvariety. By Lemma \ref{lemma:abvar}, there is a map $f:X^n \to A$ to an abelian variety such that $f(X^n )$ generates $A$ and $Z$ is contained in the pre-image of a point. By our assumptions on the endomorphism algebra of $X$, we may take $A$ to be $\Jac (X)$, and the map $f$ to be $\sum m_i \cdot a_i$, where $a_i$ denotes the map $X^n \to \Jac (X)$ obtained by projecting to the $i$th coordinate and taking the Abel--Jacobi embedding $X\to \Jac (X)$, and $m_i \in \mathbb{Z}$.

The composite map
\[
H^0 (\Jac (X),\Omega _{\Jac (X)|K})\stackrel{f^* }{\longrightarrow }H^0 (X^n ,\Omega _{X^n|K})\stackrel{\iota ^* }{\longrightarrow }H^0 (Z,\Omega _{Z|K} )
\]
is zero, where $\iota $ denotes the immersion of $Z$ into $X^n$. On the other hand, via the isomorphism $H^0 (X,\Omega )\simeq H^0 (\Jac (X),\Omega )$, $f^* $ is given by $\sum n_i \cdot \pr _i ^* $. Hence we deduce that $\sum m_i \cdot \pr _i ^* \omega $ vanishes on $Z$ for all $\omega \in H^0 (X,\Omega )$. Without loss of generality, we may assume $m_n \neq 0$. If $Z$ is not generalised diagonal, then we may choose $x=(x_1 ,\ldots ,x_n )\in Z(\overline{K})$ is not diagonal and $\pr _i |_Z$ is smooth for all $i$.Then the $x_i$ are pairwise distinct, and each of the composite maps
\[
x_i ^* \Omega _{X|K}\to  x ^*\Omega _{X^n |K}\to x^* \Omega _{Z|K}
\]
is injective. We claim that, for all $x=(x_1 ,\ldots ,x_n )\in Z(K)$, we have 
\[
H^0 (X,\Omega _{X|K} (-x_1 -\ldots -x_{n-1}))=H^0 (X,\Omega _{X|K} (-x_1 -\ldots -x_n )),
\]
Indeed, if $\omega $ lies in the left-hand side, then its pullback to $x^* \Omega _{Z|K}$ is equal to $x_n ^* \omega $, and hence $\omega $ must vanish at $x_n $. By Riemann--Roch, this implies $H^0 (X,\mathcal{O}(x_1 +\ldots +x_n ))\neq 0$, giving the desired gonality bound.
\end{proof}

\begin{proof}[Proof of Corollary \ref{cor:low_gonality}]
To prove Corollary \ref{cor:low_gonality}, it is enough to prove that there are only finitely many points in $X(K)^n _{\rk r}$ which are not contained in a diagonal subvariety. By Proposition \ref{prop:generalised_diagonal}, the Zariski closure of $X(K)^n _{\Gamma -\rk r}$ is contained in a finite union of generalised diagonal subvarieties. It only remains to show that the positive dimensional irreducible components of the Zariski closure are contained in diagonal (not merely generalised diagonal subvarieties).

We argue by induction on $n$. When $n=1$ there is nothing to prove. Suppose that $W$ is an irreducible component which is contained in the image of the inclusion $i_x :X^{n-1}\hookrightarrow X^n$ sending $(x_1 ,\ldots ,x_{n-1})$ to $(x_1 ,\ldots ,x_{n-1},x)$. Let $\Gamma '$ denote the group generated by $\Gamma $ and $\AJ (x)$. Then the pre-image of $W$ in $X^{n-1}$ will be contained in the Zariski closure of $X^{n-1}_{\Gamma ' -\rk r}$, and hence by our inductive hypothesis it must be diagonal.
\end{proof}
\subsection{$S$-unit examples}
We now discuss a more thorough classification of the Zariski closure of rank $r$ $n$-tuples of solutions to the $S$-unit equation. Given a field $F$, and $n$ elements $x_1 ,\ldots ,x_n \in F^\times -\{ 1\}$, we say that $(x_1 ,\ldots ,x_n )$ is a rank $\leq s$ $n$-tuple of solutions to the unit equation if $\langle x_1 ,\ldots ,x_n ,1-x_1 ,\ldots ,1-x_n \rangle $ is a rank $\leq s$ subgroup of $F^\times $. Let $V$ be an irreducible component of the Zariski closure of $\cup _{\rk S=s}X^n (S)$. We say that $V$ comes from rank $\leq s$ function field solutions if $\langle f_1 ,1-f_1 ,\ldots ,f_n ,1-f_n \rangle $ has rank $\leq s $ in $K(V)$, where 
\[
(f_1 ,\ldots ,f_n ):V\hookrightarrow X^n \subset (\mathbb{P}^1 )^n 
\]
is the embedding of $V$ into $X$.

\begin{corollary}
For any $s> 0$, all non-diagonal positive dimensional irreducible components of the Zariski closure of $\cup _{\rk S =s}X^n (S)$ come from rank $<n$ function field solutions to the $S$-unit equation whenever
\[
n>5\cdot (2s)^{\frac{\log (2s)+\log \log (2s)}{\log (2)}+3}\cdot \log (2s).
\]
\end{corollary}
\begin{proof}
By Theorem \ref{thm:BCFN}, all positive dimensional components of the Zariski closure are $\nabla $-special subvarieties for the universal connection, so it is enough to show that $\nabla $-special subvarieties of $X^n$ which are not diagonal or constant give function field solutions of the $S$-unit equation of low rank.

Arguing as in Lemma \ref{lemma:abvar} shows that a subvariety of $(\mathbb{P}^1 -\{ 0,1,\infty \})^n$ is $\nabla $-special for the universal connection if and only if its Albanese variety does not dominate that of $(\mathbb{P}^1 -\{0,1,\infty \})^n$, or equivalently if the image of $V$ does not generate the Albanese variety of $(\mathbb{P}^1 -\{ 0,1,\infty \})^n$. If 
\[
i=(f_1 ,\ldots ,f_n):V\hookrightarrow (\mathbb{P}^1 -\{0,1,\infty \})^n
\]
is a subvariety, then the dimension of the sub-torus of $\Alb (\mathbb{P}^1 -\{ 0,1,\infty \})^n$ generated by the image of $V$ is equal to the rank of the subgroup $\langle f_1 ,\ldots ,f_n ,1-f_1 ,\ldots ,1-f_n \rangle $ of $K(V)^\times $.
\end{proof}

Of course, a non-generalised diagonal $\nabla $-special subvariety $V$ of rank $\leq r$ only gives infinitely many rank $\leq r$ $n$-tuples of solutions to the $S$-unit equation if it has infinitely many rational points. 

We note that there are well-known examples where this does happen. For example, the set 
\[
\cup _{\rk S=5}X(S)^5
\]
is not Zariski dense, but its Zariski closure contains the two dimensionsal subvariety 
\[
\{(x,y,\frac{1-x}{1-xy},1-xy,\frac{1-y}{1-xy}):(x,y)\in (\mathbb{P}^1 -\{0,1,\infty \})^2 -\{xy=1 \} \},
\]
\bibliography{references}
\bibliographystyle{alpha}

\end{document}